\documentclass[10pt,reqno,a4paper]{amsart}

\usepackage{amssymb,amsrefs,amscd,amsmath, fullpage}
\usepackage{graphicx,nicefrac}
\usepackage[shortlabels]{enumitem}
\usepackage{mathrsfs}
\usepackage{hyperref}
\usepackage{ourbib}
\usepackage[all]{xy}
\usepackage{callouts}
\usepackage{tikz-cd}
\usetikzlibrary{fit}

\setlist{leftmargin=8mm}

\hypersetup{
    colorlinks,
    linkcolor={red!50!black},
    citecolor={blue!50!black},
    urlcolor={blue!80!black}
}

\newcommand{\tr}{\mathrm{tr}}

\newcommand{\dM}{{\partial M}}

\newcommand{\On}{\mathrm{O}}

\newcommand{\scal}{\mathrm{scal}}
\newcommand{\ric}{\mathrm{ric}}

\newcommand{\PIC}{\mathrm{PIC}}

\newcommand{\dist}{\mathrm{dist}}

\newcommand{\myicon}{$\,\,\,\triangleright$}
\newcommand{\KN}{%
  \mathbin{\mathpalette\KNaux{}}
}

\newcommand{\KNaux}[2]{%
  \bigcirc\mkern-15mu\wedge
}

\DeclareMathOperator{\id}{\mathrm{id}}
\DeclareMathOperator{\Diff}{\mathrm{Diff}} 

\DeclareMathOperator{\supp}{\mathrm{supp}}

\newtheorem{theorem}{Theorem}[section]
\newtheorem*{theorem*}{Theorem}
\newtheorem*{theoremA*}{Theorem A}
\newtheorem*{theoremB*}{Theorem B}
\newtheorem*{theoremC*}{Theorem C}
\newtheorem*{perelman1*}{Perelman's gluing theorem}
\newtheorem*{GL*}{Gromov-Lawson's doubling theorem}
\newtheorem{lemma}[theorem]{Lemma}
\newtheorem{proposition}[theorem]{Proposition}
\newtheorem{corollary}[theorem]{Corollary}

\newtheorem{fact}[theorem]{Fact}

\theoremstyle{definition}
\newtheorem{remark}[theorem]{Remark}
\newtheorem{definition}[theorem]{Definition} 
\newtheorem{example}[theorem]{Example}

\begin{document}

\title{Boundary flexibility for curvature conditions} 
\author{Helge Frerichs}
\address{Universit\"at Augsburg, Institut f\"ur Mathematik, 86135 Augsburg, Germany}
\email{\href{mailto:helge.frerichs@math.uni-augsburg.de}{helge.frerichs@math.uni-augsburg.de}}

\begin{abstract} 
We prove a general boundary gluing principle for families of Riemannian metrics with pointwise curvature restrictions on manifolds with possibly non-compact boundary.
This includes, as special cases, the classical gluing theorems by Perelman for metrics with positive Ricci curvature and convex singularities and by Gromov-Lawson and Bär-Hanke for metrics with positive scalar curvature and mean convex singularities.
Using the language of algebraic curvature cones, our theorem implies and unifies many more previous gluing results.

We also prove a boundary deformation principle that enables us to compare spaces of metrics with interior curvature restrictions and different boundary conditions.

Our construction is based on the local flexibility lemma for open partial differential relations, as well as explicit $1$-jet deformations of metrics along the boundary.
\end{abstract}



\thanks{}

\date{}

\maketitle
\tableofcontents

\section{Introduction}
In order to construct Riemannian manifolds with specific curvature constraints, it is important to identify conditions under which these constraints are preserved under boundary gluing.
Perelman~\cite{Perelman} discovered a geometric boundary condition that allows two glued metrics of positive Ricci curvature to be smoothed near the gluing region while maintaining positivity of the Ricci curvature.
Specifically, the metrics are supposed to form convex singularities along their common boundary.
Prior to Perelman's theorem, Gromov-Lawson~\cite{GL} proved a similar gluing result for metrics of positive scalar curvature with mean-convex singularities in terms of doubling.

The main objective of the present paper is to extract an underlying gluing principle that provides a full dictionary for pointwise interior curvature conditions and associated boundary conditions.
Our results are formulated in the language of algebraic curvature cones, exhibiting certain parallels with Hoelzel's research on surgery stable curvature conditions in~\cite{Hoelzel}.
In particular, the paper at hand generalises and unifies Perelman's gluing theorem, Gromov-Lawson's doubling theorem and many other gluing results that have been established in recent years.

The overall setup and framework for the underlying gluing principle to be developed was introduced by Bär-Hanke in~\cite{BH2023}.
This approach allows us to handle also non-compact boundaries and to compare \textit{spaces of metrics} with interior curvature and boundary conditions, expanding pure existence results.
Bär-Hanke opened up this path through a general deformation principle for metrics with lower scalar curvature bounds in the interior of a manifold and lower mean curvature bounds along the boundary, see~\cite[Thm.~3.7]{BH2023}.
Note that the lower scalar curvature bounds to be considered may also be negative and non-constant.
As the main achievement, Bär-Hanke's deformation principle comes in a relative and family version.
Although the principle was originally formulated only for compact boundaries, the author previously included non-compact boundaries in this framework, see~\cite{Frerichs2025}.
Remarkably, the approach is not limited to lower scalar curvature and mean curvature bounds, but fits into the environment of algebraic curvature cones.
This is the theme of the paper at hand.

To put our results in context, we first recall the theorems by Perelman and Gromov-Lawson.
For a Riemannian metric $g$ on a smooth manifold with non-empty boundary, we denote by $g_0$ the metric induced on the boundary and by $\nu$ the inward-pointing unit normal field.
The scalar-valued second fundamental form $\mathrm{II}_g$ of the boundary is defined by $\mathrm{II}_g(v,w)=-g(\nabla_v\nu,w)$ for all vector $v,w$ tangent to the boundary, where $\nabla$ denotes the ambient Levi-Civita connection.
The mean curvature of the boundary is defined as the trace $H_g=\tr_g(\mathrm{II}_g)$.

\begin{perelman1*}\textup{(Perelman, Section~4 in~\cite{Perelman})}

Let $n\geq 2$.
Let $M_1,M_2$ be compact smooth $n$-dimensional manifolds with non-empty boundary $\partial M_1=\partial M_2=:\partial M$.
Let $g_1\sqcup g_2$ be a Riemannian metric on $M_1\sqcup M_2$ with $(g_1)_0=(g_2)_0$ that has positive Ricci curvature.
Let $\mathscr{U}\subset M_1\sqcup M_2$ be a neighbourhood of $\partial M_1\sqcup\partial M_2$.

Suppose that $\mathrm{II}_{g_1}+\mathrm{II}_{g_2}$ is positive definite at every point in $\partial M$.
Then there exists a smooth Riemannian metric $f$ on $M_1\cup_{\partial M}M_2$ such that
\begin{enumerate}
\item[(a)]{$f$ has positive Ricci curvature;}
\item[(b)]{$f=g_1\sqcup g_2$ on $M_1\sqcup M_2\setminus\mathscr{U}$.}
\end{enumerate}
\end{perelman1*}

\begin{GL*}\textup{(Gromov-Lawson, Theorem~5.7 in~\cite{GL})}

Let $n\geq 2$.
Let $M$ be a compact smooth $n$-dimensional manifold with non-empty boundary $\partial M$.
Let $g$ be a Riemannian metric on $M$ that has positive scalar curvature.

Suppose that $H_g>0$ at every point in $\partial M$.
Then the double $M\cup_{\partial M}M$ carries a metric of positive scalar curvature.
\end{GL*}

Note that both Perelman's gluing theorem and Gromov-Lawson's doubling theorem do not generally hold true without any boundary conditions.
Also the existence of metrics with interior curvature constraints and boundary conditions is closely related to the topologies of the manifolds involved.
For example, if the boundary conditions were ignored, we would be able to produce metrics of positive scalar curvature or even positive Ricci curvature on the torus $T^n$, by applying the doubling or gluing theorems to $T^{n-1}\times[0,1]$, which carries a positive Ricci curvature metric.
Since the fundamental group of $T^n$ is not finite, the existence of positive Ricci curvature on $T^n$ would contradict the theorem of Bonnet-Myers.
Stronger still, the existence of positive scalar curvature is ruled out by the topological property that $T^n$ is enlargeable, see~\cite[Thm.~A]{GL}.

To discover an underlying gluing principle, two schools of thought must be joined.
\begin{enumerate}
\item[$\triangleright$]{Gluing theorems of \textit{Perelman type}:
On the one hand, there exist interior curvature conditions more rigid than positive Ricci curvature such that Perelman's gluing theorem still holds with these curvature conditions and the fixed boundary condition $\mathrm{II}_{g_1}+\mathrm{II}_{g_2}\geq 0$.
In other words, the entire second fundamental form needs to be non-negative in the sense of bilinear forms.

For example, Reiser-Wraith~\cite[Thm.~A]{RW} proved the gluing theorem for positive sectional curvature.
Schlichting~\cite[Thm.~1.1.2]{Schlichting2014} and Chow~\cite[Thm.~1]{Chow} even proved the result for positive curvature operators.}
\item[$\triangleright$]{Gluing theorems of \textit{Gromov-Lawson type}:
On the other hand, there are interior curvature conditions weaker than positive Ricci curvature for which the boundary conditions in a gluing theorem can be relaxed.
Instead of full convexity, only certain mean curvatures need to be non-negative.

In addition to Gromov-Lawson's theorem for positive scalar curvature, we would like to mention the work of Reiser-Wraith~\cite{RW} and Chow-Johne-Wan~\cite{CJW}, who studied curvature conditions with similar flexibility.}
\end{enumerate}

For a systematic point of view, consider the vector space $\mathscr{C}_B(\mathbb{R}^n)$ of algebraic curvature tensors satisfying the first Bianchi identity, see Definition~\ref{Bessedef} and~\cite[Chap.~1.G]{Besse} for a more exhaustive source.
Let $\mathsf{C}\subset\mathscr{C}_B(\mathbb{R}^n)$ be an open, convex and $\On(n)$-invariant cone.
We say that a Riemannian manifold with boundary $(N^n,g)$ \textit{satisfies $\mathsf{C}$} if the pulled-back tensors of the Riemannian curvature tensor $R_g$ meet the condition $\iota^\ast R_g(p)\in\mathsf{C}$ for all $p\in M$ and all linear isometries $\iota\colon\mathbb{R}^n\to (T_pN,g(p))$.

Furthermore, let $\mathscr{R}_{\mathsf{C}}(N)$ denote the space of smooth Riemannian metrics on $N$ that satisfy $\mathsf{C}$.
This space is equipped with the weak $C^\infty$-topology.

First, let us look at the \textit{Perelman type} direction:
For two smooth manifolds $M_1,M_2$ with non-empty boundary $\partial M_1=\partial M_2=:\partial M$, we define
\begin{equation}\nonumber
\mathscr{R}_{\mathsf{C}}^{\mathrm{II}_1+\mathrm{II}_2\geq 0}(M_1\sqcup M_2):=\{g_1\sqcup g_2\in\mathscr{R}_{\mathsf{C}}(M_1\sqcup M_2):(g_1)_0=(g_2)_0,\mathrm{II}_{g_1}+\mathrm{II}_{g_2}\geq 0\}.
\end{equation}
The last condition means that $\mathrm{II}_{g_1}+\mathrm{II}_{g_2}$ is positive semi-definite at every point in $\partial M$.

Recall that a continuous map $f\colon X\to Y$ between topological spaces is called a \textit{weak homotopy equivalence} if it induces a bijection $\pi_0(X)\cong\pi_0(Y)$ and isomorphisms $\pi_m(X,x)\cong\pi_m(Y,f(x))$ for all $x\in X$ and $m\geq 1$.

Our first main theorem concerns algebraic curvature cones that contain the cone $\mathsf{C}_{\mathcal{R}>0}$ of positive curvature operators.
Note that the manifolds involved, and also their boundaries, may be non-compact.

\begin{theoremA*}\label{ThmA} \textup{(General gluing theorem of \textit{Perelman type}, see Theorem~\ref{whe})}

Let $n\geq 2$.
Let $M_1,M_2$ be smooth $n$-dimensional manifolds with non-empty boundary $\partial M_1=\partial M_2=:\partial M$.
Let $\mathsf{C}\subset\mathscr{C}_B(\mathbb{R}^n)$ be an open, convex and $\On(n)$-invariant cone with $\mathsf{C}_{\mathcal{R}>0}\subset\mathsf{C}$.
Then the inclusion
\begin{equation}\nonumber
\mathscr{R}_{\mathsf{C}}(M_1\cup_{\partial M}M_2)\hookrightarrow\mathscr{R}^{\mathrm{II}_1+\mathrm{II}_2\geq 0}_{\mathsf{C}}(M_1\sqcup M_2)
\end{equation}
is a weak homotopy equivalence.
\end{theoremA*}

For the \textit{Gromov-Lawson} equivalent, we use a little bit more language.
Let $\mathfrak{b}$ denote the standard inner product on $\mathbb{R}^n$, i.e. $\mathfrak{b}(e_i,e_j)=\delta_{ij}$, and let $\mathfrak{b}^\perp$ be the bilinear form on $\mathbb{R}^n$ with $\mathfrak{b}^\perp(e_i,e_j)=\delta_{ij}\delta_{1i}$.
We denote by $\KN$ the Kulkarni-Nomizu product of symmetric bilinear forms.

Unlike in Theorem~A, we will not consider general cones in the Gromov-Lawson setup, but only cones containing the algebraic curvature tensor $\mathfrak{b}\KN\mathfrak{b}^\perp$.
This condition encodes the additional flexibility for the Gromov-Lawson type.
To name the associated boundary conditions, we set

\begin{align}\nonumber
\mathscr{R}_{\mathsf{C}}^{\mathrm{GL}}(M_1\sqcup M_2):=\{g_1\sqcup g_2\in\mathscr{R}_{\mathsf{C}}(M_1&\sqcup M_2):(g_1)_0=(g_2)_0,\\\nonumber
&\iota_{g_0}^\ast(\mathrm{II}_{g_1}(p)+\mathrm{II}_{g_2}(p))\KN\mathfrak{b}^\perp\in\overline{\mathsf{C}}\;\;\text{for all $p\in\partial M$}\},
\end{align}
where $\iota_{g_0}\colon\mathbb{R}^{n-1}\to(T_p(\partial M),(g_1)_0(p)=(g_2)_0(p))$ is a linear isometry and $\overline{\mathsf{C}}$ denotes the closure of $\mathsf{C}$ in $\mathscr{C}_B(\mathbb{R}^n)$.
The bilinear form $\iota_{g_0}^\ast(\mathrm{II}_{g_1}(p)+\mathrm{II}_{g_2}(p))$ on $\mathbb{R}^{n-1}$ is regarded as a bilinear form on $\mathbb{R}^n$ by means of the inclusion $\{0\}\times\mathbb{R}^{n-1}\hookrightarrow\mathbb{R}^n$ and extension with zero in the first coordinate direction.

These boundary conditions are attributed to Gromov-Lawson, indicated by the letters 'GL'.

\begin{theoremB*} \textup{(General gluing theorem of \textit{Gromov-Lawson type}, see Theorem~\ref{whe2})}

Let $n\geq 2$.
Let $M_1,M_2$ be smooth $n$-dimensional manifolds with non-empty boundary $\partial M_1=\partial M_2=:\partial M$.
Let $\mathsf{C}\subset\mathscr{C}_B(\mathbb{R}^n)$ be an open, convex and $\On(n)$-invariant cone with $\mathfrak{b}\KN\mathfrak{b}^\perp\in\mathsf{C}$.
Then the inclusion
\begin{equation}\nonumber
\mathscr{R}_{\mathsf{C}}(M_1\cup_{\partial M}M_2)\hookrightarrow\mathscr{R}_{\mathsf{C}}^{\mathrm{GL}}(M_1\sqcup M_2)
\end{equation}
is a weak homotopy equivalence.
\end{theoremB*}

Perelman's original gluing theorem clearly follows from surjectivity in the $\pi_0$-statement of Theorem~A for the case where $\mathsf{C}=\mathsf{C}_{\ric>0}$ is the cone of positive Ricci curvature.
It is easy to verify that Perelman's gluing theorem is actually implied by both Theorem~A, which is more generic, and Theorem~B, which concerns more specific curvature conditions.
From this perspective, positive Ricci curvature is a hinge for the two settings.
Furthermore, from a different perspective, the curvature conditions considered in Theorem~B form a singularity among the general curvature conditions in Theorem~A.
This statement will be clarified in Section 2, see Remark~\ref{remark}.
Both perspectives serve to explain how the two types of gluing theorems result in unification.

To the best of the authors knowledge, the full homotopical results of Theorems~A and~B were previously only known for metrics with lower scalar curvature bounds and mean convex singularities, see~\cite[Thm.~4.11]{BH2023} and~\cite[Rmk.~3.18]{Frerichs2025}.

Theorems~A and~B cover a wide range of gluing results by various authors, some of which have already been mentioned above.
In Table~\ref{overview} we list several pairs of interior curvature conditions and boundary conditions for which a gluing theorem was established in previous work.
The following abbreviations are used:
$\mathcal{R}_2>0$ means '$2$-positive curvature operator', $\PIC$ means 'positive isotropic curvature', $\ric_k>0$ means 'positive $k^{th}$ intermediate Ricci curvature', $\scal_k>0$ means '$k$-positive Ricci curvature' and $\mathcal{C}_m>0$ means 'positive $m$-intermediate curvature'.
The other terms should therefore be clear.
More detailed explanations for the curvature conditions will be given in Section 2.

At this point, we shall consider the boundary conditions:
A symmetric bilinear form $A\colon V\times V\to\mathbb{R}$ on a finite dimensional inner product space $V$ is called $k$-positive if $\sum_{i=1}^kA(u_i,u_i)>0$ for every orthonormal $k$-frame $(u_1,\dots,u_k)\subset V$.
Equivalently, the sum of the $k$ smallest eigenvalues of $A$ is $>0$, see e.g.~\cite[Lem.~1.1]{Sha}.

In the same way, one can define what it means for $A$ to be $k$-non-negative.

For the second fundamental forms $\mathrm{II}_{g_1}$ and $\mathrm{II}_{g_2}$, we say that $\mathrm{II}_{g_1}+\mathrm{II}_{g_2}$ is $k$-convex if $\mathrm{II}_{g_1}(p)+\mathrm{II}_{g_2}(p)$ is $k$-non-negative for all $p\in\partial M$.
The sum $\mathrm{II}_{g_1}+\mathrm{II}_{g_2}$ is called mean-convex if it is $(n-1)$-convex.

\begin{table}[h]
\begin{tabular}{|c|c|c|}
\hline
$\mathsf{C}$&\begin{tabular}{cc}Boundary condition\\($\mathrm{II}_{g_1}+\mathrm{II}_{g_2}$ is)\end{tabular}&Contributors\\[0.25cm]
\hline
\hline
&&\\[-0.3cm]
$\mathcal{R}>0$&convex&Schlichting~\cite[Thm.~1.1.2]{Schlichting2014}, Chow~\cite[Thm.~1]{Chow}\\[0.075cm]
\hline&&\\[-0.3cm]
$\mathcal{R}_2>0$&convex&Schlichting~\cite[Thm.~1.7.3]{Schlichting2014}\\[0.075cm]
\hline&&\\[-0.3cm]
\begin{tabular}{cc}$\PIC2$\\$\PIC1$\vspace{0.075cm}\end{tabular}&convex&Schlichting~\cite[Thm.~1.7.4]{Schlichting2014}, Chow~\cite[Thm.~1]{Chow}\\[0.075cm]
\hline&&\\[-0.3cm]
$\PIC$&$2$-convex&Schlichting~\cite[Thm.~1.7.4]{Schlichting2014}, Chow~\cite[Thm.~1]{Chow}\\[0.075cm]
\hline&&\\[-0.3cm]
$\sec>0$&convex&Kosovskii~\cite[Thm.~1.1]{Kosovskii}, Reiser-Wraith~\cite[Thm.~A]{RW}\\[0.075cm]
\hline&&\\[-0.3cm]
$\ric_k>0$&convex&Schlichting~\cite[Thm.~7.5]{Schlichting2012}, Reiser-Wraith~\cite[Thm.~A]{RW}\\[0.075cm]
\hline&&\\[-0.3cm]
$\ric>0$&convex&\begin{tabular}{cc}Perelman~\cite[Sec.~4]{Perelman}, Schlichting~\cite[Thm.~1.7.1]{Schlichting2014},\\Botvinnik-Walsh-Wraith~\cite[Thm.~2]{BWW},\\
Reiser-Wraith~\cite[Thm.~A]{RW}\vspace{0.075cm}
\end{tabular}\\
\hline&&\\[-0.3cm]
$\scal_k>0$&\begin{tabular}{cc}\footnotesize{$k$-convex, $1\leq k\leq n-2$}\\\footnotesize{$(k-1)$-convex, $k=n-1,n$}\vspace{0.075cm}\end{tabular}&Reiser-Wraith~\cite[Thm.~A]{RW}\\
\hline&&\\[-0.3cm]
$\scal>0$&mean-convex&\begin{tabular}{cc}Gromov-Lawson~\cite[Thm.~5.7]{GL}, Almeida~\cite[Thm.~1.1]{Almeida},\\Brendle-Marques-Neves~\cite[Thm.~5]{BMN},\\
Bär-Hanke~\cite[Thm.~4.11]{BH2023}, F.~\cite[Thm.~3.17]{Frerichs2025}\vspace{0.075cm}
\end{tabular}\\
\hline&&\\[-0.3cm]
$\mathcal{C}_m>0$&$m$-convex&Chow-Johne-Wan~\cite[Thm.~1.2]{CJW}\\[0.075cm]
\hline
\end{tabular}
\vspace{5pt}
\caption{Pairs of interior curvature conditions and boundary conditions with gluing theorems.}
\label{overview}
\end{table}

Theorem~A implies all gluing results where the boundary condition is 'convex'.
The gluing theorem for $\PIC$ follows from Theorem~\ref{whe}, of which Theorem~A is one part.
Theorem~B covers all remaining results, which will be shown in the main text.

Let us remark that, for $\scal_{n-1}>0$, the boundary condition in Theorem~B is slightly weaker than that in Reiser-Wraith's gluing theorem~\cite[Thm.~A]{RW}:
If $\mathrm{II}_{g_1}+\mathrm{II}_{g_2}$ is $(n-2)$-convex, then the boundary condition in Theorem~B is satisfied and the gluing theorem holds.
However, it can also be met if $\mathrm{II}_{g_1}(p)+\mathrm{II}_{g_2}(p)$ is only $(n-1)$-non-negative for some or all $p\in\partial M$.
Our boundary condition is a precise or continuous translation for this particular circumstance into the language of algebraic curvature cones.

Furthermore, our gluing theorems are applicable not only to individual curvature conditions, as shown in Table~\ref{overview}, but also to a broader class that has been defined in previous work.
To recapitulate the definition of this class, we identify the exterior product $(\Lambda^2\mathbb{R}^n,\mathfrak{b})$ together with its standard euclidean inner product with the Lie algebra $(\mathfrak{so}(n),\langle\cdot,\cdot\rangle)$ of real skew-symmetric matrices, where the inner product is given as $\langle A,B\rangle=-\tfrac{1}{2}\tr(AB)$.
We use the same notation for the complexification $(\mathfrak{so}(n,\mathbb{C}),\langle\cdot,\cdot\rangle)$.
Every curvature tensor $R\in\mathscr{C}_B(\mathbb{R}^n)$ can be regarded as a tensor $R_{\mathbb{C}}\colon\mathfrak{so}(n,\mathbb{C})\to\mathfrak{so}(n,\mathbb{C})$ through complex linear extension.

For an $\mathrm{Ad}_{\mathrm{SO}(n,\mathbb{C})}$-invariant subset $S\subset\mathfrak{so}(n,\mathbb{C})$, the definition
\begin{equation}\nonumber
\mathsf{C}(S):=\{R\in\mathscr{C}_B(\mathbb{R}^n):\langle R_{\mathbb{C}}(v),\bar{v}\rangle>0\;\;\text{for all $v\in S$}\}
\end{equation}
yields an open, convex and $\On(n)$-invariant cone.
This class of curvature cones has been investigated by Wilking in his Lie algebraic approach to Ricci flow invariant curvature conditions in~\cite{Wilking}.

Since $\mathsf{C}(\mathfrak{so}(n,\mathbb{C}))$ is the cone of positive curvature operators, Theorem~\ref{whe} (which includes Theorem~A) applies to all $\mathrm{Ad}_{\mathrm{SO}(n,\mathbb{C})}$-invariant subsets $S\subset\mathfrak{so}(n,\mathbb{C})$.
From this point of view, our results offer a connection to abstract Ricci flow theory.
Note that for specific choices of $S$, the cones $\mathsf{C}(S)$ coincide with some of the curvature conditions in Table~\ref{overview}, see~\cite{Wilking}.
In addition to $\mathcal{R}>0$, these conditions are $\mathcal{R}_2>0,\PIC2,\PIC1$ and $\PIC$.

In both Theorem~A and~B, a crucial step in the gluing process involves deforming the second fundamental form of $g_1\sqcup g_2$.
For cones with $\mathfrak{b}\KN\mathfrak{b}^\perp\notin\mathsf{C}$, this process is more complex
as it requires detailed analysis and control of the curvature contributions coming from the Gauss equation for the boundary.
The special significance of the Gauss equation is also the reason why strong convexity assumptions are needed for the boundary gluing in these cases, see Proposition~\ref{1jet}.

The difficulty just described does not occur with cones that do contain $\mathfrak{b}\KN\mathfrak{b}^\perp$, which has more favorable consequences:
Rather than just gluing two manifolds together, it is also possible to deform families of metrics on a \textit{single} manifold and compare different boundary conditions.
So far, this has only been done for metrics with lower scalar curvature and mean curvature bounds, see~\cite[Chap.~4.1]{BH2023}.
Theorem~\ref{theorem2} in the present paper generalises these results to all cones containing $\mathfrak{b}\KN\mathfrak{b}^\perp$.

To clarify the statement of Theorem~\ref{theorem2}, we will now present a simplified version of it.
The idea is as follows:
If $\mathsf{C}$ is a cone with $\mathfrak{b}\KN\mathfrak{b}^\perp\in\mathsf{C}$ and if $(M,g)$ is a Riemannian manifold with non-empty boundary $\partial M$ that satisfies $\mathsf{C}$, then one can deform the second fundamental form of $\partial M$ into an arbitrary $(0,2)$-tensor field $k$ along $\partial M$, provided that $\mathrm{II}-k$ satisfies a boundary condition of Gromov-Lawson type.

\begin{theoremC*} \textup{(Deformation principle for GL-boundary conditions, simplified version of Theorem~\ref{theorem2})}

Let $n\geq 2$.
Let $\mathsf{C}\subset\mathscr{C}_B(\mathbb{R}^n)$ be an open, convex and $\On(n)$-invariant cone with $\mathfrak{b}\KN\mathfrak{b}^\perp\in\mathsf{C}$.

Let $M$ be a smooth $n$-dimensional manifold with non-empty boundary $\partial M$ and let $g$ be a Riemannian metric on $M$ that satisfies $\mathsf{C}$.
Let $k\in C^{\infty}(\partial M; T^\ast\partial M\otimes T^\ast\partial M)$ be a symmetric $(0,2)$-tensor field along $\partial M$ satisfying
\begin{align}\nonumber
\iota_{g_0}^\ast(\mathrm{II}_{g}(p)-k(p))\KN\mathfrak{b}^\perp\in\overline{\mathsf{C}}
\end{align}
for all $p\in\partial M$ and all linear isometries $\iota_{g_0}\colon\mathbb{R}^{n-1}\to(T_p(\partial M),g_0(p))$.

For each neighbourhood $\mathscr{U}\subset M$ of $\partial M$ there exists a continuous path $(f_s)_{s\in[0,1]}$ of Riemannian metrics on $M$ such that
\begin{enumerate}
\item[(a)]{$f_s$ satisfies $\mathsf{C}$ for all $s\in[0,1]$;}
\item[(b)]{$f_0=g$;}
\item[(c)]{$\mathrm{II}_{f_1}=k$;}
\item[(d)]{$f_s=g$ on $M\setminus\mathscr{U}$ for all $s\in[0,1]$.}
\end{enumerate}
\end{theoremC*}

The paper is organised as follows:
In Section~2 we introduce the concept of flexible families of curvature conditions as a generalisation of algebraic curvature cones.
The aim of this is to encode lower curvature bounds other than zero.
We also prove a quantitative version of Gram-Schmidt's algorithm which will be important in the proof of Proposition~\ref{1jet}.
The latter is a technical key point in this paper.

Section~3 covers the main gluing principle in the setting of Theorem~A.
As a preparatory step, we uniformise geodesic collar neighbourhoods of Riemannian metrics.
This is necessary in order to study spaces of metrics on glued manifolds.

Most of the section is devoted to gluing metrics.
Specifically, we present a two-stage gluing scheme that has been initiated by Bär-Hanke in ~\cite{BH2023}.
It employs the local flexibility lemma for open partial differential relations, see~\cite{BH2022}, and then deforms the $1$-jet of metrics along the boundary.

In Section~4 we show that Gromov-Lawson flexible families of curvature conditions, which are introduced in Section~2, allow for especially general boundary deformations.
A special case of this is Theorem~C.
We will deduce Theorem~B from Theorem~C.\\
\newline
\textit{Acknowledgements.} This work is part of my doctoral dissertation project.
I would like to thank my advisor Bernhard Hanke for his guidance and support and for the inspiring discussions we had during the Lonavala Geometry Festival.

My doctoral studies are funded by a stipend from the Studienstiftung des deutschen Volkes (German Academic Scholarship Foundation).
This work was also supported by the TopMath Program from the Elite Network of Bavaria.

\section{Preliminaries}
\subsection{General definitions and notation}
The first definition addresses one of the central concepts of this paper.
\begin{definition}\label{Bessedef}
Let $V$ be a finite-dimensional vector space.
An algebraic curvature tensor on $V$ is a multilinear form $R\colon V\times V\times V\times V\to\mathbb{R}$ such that
\begin{equation}\nonumber
R(v_1,v_2,v_3,v_4)=-R(v_2,v_1,v_3,v_4)=R(v_3,v_4,v_1,v_2)
\end{equation}
and that $R$ satisfies the first Bianchi identity
\begin{equation}\nonumber
R(v_1,v_2,v_3,v_4)+R(v_2,v_3,v_1,v_4)+R(v_3,v_1,v_2,v_4)=0
\end{equation}
for all vectors $v_1,v_2,v_3,v_4\in V$.

Let $\mathscr{C}_B(V)$ denote the space of algebraic curvature tensors on $V$.
Furthermore, on a Riemannian manifold $(M,g)$ with or without boundary, we denote by $R_g$ the Riemannian curvature tensor.
Then $R_g(p)$ is an algebraic curvature tensor on $T_pM$ for all $p\in M$.
\end{definition}

We fix the following notation on $\mathbb{R}^n$:
Let $(e_1,\dots,e_n)$ denote the standard basis of $\mathbb{R}^n$ and $\mathfrak{b}$ the standard euclidean inner product.
The symbol $\perp$ indicates orthogonality with respect to $\mathfrak{b}$.
The euclidean norm on $\mathbb{R}^n$ and the induced norm on curvature tensors in $\mathscr{C}_B(\mathbb{R}^n)$ are both denoted by $\Vert\cdot\Vert$.
With this norm, $B_r(\mathscr{C}_B(\mathbb{R}^n))$ is the ball of radius $r>0$ in $\mathscr{C}_B(\mathbb{R}^n)$ around the zero curvature tensor and $B_r(R)$ is the corresponding ball around a tensor $R\in\mathscr{C}_B(\mathbb{R}^n)$.

We define $\mathfrak{b}^\perp:=\mathrm{pr}_{\perp}^\ast\mathfrak{b}$ for the orthogonal projection $\mathrm{pr}_{\perp}\colon\mathbb{R}^n\to\mathbb{R}e_1$, so that
\begin{equation}\nonumber
\mathfrak{b}^\perp(e_i,e_j)=\left\{\begin{array}{ll} 1, & \text{if $i=j=1$,}\\
         0, & \text{else.}\end{array}\right.
\end{equation}

A canonical way to construct out of two symmetric $(0,2)$-tensors on a finite-dimensional vector space $V$ an algebraic curvature tensor is the Kulkarni-Nomizu product:
Given $h,k\in V^\ast\otimes V^\ast$, their product is defined as
\begin{align}\nonumber
(h\KN k)(v_1,v_2,v_3,v_4)&:=h(v_1,v_4)k(v_2,v_3)+h(v_2,v_3)k(v_1,v_4)\\\nonumber
&\phantom{:=}-h(v_1,v_3)k(v_2,v_4)-h(v_2,v_4)k(v_1,v_3).
\end{align}
If $h$ and $k$ are symmetric, $h\KN k$ has the symmetries that make it an algebraic curvature tensor.
Note that some authors use a different sign convention in the definition of the Kulkarni-Nomizu product.

In this paper, one important instance of the Kulkarni-Nomizu product is the algebraic curvature tensor $\mathfrak{b}\KN\mathfrak{b}^\perp\in\mathscr{C}_B(\mathbb{R}^n)$ which is characterised by the property that
\begin{align}\nonumber
&(\mathfrak{b}\KN\mathfrak{b}^\perp)(e_i,e_1,e_1,e_i)=1\quad\text{for all $2\leq i\leq n$}\\\nonumber
&\text{and all other independent entries of $\mathfrak{b}\KN\mathfrak{b}^\perp$ are zero.}
\end{align}

\subsection{Flexible curvature conditions}
An open and $\On(n)$-invariant subset $\mathsf{C}\subset\mathscr{C}_B(\mathbb{R}^n)$ is called a \textit{curvature condition}.

An algebraic curvature tensor $R\in\mathscr{C}_B(V)$ on an $n$-dimensional inner product space $V$ is said to \textit{satisfy $\mathsf{C}$} if for every linear isometry $\iota\colon\mathbb{R}^n\to V$, the pullback curvature tensor $\iota^\ast R\in\mathscr{C}_B(\mathbb{R}^n)$ lies in $\mathsf{C}$.
By the $\On(n)$-invariance, it is sufficient to check this property for a single (and arbitrary) linear isometry $\iota\colon\mathbb{R}^n\to V$.
Here and throughout the following, the model space $\mathbb{R}^n$ will always be equipped with the standard inner product $\mathfrak{b}$.

An $n$-dimensional Riemannian manifold $(M,g)$ with or without boundary is said to \textit{satisfy $\mathsf{C}$} if $R_g(p)\in\mathscr{C}_B(T_pM)$ satisfies $\mathsf{C}$ for all $p\in M$.

In this paper, we do not consider individual curvature conditions but families of curvature conditions $\mathsf{C}=(\mathsf{C}(\sigma))_{\sigma\in\mathbb{R}}$ with the following properties:
\begin{enumerate}
\item[(i)]{$\overline{\mathsf{C}(\sigma_1)}\subset\mathsf{C}(\sigma_2)$ for all $\sigma_1>\sigma_2$;}
\item[(ii)]{$\overline{\mathsf{C}(\sigma_1)}+\mathsf{C}(\sigma_2)\subset\mathsf{C}(\sigma_1+\sigma_2)$ for all $\sigma_1,\sigma_2\in\mathbb{R}$;}
\item[(iii)]{$\lambda\mathsf{C}(\sigma)\subset\mathsf{C}(\lambda\sigma)$ for all $\sigma\in\mathbb{R}$ and $\lambda>0$.}
\end{enumerate}

Such a family is called \textit{continuous} if for every $\sigma\in\mathbb{R}$, there exists a continuous function $\varepsilon(\sigma)\colon\mathsf{C}(\sigma)\to(0,\infty)$ such that for every $R\in\mathsf{C}(\sigma)$, it holds $R\in\mathsf{C}(\sigma+\varepsilon(\sigma)(R))$.

\begin{remark}\label{Basicremark}
Property $\mathrm{(iii)}$ is equivalent to requesting $\lambda\mathsf{C}(\sigma)=\mathsf{C}(\lambda\sigma)$ for all $\sigma\in\mathbb{R}$ and $\lambda>0$.
By property $\mathrm{(ii)}$ and $\mathrm{(iii)}$, each curvature condition $\mathsf{C}(\sigma),\sigma\in\mathbb{R}$ is convex.
Furthermore, by property $\mathrm{(iii)}$, $\mathsf{C}(0)$ is a punctured algebraic cone.

If $\mathsf{C}(0)\neq\varnothing$, it is also clear that $\overline{\mathsf{C}(0)}$ is an actual algebraic cone with $0_{\mathscr{C}_B(\mathbb{R}^n)}\in\overline{\mathsf{C}(0)}$.
Hence, by property $\mathrm{(i)}$, for every $\varepsilon>0$ there exists $r>0$ such that $B_r(\mathscr{C}_B(\mathbb{R}^n))\subset\mathsf{C}(-\varepsilon)$.

We also mention that $\mathscr{C}_B(\mathbb{R}^n)=\cup_{\sigma\in\mathbb{R}}\mathsf{C}(\sigma)$ in this case:
Let $0\neq R\in\mathscr{C}_B(\mathbb{R}^n)$.
Let $r>0$ with $B_r(\mathscr{C}_B(\mathbb{R}^n))\subset\mathsf{C}(-1)$.
Then $R\in B_{2\Vert R\Vert}(\mathscr{C}_B(\mathbb{R}^n))=\tfrac{2\Vert R\Vert}{r}B_r(\mathscr{C}_B(\mathbb{R}^n))\subset\mathsf{C}(-2\Vert R\Vert/r)$.
\end{remark}

\begin{definition}\label{PGL}
A continuous family of curvature conditions is called \textit{Perelman flexible} or \textit{P-flexible} if there exists a real number $0<\tau\leq 1$ such that
\begin{equation}\nonumber
\mathfrak{b}\KN\mathfrak{b}^\perp+r\cdot\mathfrak{b}\KN\mathfrak{b}+B^\perp_{\tau\sqrt{r}}\subset\mathsf{C}(0)
\end{equation}
for all $0<r\leq 1$, where $B^\perp_{\tau\sqrt{r}}=B^\perp_{\tau\sqrt{r}}(\mathscr{C}_B(\mathbb{R}^n))$ is the perpendicular ball
\begin{equation}\nonumber
B_{\tau\sqrt{r}}^\perp(\mathscr{C}_B(\mathbb{R}^n)):=B_{\tau\sqrt{r}}(\mathscr{C}_B(\mathbb{R}^n))\cap\{R\in\mathscr{C}_B(\mathbb{R}^n):R(X,Y,Z,W)=0\;\text{for all $X,Y,Z,W\in e_1^\perp$}\}.
\end{equation}
A continuous family of curvature conditions is called \textit{Gromov-Lawson flexible} or \textit{GL-flexible} if
\begin{equation}\nonumber
\mathfrak{b}\KN\mathfrak{b}^\perp\in\mathsf{C}(0).
\end{equation}
\end{definition}
\begin{remark}\label{remark}
A Perelman flexible family of curvature conditions is Gromov-Lawson flexible if and only if the first condition in Definition~\ref{PGL} is also satisfied for $r=0$.
Moreover, it is easy to see that every Gromov-Lawson flexible family is Perelman flexible.
In this sense, Gromov-Lawson flexible families form a singularity among Perelman flexible families.
\end{remark}

Proposition~\ref{familyembedding} is a genericity result for continuous families of curvature conditions.
We will use the following Lemma.
\begin{lemma}\label{C+C}
Let $\mathsf{C}\subset\mathscr{C}_B(\mathbb{R}^n)$ be an open, convex and $\On(n)$-invariant cone. Then $\overline{\mathsf{C}}+\mathsf{C}\subset\mathsf{C}$.
\end{lemma}
\begin{proof}
Let $R\in\mathsf{C}$ and $S\in\overline{\mathsf{C}}$.
We choose $\varepsilon>0$ with $B_\varepsilon(R)\subset\mathsf{C}$ and $\tilde{S}\in\mathsf{C}$ with $\Vert S-\tilde{S}\Vert<\varepsilon$.
Then $R+S-\tilde{S}\in B_\varepsilon(R)\subset\mathsf{C}$, so $R+S=(R+S-\tilde{S})+\tilde{S}\in\mathsf{C}$.
\end{proof}
\begin{proposition}\label{familyembedding}
Let $\mathsf{C}\subset\mathscr{C}_B(\mathbb{R}^n)$ be an open, convex and $\On(n)$-invariant cone.
Then there exists a continuous family $(\mathsf{C}(\sigma))_{\sigma\in\mathbb{R}}$ of curvature conditions with $\mathsf{C}=\mathsf{C}(0)$.
\end{proposition}
\begin{proof}
We define a family of curvature conditions by
\begin{align}\nonumber
&\mathsf{C}(-\sigma):=\{R\in\mathscr{C}_B(\mathbb{R}^n):\dist(R,\mathsf{C})<\sigma\},\\\nonumber
&\mathsf{C}(\sigma):=\{R\in\mathsf{C}:R+\overline{\mathsf{C}(-\sigma)}\subset\mathsf{C}\}
\end{align}
for $\sigma>0$ and $\mathsf{C}(0):=\mathsf{C}$.

Let us first show that these are indeed curvature conditions.
Pick $\sigma>0$ and $R\in\mathsf{C}(-\sigma)$.
There exists $S\in\overline{\mathsf{C}}$ such that $\dist(R,S)=\dist(R,\mathsf{C})$.
The ball of radius $\sigma-\dist(R,S)$ around $R$ is still contained in $\mathsf{C}(-\sigma)$.
Thus, $\mathsf{C}(-\sigma)$ is open.
As $\mathsf{C}$ and the distance function are $\On(n)$-invariant, $\mathsf{C}(-\sigma)$ is also $\On(n)$-invariant.

Now let $R\in\mathsf{C}(\sigma)$.
Since $\mathsf{C}$ is open and $\mathscr{K}:=\overline{B_\sigma(\mathscr{C}_B(\mathbb{R}^n))}\subset\overline{\mathsf{C}(-\sigma)}$ is compact, there exists $\varepsilon>0$ such that $\tilde{R}+\mathscr{K}\subset\mathsf{C}$ for all $\tilde{R}\in B_\varepsilon(R)$.
Let $\tilde{R}\in B_\varepsilon(R)$ and $T\in\overline{\mathsf{C}(-\sigma)}$.
Choose $S\in\overline{C}$ with $\dist(T,S)=\dist(T,\mathsf{C})$.
This means $\Vert T-S\Vert\leq\sigma$ or $T-S\in\mathscr{K}$.
Hence $\tilde{R}+T=\tilde{R}+(T-S)+S\in\mathsf{C}$ by Lemma~\ref{C+C}.
The $\On(n)$-invariance follows from $\On(n)$-invariance of $\mathsf{C}$ and $\mathsf{C}(-\sigma)$.

Next we check properties $\text{(i)-(iii)}$.
Let $\sigma_1,\sigma_2\in\mathbb{R}$ with $\sigma_1>\sigma_2$.
If $\sigma_1\geq 0,\sigma_2<0$ or $\sigma_1\leq 0,\sigma_2<0$, the inclusion $\overline{\mathsf{C}(\sigma_1)}\subset\mathsf{C}(\sigma_2)$ is clear.
For $\sigma_1>0,\sigma_2\geq 0$ and $R\in\overline{\mathsf{C}(\sigma_1)}$ it holds
\begin{equation}\nonumber
R+\overline{\mathsf{C}(-\sigma_2)}\subset R+\mathsf{C}(-\sigma_1)\subset\mathsf{C}
\end{equation}
where the openness of $\mathsf{C}(-\sigma_1)$ is important for the last inclusion.
This shows $\text{(i)}$.

For $\text{(iii)}$ we pick $\sigma\in\mathbb{R}$ and $\lambda>0$.
If $\sigma<0$ and $R\in\mathsf{C}(\sigma)$, let $S\in\overline{\mathsf{C}}$ with $\dist(R,S)=\dist(R,\mathsf{C})$.
Then
\begin{equation}\nonumber
\dist(\lambda R,\mathsf{C})\leq\dist(\lambda R,\lambda S)\leq-\lambda\sigma,
\end{equation}
so $\lambda R\in\mathsf{C}(\lambda\sigma)$.
For $\sigma=0$ the inclusion $\lambda\mathsf{C}\subset\mathsf{C}$ is clear.
For $\sigma>0$ and $R\in\mathsf{C}(\sigma),T\in\overline{\mathsf{C}(-\lambda\sigma)}$, it holds
\begin{equation}\nonumber
\lambda R+T=\lambda\bigl(R+\tfrac{1}{\lambda}T\bigr)\in\lambda\cdot\Bigl(\mathsf{C}(\sigma)+\tfrac{1}{\lambda}\cdot\overline{\mathsf{C}(-\lambda\sigma)}\Bigr)\subset\lambda\cdot\Bigl(\mathsf{C}(\sigma)+\overline{\mathsf{C}(-\sigma)}\Bigr)\subset\lambda\cdot\mathsf{C}\subset\mathsf{C}.
\end{equation}

For $\text{(ii)}$ we pick $\sigma_1,\sigma_2\in\mathbb{R}$.
If $\sigma_1\leq 0,\sigma_2<0$ or $\sigma_1=\sigma_2=0$, the inclusion $\overline{\mathsf{C}(\sigma_1)}+\mathsf{C}(\sigma_2)\subset\mathsf{C}(\sigma_1+\sigma_2)$ is clear.
If $\sigma_1<0,\sigma_2=0$ and $R_1\in\overline{\mathsf{C}(\sigma_1)},R_2\in\mathsf{C}$, we choose $S_1\in\overline{\mathsf{C}}$ with $\dist(R_1,S_1)=\dist(R_1,\mathsf{C})$.
Then $S_1+R_2\in\mathsf{C}$.
Let $0<\varepsilon<1$ be so small that $S_1+R_2+\varepsilon(R_1-S_1)\in\mathsf{C}$.
We obtain
\begin{equation}\nonumber
\dist(R_1+R_2,S_1+R_2+\varepsilon(R_1-S_1))=\dist((1-\varepsilon)\cdot R_1,(1-\varepsilon)\cdot S_1)\leq (1-\varepsilon)\cdot(-\sigma_1)<\sigma_1,
\end{equation}
hence $\dist(R_1+R_2,\mathsf{C})<\sigma_1$.

In the next stage, we consider $\sigma_1\geq 0,\sigma_2\geq 0$.
Let $R_1\in\overline{\mathsf{C}(\sigma_1)},R_2\in\mathsf{C}(\sigma_2)$ and $T\in\overline{\mathsf{C}(-(\sigma_1+\sigma_2))}$.
Then $\tfrac{\sigma_1}{\sigma_1+\sigma_2}T\in\overline{\mathsf{C}(-\sigma_1)}$ and $\tfrac{\sigma_2}{\sigma_1+\sigma_2}T\in\overline{\mathsf{C}(-\sigma_2)}$.
Therefore,
\begin{equation}\nonumber
R_1+R_2+T=R_1+\tfrac{\sigma_1}{\sigma_1+\sigma_2}T+R_2+\tfrac{\sigma_2}{\sigma_1+\sigma_2}T\in\overline{\mathsf{C}}+\mathsf{C}\subset\mathsf{C}.
\end{equation}
Now let $\sigma_1>0,\sigma_2<0$ with $\sigma_1+\sigma_2<0$.
Let $R_1\in\overline{\mathsf{C}(\sigma_1)}$ and $R_2\in\mathsf{C}(\sigma_2)$.
Then $\tfrac{\sigma_1}{-\sigma_2}R_2\in\mathsf{C}(-\sigma_1)$ and $\tfrac{\sigma_1+\sigma_2}{\sigma_2}R_2\in\mathsf{C}(\sigma_1+\sigma_2)$.
This yields
\begin{equation}\nonumber
R_1+R_2=\Bigl(R_1+\tfrac{\sigma_1}{-\sigma_2}R_2\Bigr)+\tfrac{\sigma_1+\sigma_2}{\sigma_2}R_2\in\mathsf{C}+\mathsf{C}(\sigma_1+\sigma_2)\subset\mathsf{C}(\sigma_1+\sigma_2).
\end{equation}
The case $\sigma_1<0,\sigma_2>0$ and $\sigma_1+\sigma_2<0$ is similar.
Finally, let $\sigma_1>0,\sigma_2<0$ with $\sigma_1+\sigma_2\geq 0$.
Let $R_1\in\overline{\mathsf{C}(\sigma_1)},R_2\in\mathsf{C}(\sigma_2)$ and $T\in\overline{\mathsf{C}(-(\sigma_1+\sigma_2))}$.
Then
\begin{equation}\nonumber
R_1+R_2+T\in\overline{\mathsf{C}(\sigma_1)}+\mathsf{C}(\sigma_2)+\overline{\mathsf{C}(-(\sigma_1+\sigma_2))}\subset\overline{\mathsf{C}(\sigma_1)}+\mathsf{C}(-\sigma_1)\subset\mathsf{C}.
\end{equation}
The case $\sigma_1<0,\sigma_2>0$ and $\sigma_1+\sigma_2\geq 0$ is similar.

It remains to consider continuity.
For $\sigma<0$, we define
\begin{equation}\nonumber
\varepsilon(\sigma)\colon\mathsf{C}(\sigma)\to(0,\infty)\,,\,\varepsilon(\sigma)(R):=\tfrac{1}{2}\bigl(-\sigma-\dist(R,\mathsf{C})\bigr).
\end{equation}
This is a continuous function with
\begin{equation}\nonumber
R\in\mathsf{C}(\sigma+\varepsilon(\sigma)(R))=\mathsf{C}\bigl(-\dist(R,\mathsf{C})-\tfrac{1}{2}\bigl(\sigma-\dist(R,\mathsf{C})\bigr)\bigr).
\end{equation}
Let $\sigma\geq 0$ and let $\mathscr{K}:=\overline{B_\sigma(\mathscr{C}_B(\mathbb{R}^n))}$ be as in the proof of openness.
For $\mathsf{C}\neq\mathscr{C}_B(\mathbb{R}^n)$ we observe that
\begin{equation}\nonumber
\mathsf{C}\to(0,\infty)\,,\,R\mapsto\sup\{\varepsilon>0:B_\varepsilon(R)\subset\mathsf{C}\}=\dist(R,\mathscr{C}_B(\mathbb{R}^n)\setminus\mathsf{C})
\end{equation}
is a well-defined and continuous function.
Hence $\mathsf{C}(\sigma)\times\mathscr{K}\to(0,\infty)\,,\,(R,T)\mapsto\dist(R+T,\mathscr{C}_B(\mathbb{R}^n)\setminus\mathsf{C})$ and
\begin{equation}\nonumber
\varepsilon(\sigma)\colon\mathsf{C}(\sigma)\to(0,\infty)\,,\,\varepsilon(\sigma)(R):=\tfrac{1}{2}\min_{T\in\mathscr{K}}\dist(R+T,\mathscr{C}_B(\mathbb{R}^n)\setminus\mathsf{C})
\end{equation}
are continuous.

Let $R\in\mathsf{C}(\sigma)$.
We check that $R\in\mathsf{C}(\sigma+\varepsilon(\sigma)(R))$.
It holds $R+T+\tilde{R}\in\mathsf{C}$ for all $T\in\mathscr{K}$ and $\tilde{R}\in B_{2\varepsilon}(\mathscr{C}_B(\mathbb{R}^n))$ where $\varepsilon:=\varepsilon(\sigma)(R)$.
Let $T\in\overline{\mathsf{C}(-(\sigma+\varepsilon))}$.
Then $\tfrac{\sigma}{\sigma+\varepsilon}T\in\overline{\mathsf{C}(-\sigma)}$ and $\tfrac{\varepsilon}{\sigma+\varepsilon}T\in\overline{\mathsf{C}(-\varepsilon)}$.
We choose $S_1,S_2\in\overline{\mathsf{C}}$ with $\dist(\tfrac{\sigma}{\sigma+\varepsilon}T,S_1)=\dist(\tfrac{\sigma}{\sigma+\varepsilon}T,\mathsf{C})$ and $\dist(\tfrac{\varepsilon}{\sigma+\varepsilon}T,S_2)=\dist(\tfrac{\varepsilon}{\sigma+\varepsilon}T,\mathsf{C})$.

This yields $\Vert\tfrac{\sigma}{\sigma+\varepsilon}T-S_1\Vert\leq\sigma$ and $\Vert\tfrac{\varepsilon}{\sigma+\varepsilon}T-S_1\Vert\leq\varepsilon$, hence $\tfrac{\sigma}{\sigma+\varepsilon}T-S_1\in\mathscr{K}$ and $\tfrac{\varepsilon}{\sigma+\varepsilon}T-S_2\in B_{2\varepsilon}(\mathscr{C}_B(\mathbb{R}^n))$.
We conclude that
\begin{align}\nonumber
R+T=R+\bigl(&\tfrac{\sigma}{\sigma+\varepsilon}T-S_1\bigr)+\bigl(\tfrac{\varepsilon}{\sigma+\varepsilon}T-S_2\bigr)+S_1+S_2\\\nonumber
&\in R+\mathscr{K}+B_{2\varepsilon}(\mathscr{C}_B(\mathbb{R}^n))+\overline{\mathsf{C}}\subset\mathsf{C}+\overline{\mathsf{C}}\subset\mathsf{C}.\tag*{\qedhere}
\end{align}
\end{proof}
\begin{remark}
If $\mathfrak{b}\KN\mathfrak{b}\in\mathsf{C}$, one can simply define $\mathsf{C}(\sigma)=\mathsf{C}-\sigma\cdot\mathfrak{b}\KN\mathfrak{b}$ for all $\sigma\in\mathbb{R}$.
\end{remark}

\subsection{Examples of flexible families}
We want to show that all curvature conditions in Table~\ref{overview} lead to $P$-flexible or $GL$-flexible families.
First, the definitions should be reviewed and compiled.

Let $R\in\mathscr{C}_B(\mathbb{R}^n)$.
For the sectional curvature $\sec(R)(\pi)$ of a plane $\pi\subset\mathbb{R}^n$ which is spanned by an orthonormal pair $(v_1,v_2)\subset\mathbb{R}^n$, we use the convention
\begin{equation}\nonumber
\sec(R)(\pi)=R(v_1,v_2,v_2,v_1).
\end{equation}
The associated self-adjoint curvature operator $\mathcal{R}\colon\Lambda^2\mathbb{R}^n\to\Lambda^2\mathbb{R}^n$ is normalised by the condition that
\begin{equation}\nonumber
\langle\mathcal{R}(v_1\wedge v_2),v_3\wedge v_4\rangle=-R(v_1,v_2,v_3,v_4)
\end{equation}
for all $v_1,v_2,v_3,v_3\in\mathbb{R}^n$.
One can regard $\mathcal{R}$ as a symmetric bilinear form
\begin{equation}\nonumber
A_{\mathcal{R}}\colon\Lambda^2\mathbb{R}^n\times\Lambda^2\mathbb{R}^n\to\mathbb{R}\,,\,A_{\mathcal{R}}(\omega,\eta)=\langle\mathcal{R}(\omega),\eta\rangle.
\end{equation}
The Ricci tensor $\ric(R)$ is defined by $\ric(R)(v,w)=\sum_{i=1}^n R(u_i,v,w,u_i)$ for all $v,w\in\mathbb{R}^n$, where $(u_1,\dots,u_n)$ is an orthonormal basis of $\mathbb{R}^n$.

\begin{itemize}[itemsep=1em]
\item[$\triangleright$]{We say that $R$ has \textit{positive curvature operator} (denoted $\mathcal{R}>0$) if $A_{\mathcal{R}}$ is positive definite.
Our definition for $\mathcal{R}$ is made in such a way that the round sphere has a positive curvature operator.

The gluing theorem for $\mathcal{R}>0$ is due to Schlichting~\cite[Thm.~1.1.2]{Schlichting2014} and Chow~\cite[Thm.~1]{Chow}.}
\item[$\triangleright$]{We say that $R$ has \textit{$2$-positive curvature operator} (denoted $\mathcal{R}_2>0$) if $A_{\mathcal{R}}$ is $2$-positive.

This case has been treated by Schlichting~\cite[Thm.~1.7.3]{Schlichting2014}.}

\item[$\triangleright$]{Micallef-Moore~\cite{MM} introduced the concept of isotropic curvature:
We extend $R$ to a complex multilinear form $R^{\mathbb{C}}$ on $\mathbb{C}^n$ and say that $R$ has \textit{positive isotropic curvature} (denoted $\PIC$) if
\begin{align}\nonumber
R^{\mathbb{C}}(\zeta,\eta,\overline{\eta},\overline{\zeta})>0\quad&\text{for all linearly independent vectors $\zeta,\eta\in\mathbb{C}^n$}\\\nonumber
&\text{with $\langle\zeta,\zeta\rangle_{\mathbb{C}}=\langle\zeta,\eta\rangle_{\mathbb{C}}=\langle\eta,\eta\rangle_{\mathbb{C}}=0$,}
\end{align}
where $\langle\cdot,\cdot\rangle_\mathbb{C}$ is the complexified standard inner product on $\mathbb{R}^n$.
Equivalently,
\begin{align}\nonumber
R&(u_1,u_3,u_3,u_1)+R(u_1,u_4,u_4,u_1)\\\nonumber
&+R(u_2,u_3,u_3,u_2)+R(u_2,u_4,u_4,u_2)\\\nonumber
&+2\hspace{1pt}R(u_1,u_2,u_3,u_4)>0
\end{align}
for all orthonormal four-frames $(u_1,u_2,u_3,u_4)\subset\mathbb{R}^n$.

Schlichting~\cite[Thm.~1.7.4]{Schlichting2014} proved the gluing result for convex boundaries.
Chow~\cite[Thm.~1]{Chow} improved this to $2$-convex boundaries.}
\item[$\triangleright$]{We say that $R$ has $\PIC1$ if
\begin{align}\nonumber
R^{\mathbb{C}}(\zeta,\eta,\overline{\eta},\overline{\zeta})>0\quad&\text{for all linearly independent vectors $\zeta,\eta\in\mathbb{C}^n$}\\\nonumber
&\text{with $\langle\zeta,\zeta\rangle_{\mathbb{C}}\langle\eta,\eta\rangle_{\mathbb{C}}-\langle\zeta,\eta\rangle_{\mathbb{C}}^2=0$.}
\end{align}
Equivalently,
\begin{align}\nonumber
R&(u_1,u_3,u_3,u_1)+\lambda^2R(u_1,u_4,u_4,u_1)\\\nonumber
&+R(u_2,u_3,u_3,u_2)+\lambda^2R(u_2,u_4,u_4,u_2)\\\nonumber
&+2\lambda\hspace{1pt}R(u_1,u_2,u_3,u_4)>0
\end{align}
for all orthonormal four-frames $(u_1,u_2,u_3,u_4)\subset\mathbb{R}^n$ and all $\lambda\in[0,1]$.

The gluing result was proved by Schlichting~\cite[Thm.~1.7.4]{Schlichting2014} and Chow~\cite[Thm.~1]{Chow}.
}
\item[$\triangleright$]{We say that $R$ has $\PIC2$ if
\begin{equation}\nonumber
R^{\mathbb{C}}(\zeta,\eta,\overline{\eta},\overline{\zeta})>0\quad\text{for all linearly independent vectors $\zeta,\eta\in\mathbb{C}^n$.}
\end{equation}
Equivalently,
\begin{align}\nonumber
R&(u_1,u_3,u_3,u_1)+\lambda^2R(u_1,u_4,u_4,u_1)\\\nonumber
&+\mu^2R(u_2,u_3,u_3,u_2)+\lambda^2\mu^2R(u_2,u_4,u_4,u_2)\\\nonumber
&+2\lambda\mu\hspace{1pt}R(u_1,u_2,u_3,u_4)>0
\end{align}
for all orthonormal four-frames $(u_1,u_2,u_3,u_4)\subset\mathbb{R}^n$ and all $\lambda,\mu\in[0,1]$.

The gluing result was proved by Schlichting~\cite[Thm.~1.7.4]{Schlichting2014} and Chow~\cite[Thm.~1]{Chow}.
}

\item[$\triangleright$]{We say that $R$ has \textit{positive sectional curvature} (denoted $\sec>0$) if $R(v_1,v_2,v_2,v_1)>0$ for all orthonormal pairs $(v_1,v_2)\subset\mathbb{R}^n$.

The gluing theorem for $\sec>0$ was proved by Kosovskii~\cite[Thm.~1.1]{Kosovskii} in the setting of Alexandrov spaces and by Reiser-Wraith~\cite[Thm.~A]{RW} in a smooth version.}
\item[$\triangleright$]{Let $1\leq k\leq n-1$.
The tensor $R$ has \textit{positive $k^{th}$-intermediate Ricci curvature} (denoted $\ric_k>0$) if for all orthonormal $(k+1)$-frames $(v,u_1,\dots,u_k)\subset\mathbb{R}^n$ the sum $\sum_{i=1}^k R(u_i,v,v,u_i)$ is positive.

The gluing result for $\ric_k>0$ was proved by Reiser-Wraith~\cite[Thm.~A]{RW} in the general case. The case $k=2$ was noted by Schlichting~\cite[Thm.~7.5]{Schlichting2012} before.}
\item[$\triangleright$]{We say that $R$ has \textit{positive Ricci curvature} (denoted $\ric>0$) if $\ric(R)$ is positive definite.

Detailed proofs for Perelman's original gluing theorem have been given by Schlichting~\cite[Thm.~1.7.1]{Schlichting2014}, Botvinnik-Walsh-Wraith~\cite[Thm.~2]{BWW} and Reiser-Wraith~\cite[Thm.~A]{RW}.}
\item[$\triangleright$]{Let $1\leq k\leq n$.
The tensor $R$ has \textit{$k$-positive Ricci curvature} (denoted $\scal_k>0$) if $\ric(R)$ is $k$-positive.
This curvature notion goes back to Wolfson~\cite{Wolfson}.

The gluing result for $\scal_k>0$ is due to Reiser-Wraith~\cite[Thm.~A]{RW}.}
\item[$\triangleright$]{The scalar curvature of $R$ is $\scal(R)=\sum_{i=1}^n\ric(R)(u_i,u_i)$, where $(u_1,\dots,u_n)$ is an orthonormal basis of $\mathbb{R}^n$.
It has \textit{positive scalar curvature} (denoted $\scal>0$) if $\scal(R)>0$.

With regard to the gluing theorem, the special case of doublings was considered by Gromov-Lawson~\cite[Thm.~5.7]{GL} and Almeida~\cite[Thm.~1.1]{Almeida}.
A proof for the general case has been provided by Brendle-Marques-Neves~\cite[Thm.~5]{BMN}.
Bär-Hanke~\cite[Thm.~4.11]{BH2023} extended the gluing to a homotopical version.
The author~\cite[Thm.~3.17]{Frerichs2025} contributed the case where the boundary is non-compact.}
\item[$\triangleright$]{Let $1\leq m\leq n-1$ and let $(u_1,\dots,u_m)$ be an orthonormal $m$-frame in $\mathbb{R}^n$.
We extend the frame to an orthonormal basis $(u_1,\dots,u_n)$ of $\mathbb{R}^n$.
Then define
\begin{equation}\nonumber
\mathcal{C}_m(u_1,\dots,u_m)=\sum_{p=1}^m\sum_{q=p+1}^n R(u_p,u_q,u_q,u_p).
\end{equation}
We say that $R$ has \textit{positive $m$-intermediate curvature} (denoted $\mathcal{C}_m>0$) if $\mathcal{C}_m(u_1,\dots,u_m)>0$ for all orthonormal $m$-frames $(u_1,\dots,u_m)\subset\mathbb{R}^n$.
This concept was introduced by Brendle-Hirsch-Johne~\cite{BHJ}.

The gluing result is due to Chow-Johne-Wan~\cite[Thm.~1.2]{CJW}.}
\end{itemize}

There are some implications between the above mentioned curvature properties, which are collected in Figure~\ref{diagram}.
We write $\mathsf{C}_{\mathcal{R}>0}\longrightarrow\mathsf{C}_{\PIC2}$ to mean that the cone of positive curvature operators is contained in the cone of tensors with $\PIC2$, and analogously for the other relations.
Most of the inclusions are either obvious or are justified in~\cite[Chap.~7]{Brendle}.
The only exceptions are $\mathsf{C}_{\PIC}\subset\mathsf{C}_{\hspace{1pt}\mathcal{C}_m>0}$ for $m\geq 3$ and $\mathsf{C}_{\ric_k>0}\subset\mathsf{C}_{\hspace{1pt}\mathcal{C}_m>0}$ for $m\geq k$.
The inclusion for $\PIC$ has recently been demonstrated by Chow-Wang~\cite[Prop.~6]{CW}.
The inclusion for $\ric_k>0$ is proven in Fact~\ref{inclusion}.
\begin{figure}[h]
\begin{tikzcd}
\mathsf{C}_{\mathcal{R}>0}\arrow[rrr]\arrow[ddd]&&&\mathsf{C}_{\PIC2}\arrow[rrrr]\arrow[ddd]&&&&\mathsf{C}_{\sec>0}\arrow[ddd]\\
\\
\\
\mathsf{C}_{\mathcal{R}_2>0}\arrow{rrr}&&&\mathsf{C}_{\PIC1}\arrow[rrrr,"(k\geq 2)"]\arrow[dddd]&&&&\raisebox{-0.8ex}{$\mathsf{C}_{\ric_k>0}$}\arrow[d]\arrow[lldd,"(m\geq k)" swap]\\
&&&&&&&\mathsf{C}_{\ric>0}\arrow{dd}\\[-1.7em]
&&&&&\mathsf{C}_{\hspace{1pt}\mathcal{C}_m>0}\arrow{rrdd}&&\\[-1.7em]
&&&&&&&\mathsf{C}_{\scal_k>0}\arrow{d}\\
&&&\mathsf{C}_{\PIC}\arrow{rrrr}\arrow[rruu,"(m\geq 3)"]&&&&\mathsf{C}_{\scal>0}
\end{tikzcd}
\caption{Inclusions of curvature cones.}
\label{diagram}
\end{figure}
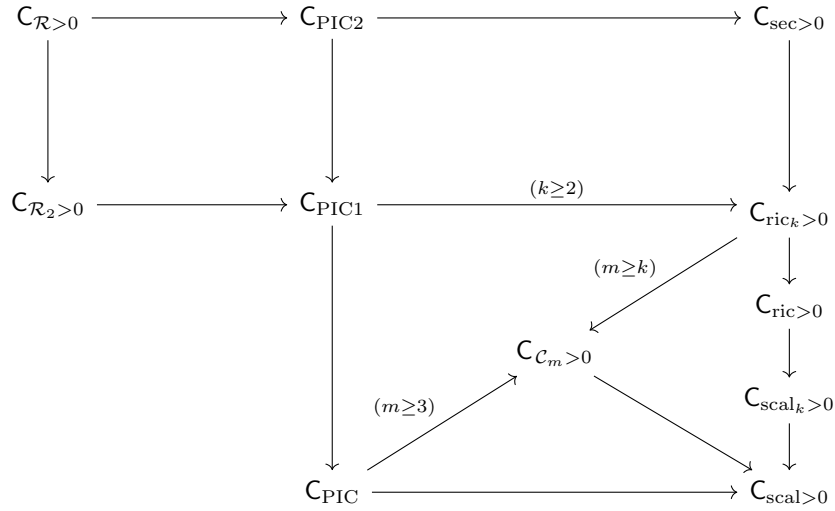

\begin{fact}\label{inclusion}
It holds $\mathsf{C}_{\ric_k>0}\subset\mathsf{C}_{\hspace{1pt}\mathcal{C}_m>0}$ for all $1\leq k\leq m\leq n-1$.
\end{fact}
\begin{proof}
Let $1\leq k\leq m\leq n-1$.
Let $R\in\mathsf{C}_{\ric_k>0}$ and let $(u_1,\dots,u_m)$ be an orthonormal $m$-frame in $\mathbb{R}^n$.
We extend this frame to an orthonormal basis $(u_1,\dots,u_n)$ of $\mathbb{R}^n$.
For each $1\leq p\leq m$, every summand $R(u_p,u_q,u_q,u_p)$ with $1\leq q\leq n$ appears exactly once in the sum for $\mathcal{C}_m(u_1,\dots,u_m)$.

We first consider the sum where for each $1\leq p\leq n$, every summand $R(u_p,u_q,u_q,u_p)$ with $1\leq q\leq n$ appears exactly once.
It holds
\begin{equation}\nonumber
\sum\limits_{1\leq p<q\leq n}R(u_p,u_q,u_q,u_p)=\sum\limits_{i=1}^n\left(\tfrac{1}{2}\cdot\sum\limits_{l=1}^nR(u_i,u_l,u_l,u_i)\right).
\end{equation}
In order to obtain $\mathcal{C}_m(u_1,\dots,u_m)$ from this sum, we have to delete all summands where both indices $i$ and $l$ are $>m$.
Hence,
\begin{equation}\nonumber
\mathcal{C}_m(u_1,\dots,u_m)=\sum_{i=1}^m\left(\tfrac{1}{2}\cdot\sum\limits_{l=1}^nR(u_i,u_l,u_l,u_i)\right)+\sum_{i=m+1}^n\left(\tfrac{1}{2}\cdot\sum\limits_{l=1}^{m}R(u_i,u_l,u_l,u_i)\right).
\end{equation}
The assertion follows.
\end{proof}

Figure~\ref{diagram} shows conditions for positive curvature.
In all examples, there is a standard way to define curvature $>\sigma$ for bounds $\sigma\in\mathbb{R}$ other than zero.
For instance, we say that $R$ has \textit{curvature operator $>\sigma$} (denoted $\mathcal{R}>\sigma$) if $A_{\mathcal{R}}(\omega,\omega)>\sigma$ for all $\omega\in\Lambda^2\mathbb{R}^n$ with $\Vert\omega\Vert=1$.
Then $\mathsf{C}_{\mathcal{R}}=(\mathsf{C}_{\mathcal{R}>\sigma})_{\sigma\in\mathbb{R}}$ is a continuous family of curvature conditions.
The same terminology applies to all other curvature conditions in Figure~\ref{diagram}.

The first goal is to demonstrate that all of these families are P-flexible.

\begin{proposition}
There exists $0<\tau\leq 1$ such that $\mathfrak{b}\KN\mathfrak{b}^\perp+r\cdot\mathfrak{b}\KN\mathfrak{b}+R$ has positive curvature operator for all $0<r\leq 1$ and $R\in B^\perp_{\tau\sqrt{r}}$.
\end{proposition}
\begin{proof}
Let $0<\tau,r\leq 1$ and let $R\in B^\perp_{\tau\sqrt{r}}$.
We view the curvature operator of $R$ as a symmetric bilinear form.
In the standard basis $(e_i\wedge e_l)_{1\leq i<l\leq n}$ of $\Lambda^2\mathbb{R}^n$, it is represented by the entries of $R$.

In the same manner, the curvature operator associated to $\mathfrak{b}\KN\mathfrak{b}^\perp+r\cdot\mathfrak{b}\KN\mathfrak{b}+R$ is given by the matrix
\begin{equation}\nonumber
\left(\begin{array}{c|c}
A &B\\
\hline\\[-0.35cm]
B^{\mathsf{T}}& C
\end{array}\right)
=
\left(\begin{array}{ccc|ccc}
1+r+R_{1221}&\cdots&R_{1n21}&R_{2321}&\cdots&\\
\vdots&\ddots&\vdots&\vdots&&\\
R_{12n1}&\cdots&1+r+R_{1nn1}&R_{23n1}&\cdots&\\
\hline
\vphantom{\rule{0pt}{0.35cm}}
R_{1232}&\cdots& R_{1n32}&r&\cdots&0\\
\vdots&&\vdots&\vdots&\ddots&\vdots\\
&& &0&\cdots&r
\end{array}\right).
\end{equation}
Since $C$ is positive definite, the block matrix is positive definite if the Schur complement $A-BC^{-1}B^\mathsf{T}$ is positive definite.
To show this, we first consider the matrix $A$:

The eigenvalues of $A$ are given exactly by $1+r+\lambda$, where $\lambda$ is an eigenvalue of $A-(1+r)\cdot I_n$.
It holds $|\lambda|\leq\Vert R\Vert<\tau\sqrt{r}$, so $\lambda>-\tau\sqrt{r}$.
This yields
\begin{equation}\label{EVestimate}
1+r+\lambda\geq 1+r-\tau\sqrt{r}=\left(1-\frac{\tau}{2}\sqrt{r}\right)^2+\left(1-\frac{\tau^2}{4}\right)r\geq\frac{1}{4}.
\end{equation}
Now we have $A-BC^{-1}B^{\mathsf{T}}=A-\tfrac{1}{r}BB^{\mathsf{T}}$ with
\begin{equation}\nonumber
\left\Vert\tfrac{1}{r}BB^{\mathsf{T}}\right\Vert\leq\tfrac{1}{r}\Vert B\Vert^2<\tfrac{1}{r}\left(\tau\sqrt{r}\right)^2=\tau^2.
\end{equation}
Since the eigenvalues of a matrix depend continuously on its entries and since the estimate~\eqref{EVestimate} would hold on the compact domain $\{0\leq r\leq 1,|\lambda|\leq\sqrt{r}\}$ as well, one can choose $0<\tau\leq 1$ small enough such that $A-BC^{-1}B^{\mathsf{T}}$ has only positive eigenvalues, independent of $0<r\leq 1$ and $R\in B_{\tau\sqrt{r}}^\perp$.
\end{proof}

We now show that the families $\mathsf{C}_{\ric},\mathsf{C}_{\scal_k},\mathsf{C}_{\scal}$ and $\mathsf{C}_{\hspace{1pt}\mathcal{C}_m}$ are even GL-flexible.

\begin{example}
It holds
\begin{itemize}
\item[\textit{(a)}]{$\mathfrak{b}\KN\mathfrak{b}^\perp\in\mathsf{C}_{\ric>0}$,}
\item[\textit{(b)}]{$\mathfrak{b}\KN\mathfrak{b}^\perp\in\mathsf{C}_{\hspace{1pt}\mathcal{C}_m>0}$ for all $1\leq m\leq n-1$.}
\end{itemize}
\end{example}
\begin{proof}
\textit{(a)} Let $u\in\mathbb{R}^n$ be a unit vector.
We extend this to an orthonormal basis $(u_1,u_2,\dots,u_n)$ of $\mathbb{R}^n$ where $u_1:=u$.
It holds
\begin{equation}\nonumber
\ric(\mathfrak{b}\KN\mathfrak{b}^\perp)(u,u)=\sum_{i=2}^n(\mathfrak{b}\KN\mathfrak{b}^\perp)(u_1,u_i,u_i,u_1)=\sum_{i=2}^n(\mathfrak{b}^\perp(u_i,u_i)+\mathfrak{b}^\perp(u_1,u_1))>0
\end{equation}
because all summands are $\geq 0$ and one of the basis vectors must have a non-trivial $e_1$-component.

\textit{(b)} Let $1\leq m\leq n-1$ and let $(u_1,\dots,u_m)$ be an orthonormal $m$-frame in $\mathbb{R}^n$.
We extend this frame to an orthonormal basis $(u_1,\dots,u_n)$ of $\mathbb{R}^n$.
It holds
\begin{align}\nonumber
\mathcal{C}_m(\mathfrak{b}\KN\mathfrak{b}^\perp)(u_1,\dots,u_m)&=\sum_{p=1}^m\sum_{q=p+1}^n(\mathfrak{b}\KN\mathfrak{b}^\perp)(u_p,u_q,u_q,u_p)\\\nonumber
&=\sum_{p=1}^m\sum_{q=p+1}^n(\mathfrak{b}^\perp(u_q,u_q)+\mathfrak{b}^\perp(u_p,u_p))>0
\end{align}
because every basis vector appears in the sum for $\mathcal{C}_m$.
\end{proof}

\begin{example}
The following definition provides a general notion of \textit{curvature almost $\geq\sigma$}.
Let $\mathsf{C}=(\mathsf{C}(\sigma))_{\sigma\in\mathbb{R}}$ be a P-flexible family of curvature conditions.
For $\varepsilon>0$ and $\sigma\in\mathbb{R}$, we set
\begin{equation}\nonumber
\mathsf{C}_{\varepsilon}(\sigma):=\{R\in\mathscr{C}_{B}(\mathbb{R}^n):R+\varepsilon\hspace{1pt}\scal(R)\cdot\mathfrak{b}\KN\mathfrak{b}\in\mathsf{C}(\sigma)\}.
\end{equation}
Then $\mathsf{C}_{\varepsilon}=(\mathsf{C}_{\varepsilon}(\sigma))_{\sigma\in\mathbb{R}}$ is a GL-flexible family of curvature conditions.
\end{example}

\subsection{Gram-Schmidt's algorithm revisited}
We consider a quantitative version of the Gram-Schmidt algorithm.
\begin{lemma}\label{QGS}
For every $n\in\mathbb{N}$ and $c>0$, there exists a constant $C>0$ and a real number $\delta>0$ with the following property:

Let $(V,g)$ be an $n$-dimensional inner product space.
Let $0\leq  t\leq\delta$.
Suppose that $(b_1,\dots,b_m)$ is a basis of V with
\begin{equation}\nonumber
|g(b_i,b_l)-\delta_{il}|\leq c\cdot t\quad\text{for all $1\leq i,l\leq n$.}
\end{equation}

Then there exists a $g$-orthonormal basis $(u_1,\dots,u_m)$ of $V$ such that
\begin{equation}\nonumber
|g(b_i,u_l)-\delta_{il}|\leq C\cdot t\quad\text{for all $1\leq i,l\leq n$.}
\end{equation}
\end{lemma}
\begin{proof}
Let $n\in\mathbb{N}$ and $c>0$.
Let $(V,g)$ be an $n$-dimensional inner product space.
We pick a basis $(b_1,\dots,b_m)$ of $V$ and apply the Gram-Schmidt algorithm to obtain a $g$-orthonormal basis $(u_1,\dots,u_m)$.
This is to say that
\begin{equation}\nonumber
u_k=\frac{u_k'}{|u_k'|_g}\quad\text{with}\quad u_k'=b_k-\sum\limits_{l=1}^{k-1}g(b_k,u_l)\cdot u_l
\end{equation}
for all $1\leq k\leq n$.
We construct a sequence of constants $c\leq C_1\leq\dots\leq C_n$ and a sequence of real numbers $1\geq\delta_1\geq\dots\geq\delta_n>0$ such that the following holds for all $1\leq k\leq n$:
If
\begin{equation}\nonumber
|g(b_i,b_l)-\delta_{il}|\leq c\cdot t\quad\text{for some $0\leq t\leq\delta_k$ and for all $1\leq i,l\leq n$,}
\end{equation}
then it holds
\begin{equation}\nonumber
|g(b_i,u_l)-\delta_{il}|\leq C_k\cdot t\quad\text{for all $1\leq i\leq n$ and $1\leq l\leq k$.}
\end{equation}
Suppose that $C_k$ and $\delta_k$ have been established for some $1\leq k\leq n-1$ and that both of them only depend on $c$.
The aim is to construct $C_{k+1}$ and $\delta_{k+1}$.

Let us assume that there exists $0\leq t\leq\delta_k$ such that $|g(b_i,b_l)-\delta_{il}|\leq c\cdot t$ for all $1\leq i,l\leq n$.
It is sufficient to consider $t>0$.
For $1\leq i\leq n$ with $i\neq k+1$, we compute
\begin{align}\nonumber
|g(b_i,u_{k+1}')|&\leq |g(b_i,b_{k+1})|+\sum\limits_{l=1}^k|g(b_{k+1},u_l)|\cdot|g(b_i,u_l)|\\\nonumber
&\leq ct+\sum\limits_{l=1}^kC_kt\cdot(1+C_kt)\\\nonumber
&=(c+kC_k\cdot(1+C_kt))\cdot t\leq(c+kC_k\cdot(1+C_k))\cdot t.
\end{align}
Furthermore,
\begin{align}\nonumber
g(u_{k+1}',u_{k+1}')&=g(b_{k+1},b_{k+1})-\sum\limits_{l=1}^kg(b_{k+1},u_l)^2\\\nonumber
&\geq 1-ct-k\cdot(C_kt)^2\geq 1-(c+kC_k^2)\cdot t.
\end{align}
We choose $\delta_{k+1}\leq\delta_k$ so small that $(c+kC_k^2)\cdot\delta_{k+1}\leq\tfrac{3}{4}$.
It follows that $|u_{k+1}'|_g\geq\tfrac{1}{2}$ if $t\leq\delta_{k+1}$.
Thus, for $t\leq\delta_{k+1}$ and $1\leq i\leq n$ with $i\neq k+1$,
\begin{equation}\nonumber
|g(b_i,u_{k+1})|\leq 2(c+kC_k\cdot(1+C_k))\cdot t.
\end{equation}
If $i=k+1$, it holds $g(b_{k+1},u_{k+1})=|u_{k+1}'|_g$, so
\begin{align}\nonumber
\frac{1}{t}\bigl(g(b_{k+1},u_{k+1})-1\bigr)\leq\frac{1}{t}\bigl(\sqrt{1+ct}-1\bigr)=f'(r_t)
\end{align}
for $f(r)=\sqrt{1+cr}$ and some $r_t\in(0,t)$.
Therefore, on the one hand,
\begin{equation}\nonumber
\frac{1}{t}\bigl(g(b_{k+1},u_{k+1})-1\bigr)\leq\frac{c}{2\sqrt{1+cr_t}}\leq c.
\end{equation}
On the other hand,
\begin{equation}\nonumber
\frac{1}{t}\bigl(1-g(b_{k+1},u_{k+1})\bigr)\leq\frac{1}{t}\left(1-\sqrt{1-(c+kC_k^2)\cdot t}\right)\leq c+kC_k^2
\end{equation}
with the same type of argument.
This yields
\begin{equation}\nonumber
|g(b_{k+1},u_{k+1})-1|\leq(c+kC_k^2)\cdot t.
\end{equation}
We see that $C_{k+1}:=2c+2kC_k\cdot(1+C_k)$ does the job.
It is easy to check that $C_1:=2c$ and $\delta_1:=\tfrac{3}{4c}$ work for the first construction step.
The assertion holds with $C:=C_n$ and $\delta:=\delta_n$.
\end{proof}

\begin{corollary}\label{uniforma}
Let $\mathscr{K}$ be a compact subspace of $\mathscr{C}_B(\mathbb{R}^n)$.
Let $r>0$.
Then there exists $a>0$ such that for all $R\in\mathscr{K}$ and all linear isomorphisms $j\colon\mathbb{R}^n\to\mathbb{R}^n$ with
\begin{equation}\nonumber
|\mathfrak{b}(j(e_i),j(e_l))-\delta_{il}|<a\quad\text{for all $1\leq i,l\leq n$},
\end{equation}
there exists a linear isometry $\iota\colon\mathbb{R}^n\to\mathbb{R}^n$ such that $j^\ast R\in B_r(\iota^\ast R)$.
\end{corollary}

\begin{proof}
Firstly, there exists $b>0$ such that for all $1\leq i_1,i_2,i_3,i_4\leq n$, it holds
\begin{equation}\label{anearly1}
|R(v_1,v_2,v_3,v_4)-R(e_{i_1},e_{i_2},e_{i_3},e_{i_4})|<\frac{r}{n^2}
\end{equation}
for all $v_1,v_2,v_3,v_4\in\mathbb{R}^n$ with $\Vert(v_1,v_2,v_3,v_4)-(e_{i_1},e_{i_2},e_{i_3},e_{i_4})\Vert<2b$ and all $R\in\mathscr{K}$.
The existence of $b$ is due to the compactness of $\mathscr{K}$.
Then we find $a>0$ such that for each linear isomorphism $j\colon\mathbb{R}^n\to\mathbb{R}^n$ with $|\mathfrak{b}(j(e_i),j(e_l))-\delta_{il}|<a$, there exists a linear isometry $\iota\colon\mathbb{R}^n\to\mathbb{R}^n$ such that
\begin{equation}\nonumber
|\mathfrak{b}(j(e_i),\iota(e_l))-\delta_{il}|<\frac{b}{\sqrt{n}}\quad\text{for all $1\leq i,l\leq n$}.
\end{equation}
This follows from Lemma~\ref{QGS}, but we do not require its full content at this stage.
It holds
\begin{align}\nonumber
\Vert j^\ast R-\iota^\ast R\Vert^2&=\Vert(\iota^{-1}\circ j)^\ast R-R\Vert^2\\\nonumber
&\leq\sum\limits_{i_1,i_2,i_3,i_4}|(\iota^{-1}\circ j)^\ast R(e_{i_1},e_{i_2},e_{i_3},e_{i_4})-R(e_{i_1},e_{i_2},e_{i_3},e_{i_4})|^2\\\label{anearly2}
&=\sum\limits_{i_1,i_2,i_3,i_4}|R\bigl((\iota^{-1}\circ j)e_{i_1},\dots,(\iota^{-1}\circ j)e_{i_4}\bigr)-R(e_{i_1},e_{i_2},e_{i_3},e_{i_4})|^2.
\end{align}
For all $1\leq i\leq n$, we see that
\begin{align}\nonumber
\Vert(\iota^{-1}\circ j)e_i-e_i\Vert^2&=\Vert\sum\limits_{l=1}^n(\mathfrak{b}((\iota^{-1}\circ j)e_i,e_l)-\delta_{il})\cdot e_l\Vert^2\\\nonumber
&=\sum_{l=1}^n(\mathfrak{b}(j(e_i),\iota(e_l))-\delta_{il})^2<b^2.
\end{align}
Consequently, for all $1\leq i_1,i_2,i_3,i_4\leq n$,
\begin{equation}\nonumber
\Vert\bigl((\iota^{-1}\circ j)e_{i_1},\dots,(\iota^{-1}\circ j)e_{i_4}\bigr)-\bigl(e_{i_1},e_{i_2},e_{i_3},e_{i_4}\bigr)\Vert<2b.
\end{equation}
Equations~\eqref{anearly1} and~\eqref{anearly2} yield $\Vert j^\ast R-\iota^\ast R\Vert<r$.
\end{proof}

\section{The Perelman gluing}
Throughout this section we fix a Perelman flexible family $\mathsf{C}$ of curvature conditions along with a real number $0<\tau\leq 1$ such that $\mathfrak{b}\KN\mathfrak{b}^\perp+r\cdot\mathfrak{b}\KN\mathfrak{b}+B^\perp_{\tau\sqrt{r}}\subset\mathsf{C}(0)$ for all $0<r\leq 1$.

Furthermore, let $M$ be a smooth manifold of dimension $n\geq 2$ with non-empty (possibly non-compact) boundary and let $\sigma\colon M\to\mathbb{R}$ be a continuous function.
Both of them remain unchanged until Proposition~\ref{1jet}.

The space of Riemannian metrics on $M$ is denoted by $\mathscr{R}(M)$.
It is endowed with the weak $C^{\infty}$-topology.
We say that a metric $g$ on $M$ satisfies $\mathsf{C}(\sigma)$ if $R_g(p)$ satsifies $\mathsf{C}(\sigma(p))$ for all $p\in M$.
Set
\begin{equation}\nonumber
\mathscr{R}_{\mathsf{C}(\sigma)}(M):=\{g\in\mathscr{R}(M):\text{$g$ satisfies $\mathsf{C}(\sigma)$}\},
\end{equation}
where the topology is inherited from $\mathscr{R}(M)$.

To establish the gluing principle, it is essential that the condition of satisfying $\mathsf{C}(\sigma)$ defines a second order open partial differential relation on the total space of pointwise metrics over $M$.
We set
\begin{equation}\nonumber
E:=\{(p,g_p)\in M\times(T^\ast M\otimes T^\ast M):\text{$g_p$ is an inner product on $T_pM$}\}
\end{equation}
and denote by $J^2E$ the second jet manifold of $E$.
Let $\pi_M\colon J^2E\to M$ be the source projection and $\pi_E\colon J^2E\to E$ the target projection.
To each $2$-jet $j^2_pg\in E$ one can associate an algebraic curvature tensor $R(j^2_pg)\in\mathscr{C}_B(T_pM)$.
Put
\begin{equation}\nonumber
\mathscr{P}:=\{j^2_pg\in J^2E:R(j^2_pg)\;\text{satisfies $\mathsf{C}(\sigma(p))$}\}.
\end{equation}
We will show that $\mathscr{P}$ is an open subset of $J^2E$.

\begin{lemma}\label{openPDR}
Let $\mathscr{K}\subset\mathscr{P}$ be a compact subset.
Then there exist $r>0$ and $\varepsilon>0$ such that for every $j^2_pg\in\mathscr{K}$ there exists a neighbourhood $j^2_pg\in\mathscr{N}\subset J^2E$ with
\begin{equation}\nonumber
B_r\bigl(\iota_{h(q)}^\ast R(j^2_qh)\bigr)\in\mathsf{C}(\sigma(p)+\varepsilon)
\end{equation}
for all $j^2_qh\in\mathscr{N}$ and all linear isometries $\iota_{h(q)}\colon\mathbb{R}^n\to(T_qM,h(q))$.
\end{lemma}
\begin{proof}
Let $\mathscr{A}\subset J^2E$ be a compact neighbourhood of $\mathscr{K}$.
Put $A:=\pi_M(\mathscr{A})$ and $K:=\pi_M(\mathscr{K})$.
Both are compact subsets of $M$.
Without loss of generality, we can assume that $A$ is covered by a single chart.
Otherwise, $A$ can be covered by a finite number of compact neighbourhoods which are covered by a single chart.

Using a Gram-Schmidt process, we construct a continuous family of local frames $(e_1^g,\dots,e_n^g)$ for $TM$ over $A$ so that $(e_2^{g_p}(p),\dots,e_n^{g_p}(p))$ is an orthonormal basis of $(T_pM,g_p)$ for all $p\in A$ and $g_p\in\pi_E(\mathscr{A})$.
Then define
\begin{equation}\nonumber
\vartheta(j^2_pg)(x_1,\dots,x_n):=x_1e_1^{g(p)}(p)+\dots+x_ne_n^{g(p)}(p)
\end{equation}
as a linear isometry $\mathbb{R}^n\to(T_pM,g(p))$.
This yields a continuous map
\begin{equation}\nonumber
\vartheta\colon\mathscr{A}\to C^\infty(\mathbb{R}^n;TM).
\end{equation}
We obtain a continuous map
\begin{equation}\nonumber
R^{\vartheta}:=\vartheta^\ast R\colon\mathscr{A}\to\mathscr{C}_B(\mathbb{R}^n)\,,\,R^\vartheta(j^2_pg)=\vartheta(j^2_pg)^\ast R(j^2_pg).
\end{equation}
The image of $R^\vartheta$ is a compact subset of $\mathscr{C}_B(\mathbb{R}^n)$.

Now let $p_0\in K$.
It holds $R^\vartheta(j^2_{p_0}g)\in\mathsf{C}(\sigma(p_0))$ for all $j^2_{p_0}g\in\mathscr{K}_{p_0}:=\mathscr{K}\cap(J^2E)_{p_0}$, where $(J_2E)_{p_0}:=\pi_M^{-1}(p_0)$.
The function $\varepsilon(\sigma(p_0))$ attains a minimum $\varepsilon(p_0)$ on the compact set of curvature tensors $\{R^\vartheta(j^2_{p_0}g):j^2_{p_0}g\in\mathscr{K}_{p_0}\}$.
Therefore, by continuity of the family $\mathsf{C}$,
\begin{equation}\nonumber
R^\vartheta(j^2_{p_0}g)\in\mathsf{C}(p_0):=\mathsf{C}\bigl(\sigma(p_0)+\varepsilon(p_0)\bigr)
\end{equation}
for all $j^2_{p_0}g\in\mathscr{K}_{p_0}$.
Since $\mathsf{C}(p_0)$ is open and $\mathscr{K}_{p_0}$ is compact, there exists $r(p_0)>0$ with $B_{r(p_0)}(R^\vartheta(j^2_{p_0}g))\subset\mathsf{C}(p_0)$ for all $j^2_{p_0}g\in\mathscr{K}_{p_0}$.
For each $j^2_{p_0}g\in\mathscr{K}_{p_0}$ we choose a neighbourhood $j^2_{p_0}g\in\mathscr{N}(j^2_{p_0}g)\subset\mathring{\mathscr{A}}$ such that
\begin{align}\nonumber
&|\sigma(p)-\sigma(p_0)|<\tfrac{1}{2}\varepsilon(p_0)\quad\text{and}\\\nonumber
&\Vert R^\vartheta(j^2_ph)-R^\vartheta(j^2_{p_0}g)\Vert<\tfrac{1}{2}r(p_0)
\end{align} 
for all $p\in\pi_M(\mathscr{N}(j^2_{p_0}g))$ and $j^2_ph\in\mathscr{N}(j^2_{p_0}g)\cap(J^2E)_{p}$.
Then
\begin{equation}\nonumber
B_{r(p_0)/2}\bigl(R^\vartheta(j^2_ph)\bigr)\subset\mathsf{C}\bigl(\sigma(p)+\tfrac{1}{2}\varepsilon(p_0)\bigr)
\end{equation}
holds for all such $p$ and $j^2_ph$.
The space $\mathscr{K}$ is covered by the union of neighbourhoods $\mathscr{N}(j^2_{p_0}g)$ around all $2$-jets $j^2_{p_0}g,p_0\in K$.
We pass to a finite subcover of neighbourhoods around a finite number of $2$-jets $j^2_{p_1}g_1,\dots,j^2_{p_k}g_k\in\mathscr{K}$.
Set
\begin{equation}\nonumber
r:=\min_{1\leq i\leq k}\tfrac{1}{4}r(p_i)\quad\text{and}\quad\varepsilon:=\min_{1\leq i\leq k}\tfrac{1}{2}\varepsilon(p_i).
\end{equation}
This leads us to conclude that
\begin{equation}\nonumber
B_{2r}\bigl(R^\vartheta(j^2_pg)\bigr)\subset\mathsf{C}\bigl(\sigma(p)+\varepsilon)
\end{equation}
for all $j^2_pg\in\mathscr{K}$.
Given a $2$-jet $j^2_pg\in\mathscr{K}$ there exists a neighbourhood $j^2_pg\in\mathscr{N}\subset\mathring{\mathscr{A}}$ such that $\Vert R^\vartheta(j^2_qh)-R^\vartheta(j^2_pg)\Vert<r$ for all $j^2_qh\in\mathscr{N}$.
Thus,
\begin{equation}\nonumber
B_r\bigl(R^\vartheta(j^2_qh)\bigr)\subset\mathsf{C}\bigl(\sigma(p)+\varepsilon\bigr)
\end{equation}
for all $j^2_qh\in\mathscr{N}$.

Finally we pick $j^2_pg\in\mathscr{K}$, a neighbourhood $j^2_pg\in\mathscr{N}\subset J^2E$ as above, an element $j^2_qh\in\mathscr{N}$ and an arbitrary linear isometry $\iota_{h(q)}\colon\mathbb{R}^n\to(T_qM,h(q))$.
Put
\begin{equation}\nonumber
\iota:=\vartheta(j^2_qh)^{-1}\circ\iota_{h(q)}\in\On(n).
\end{equation}
It holds
\begin{align}\nonumber
B_r\bigl(\iota_{h(q)}^\ast R(j^2_qh)\bigr)&=B_r\bigl(\iota^\ast R^\vartheta(j^2_qh)\bigr)\\\nonumber
&=\iota^\ast B_r\bigl(R^\vartheta(j^2_qh)\bigr)\subset\mathsf{C}\bigl(\sigma(p)+\varepsilon\bigr)
\end{align}
by the $\On(n)$-invariance of $\mathsf{C}(\sigma(p)+\varepsilon)$.
We also recognize that the space
\begin{equation}\nonumber
\{\iota_{g(p)}^\ast R(j^2_pg):j^2_pg\in\mathscr{K},\iota_{g(p)}\colon\mathbb{R}^n\to (T_pM,g(p))\;\text{a linear isometry}\}\subset\mathscr{C}_B(\mathbb{R}^n)
\end{equation}
is the image of the canonical action $\On(n)\times R^\vartheta(\mathscr{K})\to\mathscr{C}_B(\mathbb{R}^n)$ and therefore compact.
This will be used in Corollary~\ref{uniformcor2}.
\end{proof}
\begin{corollary}\label{Popen}
$\mathscr{P}$ is an open subset of $J^2E$.
\end{corollary}
\begin{proof}
Let $j^2_pg\in\mathscr{P}$.
We apply Lemma~\ref{openPDR} for the compact set $\mathscr{K}=\{j^2_pg\}$ and obtain $\varepsilon>0$ and a neighbourhood $j^2_pg\in\mathscr{N}\subset J^2E$ such that
\begin{equation}\nonumber
\iota_{h(q)}^\ast R(j^2_qh)\in\mathsf{C}(\sigma(p)+\varepsilon)
\end{equation}
for all $j^2_qh\in\mathscr{N}$ and all linear isometries $\iota_{h(q)}\colon\mathbb{R}^n\to(T_qM,h(q))$.
Choosing $\mathscr{N}$ smaller if necessary, we can assume that $|\sigma(q)-\sigma(p)|<\varepsilon$ for all $q\in\pi_M(\mathscr{N})$.
This yields
\begin{equation}\nonumber
\iota_{h(q)}^\ast R(j^2_qh)\in\mathsf{C}(\sigma(q))
\end{equation}
or $j^2_qh\in\mathscr{P}$ for all such $j^2_qh$.
\end{proof}
Lemma~\ref{openPDR} has favorable consequences for families of Riemannian metrics on $M$.

\begin{lemma}\label{uniform}
Let $K$ be a compact Hausdorff space. Let $g\colon K\to\mathscr{R}(M)$ be a continuous family of Riemannian metrics.
Put
\begin{equation}\nonumber
P:=\{p\in M:R_{g(\xi)}(p)\;\text{satisfies $\mathsf{C}(\sigma(p))$ for all $\xi\in K$}\}
\end{equation}
and let $A\subset P$ be a compact subset.
Then there exist $r>0$ and $\varepsilon>0$ such that for every $p\in A$ there exists a neighbourhood $p\in U\subset M$ with
\begin{equation}\nonumber
B_r\bigl(\iota_{g(\xi)(q)}^\ast R_{g(\xi)}(q)\bigr)\in\mathsf{C}(\sigma(p)+\varepsilon)
\end{equation}
for all $\xi\in K,q\in U$ and all linear isometries $\iota_{g(\xi)(q)}\colon\mathbb{R}^n\to(T_qM,g(\xi)(q))$.
\end{lemma}
\begin{proof}
We apply Lemma~\ref{openPDR} for the compact space
\begin{equation}\nonumber
\mathscr{K}=\{j_p^2g(\xi):\xi\in K,p\in A\}
\end{equation}
and obtain $\varepsilon>0$ and $r>0$ with its distinguished properties.
Let $p\in A$.
For every $\xi_0\in K$ there exists a neighbourhood $j^2_pg(\xi_0)\subset\mathscr{N}_{\xi_0}\subset J^2E$ with
\begin{equation}\nonumber
B_r\bigl(\iota_{h(q)}^\ast R(j^2_qh)\bigr)\subset\mathsf{C}(\sigma(p)+\varepsilon)
\end{equation}
for all $j^2_qh\in\mathscr{N}_{\xi_0}$ and all linear isometries $\iota_{h(q)}\colon\mathbb{R}^n\to(T_qM,h(q))$.

Since the assignment $K\times M\to J^2E,(\xi,q)\mapsto j^2_qg(\xi)$ is continuous, we find neighbourhoods $p\in U_{\xi_0}\subset M$ and $\xi_0\in F_{\xi_0}\subset K$ such that $j^2_qg(\xi)\in\mathscr{N}_{\xi_0}$ for all $q\in U_{\xi_0}$ and $\xi\in F_{\xi_0}$.
We pass to a finite subcover $F_{\xi_1},\dots,F_{\xi_k}$ of $K$ and set
\begin{equation}\nonumber
U:=U_{\xi_1}\cap\dots\cap U_{\xi_k}.
\end{equation}
This $U$ does the job for $p$.
\end{proof}
\begin{corollary}\label{uniformcor1}
$P$ is an open subset of $M$.
\end{corollary}
\begin{proof}
The proof is similar to that in Corollary~\ref{Popen}.
\end{proof}
\begin{definition}
Let $(V,g)$ be an $n$-dimensional inner product space.
Let $a>0$.
A linear isomorphism $j\colon\mathbb{R}^n\to V$ is called \textit{$a$-nearly $g$-isometric} if
\begin{equation}\nonumber
|g(j(e_i),j(e_l))-\delta_{il}|<a\quad\text{for all $1\leq i,l\leq n$}.
\end{equation}
\end{definition}

\begin{corollary}\label{uniformcor2}
Let $K$ be a compact Hausdorff space.
Let $g\colon K\to\mathscr{R}_{\mathsf{C}(\sigma)}(M)$ be a continuous family of Riemannian metrics that satisfy $\mathsf{C}(\sigma)$ and let $A\subset M$ be a compact subset.

Then there exist $\varepsilon>0$ and $a>0$ such that
\begin{equation}\nonumber
j^\ast R_{g(\xi)}(p)\in\mathsf{C}(\sigma(p)+\varepsilon)
\end{equation}
for all $\xi\in K,p\in A$ and all $a$-nearly $g(\xi)(p)$-isometric isomorphisms $j\colon\mathbb{R}^n\to T_pM$.
\end{corollary}
\begin{proof}
It holds $P=M$.
Lemma~\ref{uniform} gives $r>0$ and $\varepsilon>0$ such that
\begin{equation}\nonumber
B_r\bigl(\iota_{g(\xi)(p)}^\ast R_{g(\xi)}(p)\bigr)\subset\mathsf{C}(\sigma(p)+\varepsilon)
\end{equation}
for all $\xi\in K,p\in A$ and all linear isometries $\iota_{g(\xi)(p)}\colon\mathbb{R}^n\to(T_pM,g(\xi)(p))$.
We apply Corollary~\ref{uniforma} for $r$ and the compact space
\begin{equation}\nonumber
\mathscr{K}:=\{\iota_{g(\xi)(p)}^\ast R_{g(\xi)}(p):\xi\in K,p\in A,\iota_{g(\xi)(p)}\colon\mathbb{R}^n\to (T_pM,g(\xi)(p))\;\text{a linear isometry}\}\subset\mathscr{C}_B(\mathbb{R}^n).
\end{equation}
The Lemma provides $a>0$ which does the job:

Let $\xi\in K,p\in A$ and let $j\colon\mathbb{R}^n\to T_pM$ be an $a$-nearly $g(\xi)(p)$-isometric isomorphism.
Let $\iota_{g(\xi)(p)}\colon\mathbb{R}^n\to(T_pM,g(\xi)(p))$ be an arbitrary linear isometry.
Then
\begin{align}\nonumber
|\mathfrak{b}((\iota_{g(\xi)(p)}^{-1}\circ j)(e_i&),(\iota_{g(\xi)(p)}^{-1}\circ j)(e_l))-\delta_{il}|\\\nonumber
&=|g(\xi)(p)(j(e_i),j(e_l))-\delta_{il}|<a\quad\text{for all $1\leq i,l\leq n$}.
\end{align}
Put $R:=\iota_{g(\xi)(p)}^\ast R_{g(\xi)}(p)$ and let $\iota\colon\mathbb{R}^{n}\to\mathbb{R}^{n}$ be a linear isometry as in Corollary~\ref{uniforma}, which means that
\begin{equation}\nonumber
j^\ast R_{g(\xi)}(p)=(\iota_{g(\xi)(p)}^{-1}\circ j)^\ast R\in B_r(\iota^\ast R).
\end{equation}
Since $\iota_{g(\xi)(p)}\circ\iota\colon\mathbb{R}^n\to(T_pM,g(\xi)(p))$ is a linear isometry, we conclude that $j^\ast R_{g(\xi)}(p)\in\mathsf{C}(\sigma(p)+\varepsilon)$.
\end{proof}

\subsection{Uniformisation for geodesic collar neighbourhoods}
For each Riemannian metric $g$ on $M$ we denote by
\begin{enumerate}
\item[$\triangleright$]{$g_0\in C^{\infty}(\partial M; T^\ast\partial M\otimes T^\ast\partial M)$ the metric induced on $\partial M$;}
\item[$\triangleright$]{$\nu_g$ the inward pointing unit normal vector field along $\partial M$;}
\item[$\triangleright$]{$\mathrm{II}_g$ the second fundamental form of $\partial M\subset M$ with respect to $\nu$.}
\end{enumerate}
Given a Riemannian metric $g$ on $M$ its normal exponential map provides a collar neighbourhood
\begin{equation}\nonumber
\varrho_g\colon V\to U^g\subset M\,,\,\varrho_g(t,p)=\exp_p(t\nu_g),
\end{equation}
where $V\subset[0,\infty)\times\partial M$ is an open neighbourhood of $\{0\}\times\partial M$.
We call $\varrho_g$ a \textit{geodesic collar neighbourhood}.
More specifically, there exists a continuous positive function $\eta\colon\partial M\to\mathbb{R}$ such that $U^g=U^g_\eta$ is the diffeomorphic image of
\begin{equation}\nonumber
V_\eta:=\{(t,p)\in[0,\infty)\times\partial M:t<\eta(p)\}
\end{equation}
under the normal exponential map, i.e. $\varrho_g\colon V_\eta\to U_\eta^g$ is a geodesic collar neighbourhood.

When $\tilde{M}$ is a second smooth manifold with boundary $\partial\tilde{M}=\partial M$ and $\tilde{g}$ is a Riemannian metric on $\tilde{M}$, the geodesic collar neighbourhoods $\varrho_g$ and $\varrho_{\tilde{g}}$ induce a smooth structure on the topological attachment space $M\cup_{\partial M}\tilde{M}$.
The resulting smooth manifold will be referred to as $M\cup_{g\sqcup\tilde{g}}\tilde{M}$ when we want to emphasise the smooth structure.

The main gluing principle in Theorem~\ref{theorem1} will work in such a way that it deforms an initial metric $g\sqcup\tilde{g}$ on $M\sqcup\tilde{M}$ with $g_0=\tilde{g}_0$ into a metric $f\sqcup\tilde{f}\in\mathscr{R}(M\sqcup\tilde{M})$ with the property that $f\cup\tilde{f}$ is a smooth Riemannian metric on $M\cup_{f\sqcup\tilde{f}}\tilde{M}$.
The property of $f\cup\tilde{f}$ beeing smooth on $M\cup_{f\sqcup\tilde{f}}\tilde{M}$ can be expressed in specific formulae:
According to the generalised Gauss lemma, pulling back a metric $g\in\mathscr{R}(M)$ along $\varrho_g$ yields a generalised cylinder metric
\begin{equation}\nonumber
g=\mathrm{d}t^2+g_t\in\mathscr{R}(V_\eta)
\end{equation}
where $g_t$ is given by $g_\bullet\colon V_\eta\to T^\ast\partial M\otimes T^\ast\partial M, g_t(p)(X,Y)=(\varrho_g^\ast g)(t,p)(X,Y)$.
The family $g_\bullet$ gives rise to new families $\dot{g}_\bullet,\ddot{g}_\bullet,g_\bullet^{(\ell)}\colon V_\eta\to T^\ast\partial M\otimes T^\ast\partial M$ of $(0,2)$-tensor fields on $\partial M$ defined by
\begin{align}\nonumber
\dot{g}_t(p)(X,Y)&:=\frac{\mathrm{d}}{\mathrm{d}t}(g_t(p)(X,Y)),\\\nonumber
\ddot{g}_t(p)(X,Y)&:=\frac{\mathrm{d}^2}{\mathrm{d}t^2}(g_t(p)(X,Y)),\\\nonumber
g_t^{(\ell)}(p)(X,Y)&:=\frac{\mathrm{d}^{(\ell)}}{\mathrm{d}t^{(\ell)}}(g_t(p)(X,Y))
\end{align}
for all $(t,p)\in V$ and $X,Y\in T_p\dM$.
In the case of gluing, $f\cup\tilde{f}$ is smooth on $M\cup_{f\sqcup\tilde{f}}\tilde{M}$ iff $f_0^{(\ell)}=(-1)^\ell\cdot\tilde{f}_0^{(\ell)}$ for all $\ell\geq 0$.

We can now explain the reason for the present subsection.
Theorem~\ref{theorem1} will be formulated for \textit{continuous families} of Riemannian metrics $g\sqcup\tilde{g}\colon K\to\mathscr{R}(M\sqcup\tilde{M})$, where $K$ is a compact Hausdorff space.
Then the respective geodesic collar neighbourhoods on $M$ and $\tilde{M}$ may differ between members of such a family.
Note that different choices of collar neighbourhoods lead to the same diffeomorphism type on $M\cup_{\partial M}\tilde{M}$, see e.g.~\cite[Chap.~8, Thm.~2.1]{Hirsch}.
However, the smooth structures induced by different metrics in the same family are not the same in general.
This issue arises when investigating spaces of metrics on glued manifolds, as we will see in the proof of Theorem~\ref{whe}.
Therefore, it is necessary to uniformise the normal exponential maps.
We perform this as a preparatory step.
The process concludes with Corollary~\ref{uniformgcn}.

\begin{lemma}\label{cover1}
Let $r>0$.
There exists a complete Riemannian metric $m\in\mathscr{R}(M)$ and an open cover $(U_\alpha)_{\alpha\in I}$ of $M$ such that each $U_\alpha$ is relatively compact and that
\begin{equation}\nonumber
I_\alpha:=\{\beta\in I:\overline{N}_m(U_\alpha,r)\cap\overline{N}_m(U_\beta,r)\neq\varnothing\}
\end{equation}
is finite for every $\alpha\in I$.
Here $\overline{N}_m(U_\alpha,r):=\{p\in M:d_m(p,U_\alpha)\leq r\}$ denotes the closed $r$-neighbourhood around $U_\alpha$ with respect to $m$.
\end{lemma}
\begin{proof}
See~\cite[Lem.~3.1]{Frerichs2025}.
\end{proof}
We choose $r=1$ and fix a background metric $m\in\mathscr{R}(M)$ together with an open cover $(U_\alpha)_{\alpha\in I}$ of $M$ as in the lemma.
We would like to mention that the balls $\overline{B}_m(U_\alpha,1)$ are compact due to the Hopf-Rinow theorem, see \cite[Prop.~2.5.22 and Thm.~2.5.28]{BBI}.
This is applicable since Riemannian manifolds with boundary are length spaces, see e.g. \cite[Prop.~3.18]{BrH}.

Suppose that there is a distinguished metric $g(\xi_0),\xi_0\in K$ in a given family $g\colon K\to\mathscr{R}(M)$.
Let $I\subset\mathbb{R}$ be an open interval with $[0,1]\subset I$ and let $\rho\colon\mathbb{R}\to[0,1]$ be a smooth function with
\begin{enumerate}
\item[$\triangleright$]{$\rho(s)=0$ for all $s\leq 0$;}
\item[$\triangleright$]{$\rho(s)=1$ for all $s\geq 1$;}
\item[$\triangleright$]{$\supp\rho'\subset(0,1)$.}
\end{enumerate}

We consider the family of time-dependent Riemannian metrics given by $(1-\rho(s))g(\xi_0)+\rho(s)g(\xi)$ for $\xi\in K$ and $s\in I$.
There exists an associated continuous family of geodesic collar neighbourhoods $\varrho_{\xi,s}\colon V\to U^{\xi,s}$ which is defined on a uniform neighbourhood $V\subset[0,\infty)\times\partial M$, see~\cite[Sec.~3]{Frerichs2025}.
We drop the subscript $\eta$ at this point.
In particular, it holds $\varrho_{\xi,0}=\varrho_{g(\xi_0)}$ and $\varrho_{\xi,1}=\varrho_{g(\xi)}$ for all $\xi\in K$.

\begin{proposition}\label{diff}
Let $(K,\xi_0)$ be a compact pointed Hausdorff space and let $g\colon K\to\mathscr{R}(M)$ be a continuous family of Riemannian metrics.
Let $\delta\colon M\to\mathbb{R}$ be a continuous positive function.

For each neighbourhood $\mathscr{U}$ of $\dM$, there exists a smaller neighbourhood $\dM\subset U\subset\mathscr{U}$ and a continuous map
\begin{equation}\nonumber
\Omega\colon K\times[0,1]\to\Diff(M)
\end{equation}
such that the following holds for all $\xi\in K$ and $s\in I$:
\begin{enumerate}
\item[(a)]{$\Omega(\xi,0)=\id_M$;}
\item[(b)]{$\Omega(\xi_0,s)=\id_M$;}
\item[(c)]{$\Omega(\xi,s)=\varrho_{\xi,s}\circ\varrho_{g(\xi_0)}^{-1}$ on U, in particular $\Omega(\xi,1)=\varrho_{g(\xi)}\circ\varrho_{g(\xi_0)}^{-1}$ on $U$;}
\item[(d)]{$\Omega(\xi,s)=\id$ on $M\setminus\mathscr{U}$;}
\item[(e)]{$d_m(\Omega(\xi,s)(p),p)<\delta(p)$ for all $p\in M$;}
\end{enumerate}
\end{proposition}
\begin{proof}
See~\cite[Prop.~3.3]{Frerichs2025}.
\end{proof}
\begin{corollary}\label{uniformgcn}
Let $K$ be a compact Hausdorff space and let $g\colon K\to\mathscr{R}_{\mathsf{C}(\sigma)}(M)$ be a continuous family of Riemannian metrics that satisfy $\mathsf{C}(\sigma)$.
Let $\zeta\in\mathscr{R}(M)$ be another distinguished Riemannian metric.

For each neighbourhood $\mathscr{U}$ of $\dM$, there exists a smaller neighbourhood $\dM\subset U\subset\mathscr{U}$ and a continuous map
\begin{equation}\nonumber
f\colon K\times[0,1]\to\mathscr{R}_{\mathsf{C}(\sigma)}(M)
\end{equation}
so that the following holds for all $\xi\in K$ and $s\in[0,1]$:
\begin{enumerate}
\item[(a)]{$f(\xi,0)=g(\xi)$;}
\item[(b)]{$\nu_{f(\xi,1)}=\nu_{\zeta}$ and $\varrho_{f(\xi,1)}=\varrho_{\zeta}$ on $\varrho_{\zeta}^{-1}(U)$;}
\item[(c)]{$f(\xi,s)_t=g(\xi)_t$ near $\{0\}\times\dM$;}
\item[(d)]{$f(\xi,s)=g(\xi)$ on $M\setminus\mathscr{U}$.}
\end{enumerate}
\end{corollary}
\begin{proof}
Let $\mathscr{U}$ be a neighbourhood of $\partial M$.
Let $\xi_0$ be an arbitrary object.
Then $K\cup\{\xi_0\}$ with the topology of disjoint union is again a compact Hausdorff space.
We put $g(\xi_0):=\zeta$.

Each metric $f(\xi,s)$ will be a pullback of $g(\xi)$ along some diffeomorphism of $M$.
For the construction, we consider the open cover $(U_\alpha)_{\alpha\in I}$ from Lemma~\ref{cover1}.
Let $\alpha\in I$.
By compactness and Lemma~\ref{uniform}, there exists $\delta_\alpha>0$ such that
\begin{equation}\label{littlemovement}
\iota_{g(\xi)(q)}^\ast R_{g(\xi)}(q)\in\mathsf{C}(\sigma(p))
\end{equation}
for all $\xi\in K$, as well as $p,q\in\overline{N}_m(U_\alpha,1)$ with $d_m(p,q)<\delta_\alpha$ and all linear isometries $\iota_{g(\xi)(q)}\colon\mathbb{R}^n\to(T_qM,g(\xi)(q))$.

Let $(\psi_\alpha)_{\alpha\in I}$ be a partition of unity subordinate to $(U_\alpha)_{\alpha\in I}$.
Set
\begin{equation}\nonumber
\delta\colon M\to\mathbb{R}\,,\,\delta(p)=\frac{1}{2}\sum\limits_{\alpha\in I}\psi_\alpha(p)\cdot\min\bigl\{\min_{\beta\in I_\alpha}\delta_\beta,1\bigr\}.
\end{equation}
This is a well-defined continuous function with $0\leq\delta(p)\leq\tfrac{1}{2}$ for all $p\in M$ and $\delta(p)<\delta_\alpha$ for all $\alpha\in I$ and $p\in U_\alpha$.

Proposition~\ref{diff}, applied for the compact Hausdorff space $K\cup\{\xi_0\}$ and the function $\delta$, yields a neighbourhood $\partial M\subset U\subset\mathscr{U}$ and a family $\Omega\colon(K\cup\{\xi_0\})\times[0,1]\to\Diff(M)$ with its distinguished properties~\ref{diff}$(a)-(e)$.
We put
\begin{equation}\nonumber
f\colon K\times[0,1]\to\mathscr{R}(M)\,,\,f(\xi,s):=\Omega(\xi,s)^\ast g(\xi).
\end{equation}
This operation preserves the pointwise curvature constraint, which is shown as follows:
Let $\xi\in K,s\in[0,1],p\in M$ and let $\iota_{f(\xi,s)(p)}\colon\mathbb{R}^n\to(T_pM,f(\xi,s)(p))$ be a linear isometry.
It holds
\begin{align}\nonumber
\iota_{f(\xi,s)(p)}^\ast R_{f(\xi,s)}(p)&=\iota_{f(\xi,s)(p)}^\ast\bigl(\mathrm{d}_p\Omega(\xi,s)^\ast R_{g(\xi)}\bigl(\Omega(\xi,s)(p)\bigr)\bigr)\\\nonumber
&=\bigl(\mathrm{d}_p\Omega(\xi,s)\circ\iota_{f(\xi,s)(p)}\bigr)^\ast R_{g(\xi)}\bigl(\Omega(\xi,s)(p)\bigr).
\end{align}
Observe that $\mathrm{d}_p\Omega(\xi,s)\circ\iota_{f(\xi,s)(p)}$ is a linear isometry $\mathbb{R}^n\to(T_{\Omega(\xi,s)(p)}M,g(\xi)(\Omega(\xi,s)(p)))$.

Now let $\alpha\in I$ and $p\in U_\alpha$.
By Proposition~\ref{diff}$(e)$, we have $d_m(\Omega(\xi,s)(p),p)<\delta(p)<1$, i.e. $\Omega(\xi,s)(p)\in\overline{N}_m(U_\alpha,1)$.
The same inequality gives $d_m(\Omega(\xi,s)(p),p)<\delta(p)<\delta_\alpha$ so that
\begin{equation}\nonumber
\iota_{f(\xi,s)}^\ast R_{f(\xi,s)}(p)\in\mathsf{C}(\sigma(p))
\end{equation}
due to equation~\eqref{littlemovement}.
Hence, $f(\xi,s)$ satisfies $\mathsf{C}(\sigma)$.

It remains to check properties $(a)-(d)$.
This has been done in~\cite[Cor.~3.5]{Frerichs2025}, but we replicate it for the sake of completeness.
\begin{enumerate}
\item[$(a)$]{follows from Proposition~\ref{diff}$(a)$;}
\item[$(b)$]{Let $\xi\in K$ and $(t,p)\in\varrho_{\zeta}^{-1}(U)$.
The map $\Omega(\xi,1)\colon(M,f(\xi,1))\to(M,g(\xi))$, as an isometry, maps normal geodesics to normal geodesics.
By Proposition~\ref{diff}$(c)$ we have
\begin{align}\nonumber
\Omega(\xi,1)(\varrho_{\zeta}(t,p))=\varrho_{g(\xi)}(t,p)&=\exp_p^{g(\xi)}(t\nu_{g(\xi)}(p))\\\nonumber
&=\Omega(\xi,1)(\exp_p^{f(\xi,1)}(t\nu_{f(\xi,1)}(p)))\\\nonumber
&=\Omega(\xi,1)(\varrho_{f(\xi,1)}(t,p))
\end{align}
since $\Omega$ preserves boundary points.
Thus, $\varrho_{f(\xi,1)}(t,p)=\varrho_{\zeta}(t,p)$.
In particular,
\begin{equation}\nonumber
\nu_{f(\xi,1)}(p)=\left.\frac{\partial}{\partial t}\right|_{t=0}\varrho_{f(\xi,1)}(t,p)=\left.\frac{\partial}{\partial t}\right|_{t=0}\varrho_{\zeta}(t,p)=\nu_{\zeta}(p).
\end{equation}}
\item[$(c)$]{Let $\xi\in K$.
We choose a neighbourhood $V\subset[0,\infty)\times\partial M$ of $\dM$ such that $\varrho_{f(\xi,s)}\colon V\to U^{f(\xi,s)}$ is a geodesic collar neighbourhood with $U^{f(\xi,s)}\subset U$ for all $s\in[0,1]$.
Consider the diagram
\begin{equation}\nonumber
\begin{tikzcd}
(V,\mathrm{d}t^2+g(\xi)_t)\arrow{rr}{\varrho_{g(\xi)}}&&(U^{g(\xi)},g(\xi))\\
\\
(V,\mathrm{d}t^2+f(\xi,s)_t)\arrow{rr}{\varrho_{f(\xi,s)}}\arrow{uu}{\id}&&(U^{f(\xi,s)},f(\xi,s))\arrow{uu}{\Omega(\xi,s)}
\end{tikzcd}
\end{equation}
It commutes since $\Omega(\xi,s)$ maps normal geodesics to normal geodesics.
As the horizontal maps and the vertical map on the right are isometries, so is the identity on the left.}
\item[$(d)$]{follows from Proposition~\ref{diff}$(d)$.\qedhere}
\end{enumerate}
\end{proof}

\subsection{A two-stage gluing scheme}

The aim is to generalise Perelman's  original gluing principle for metrics of positive Ricci curvature to Perelman flexible families.
To this end, we establish a standard type of Riemannian metrics within the gluing principle.
These are called \textit{$\Lambda$-normal metrics}.
They give better control over the curvature.

For the definition, we consider a Riemannian metric $g$ on $M$.
In a geodesic collar neighbourhood $\varrho_g\colon V_\eta\to U^g$, we identify it as a generalised cylinder metric $g=\mathrm{d}t^2+g_t$.
Then it holds $\mathrm{II}_g=-\tfrac{1}{2}\dot{g}_0$ as a $(0,2)$-tensor field on $\partial M$, see~\cite[Prop.~4.1]{BGM}.
The associated Weingarten map of $\partial M$ is denoted by $W_g$.

More generally, for each $(t,p)\in V_\eta$, we find a neighbourhood $p\in V_p\subset\partial M$ such that $\{t\}\times V_p\subset V_\eta$.
The second fundamental form for this slice is $\mathrm{II}_t=-\tfrac{1}{2}\dot{g}_t$ (with respect to the normal $\nu=\tfrac{\partial}{\partial t}$).
They build a map $\mathrm{II}_\bullet\colon V_\eta\to T^\ast\partial M\otimes T^\ast\partial M$.
So do the corresponding Weingarten maps $W_\bullet\colon V_\eta\to T^\ast\partial M\otimes T(\partial M)$, uniquely defined by
\begin{equation}\nonumber
\langle W_t(p)(X),Y\rangle_{g_t}=\mathrm{II}_t(p)(X,Y)=-\frac{1}{2}\dot{g}_t(p)(X,Y).
\end{equation}

There are three kinds of entries of $R_{g}=R_{\hspace{1pt}\mathrm{d}t^2+g_t}$ within $U^g$, depending on how often the unit normal $\nu$ appears.
Let $(t,p)\in V_\eta$ and $X,Y,Z,W\in T_p(\partial M)$.
By \cite[Prop.~4.1]{BGM}, it holds
\begin{align}\label{0nu}
&R_{g}(t,p)(X,Y,Z,W)=R_{g_t}(X,Y,Z,W)-\frac{1}{2}(\mathrm{II}_t\KN\mathrm{II}_t)(X,Y,Z,W),\\\label{1nu}
&R_{g}(t,p)(X,Y,Z,\nu)=(\nabla_X^{g_t}\mathrm{II}_t)(Y,Z)-(\nabla_Y^{g_t}\mathrm{II}_t)(X,Z),\\\label{2nu}
&R_{g}(t,p)(X,\nu,\nu,Y)=-\frac{1}{2}\ddot{g}_t(X,Y)+\mathrm{II}_t(W_t(X),Y).
\end{align}
Note that equation~\eqref{0nu} is nothing but the Gauss equation and that~\eqref{1nu} is the Codazzi-Mainardi equation.

\begin{definition}
Let $\Lambda\in C^\infty(\partial M;\mathbb{R})$.
A Riemannian metric $g\in\mathscr{R}(M)$ is called \textit{$\Lambda$-normal} if the pullback metrics $g_\bullet$ with respect to some small geodesic collar neighbourhood $\varrho_g\colon V_\eta\to U^g$ are given by
\begin{align}\nonumber
g_t(p)&=g_0(p)+t\cdot\dot{g}_0(p)-\Lambda(p)t^2\cdot g_0(p)\\\nonumber
&=g_0(p)-2t\cdot\mathrm{II}_g(p)-\Lambda(p)t^2\cdot g_0(p),\;(t,p)\in V_\eta.
\end{align}
\end{definition}

Every Riemannian metric on $M$ can be deformed into a $\Lambda$-normal metric while preserving a pointwise curvature constraint, with the original $1$-jet along the boundary.
To prove this, we first introduce a preparatory lemma.

\begin{lemma}\label{cover2}
There exist countable covers $\left(U_\alpha^1\right)_{\alpha\in\mathbb{N}}\subset\left(U_\alpha^2\right)_{\alpha\in\mathbb{N}}\subset\left(U_\alpha^3\right)_{\alpha\in\mathbb{N}}$ of $\dM$ such that the following holds for all $\alpha\in\mathbb{N}$:
\begin{enumerate}
\item[\myicon]{Each subset $U_\alpha^1,U_\alpha^2,U_\alpha^3\subset\dM$ is open and relatively compact;}
\item[\myicon]{$I_\alpha:=\{\beta\in\mathbb{N}:U_\alpha^3\cap U_\beta^3\neq\varnothing\}$ is finite;}
\item[\myicon]{$\overline{U_\alpha^1}\subset U_\alpha^2$ and $\overline{U_\alpha^2}\subset U_\alpha^3$.}
\end{enumerate}
\end{lemma}
\begin{proof}
See~\cite[Lem.~3.14]{Frerichs2025}.
\end{proof}

\begin{proposition}\label{2jet}
Let $K$ be a compact Hausdorff space and let $g\colon K\to \mathscr{R}_{\mathsf{C}(\sigma)}(M)$ be a continuous family of Riemannian metrics that satsify $\mathsf{C}(\sigma)$.

Then there exists a smooth function $\Lambda_0\in C^\infty(\partial M;\mathbb{R})$ such that for each function $\Lambda\in C^\infty(\partial M;\mathbb{R})$ with $\Lambda\geq\Lambda_0$ and each neighbourhood $\mathscr{U}$ of $\partial M$ there exists a continuous map
\begin{equation}\nonumber
f\colon K\times[0,1]\to\mathscr{R}_{\mathsf{C}(\sigma)}(M)
\end{equation}
such that the following holds for all $\xi\in K$ and $s\in[0,1]$:
\begin{enumerate}
\item[(a)]{$f(\xi,0)=g(\xi)$;}
\item[(b)]{$f(\xi,1)$ is $\Lambda$-normal;}
\item[(c)]{if $g(\xi)$ is $\tilde{\Lambda}$-normal, then $f(\xi,s)$ is $((1-s)\tilde{\Lambda}+s\Lambda)$-normal;}
\item[(d)]{$f(\xi,s)_0=g(\xi)_0$ and $\mathrm{II}_{f(\xi,s)}=\mathrm{II}_{g(\xi)}$;}
\item[(e)]{$\ddot{f}(\xi,s)_0=(1-s)\ddot{g}(\xi)_0-2s\Lambda g(\xi)_0$;}
\item[(f)]{$f(\xi,s)^{(\ell)}_0=(1-s)g(\xi)^{(\ell)}_0$ for all $\ell\geq 3$;}
\item[(g)]{$f(\xi,s)=g(\xi)$ on $M\setminus\mathscr{U}$.}
\end{enumerate}
\end{proposition}
\begin{proof}
Based on Corollary~\ref{uniformgcn}, it can be assumed that all metrics $g(\xi),\xi\in K$ have the same normal exponential map.

Let $\mathscr{U}$ be a neighbourhood of $\partial M$ and let $\varrho\colon V_\eta\to U$ be a common geodesic collar neighbourhood for all metrics in $g$ satisfying $U\subset\mathscr{U}$.
One can identify the metrics $g(\xi),\xi\in K$ on $U$ with its associated generalised cylinder metrics on $V_\eta$.
We will drop the subscript $\eta$ from now on.

Let $\left(U_\alpha\right)_{\alpha\in\mathbb{N}}=\left(U_\alpha^3\right)_{\alpha\in\mathbb{N}}$ be a cover of $\partial M$ as in Lemma~\ref{cover2}, and let $(\psi_\alpha)_{\alpha\in\mathbb{N}}$ be a partition of unity subordinate to $\left(U_\alpha\right)_{\alpha\in\mathbb{N}}$.

For each $\alpha\in\mathbb{N}$, we choose $0<\kappa_\alpha\leq 1$ sufficiently small so that $\iota_{g(\xi)(p)}^\ast R_{g(\xi)}(p)-s\kappa_\alpha\cdot\mathfrak{b}\KN\mathfrak{b}\in\mathsf{C}(\sigma(p))$ for all $\xi\in K,s\in[0,1],p\in U_\alpha$ and all linear isometries $\iota_{g(\xi)(p)}\colon\mathbb{R}^{n}\to(T_pM,g(\xi))$.
We can do this by virtue of Corollary~\ref{uniformcor2}.
Then set
\begin{equation}\nonumber
\Lambda_\alpha:=\frac{1}{\kappa_\alpha}\cdot\left(\frac{n-1}{\tau}\right)^2\cdot\sup\left\{\left|\tfrac{1}{2}\ddot{g}(\xi)_0(X,Y)\right|^2:\xi\in K, p\in U_\alpha\;\;\text{and}\;\;|X|_{g(\xi)_0(p)}=|Y|_{g_0(\xi)(p)}=1\right\}
\end{equation}
and 
\begin{equation}\nonumber
\Lambda_0\colon\partial M\to\mathbb{R}\,,\,\Lambda_0:=1+\sum\limits_{\alpha\in\mathbb{N}}\psi_\alpha\cdot\max_{\beta\in I_\alpha}\Lambda_\beta.
\end{equation}
Observe that $\Lambda_0\geq 1$ and $\Lambda_0(p)>\Lambda_\alpha$ when $p\in U_\alpha$.

Now let $\Lambda\colon\partial M\to\mathbb{R}$ be a smooth function with $\Lambda\geq\Lambda_0$.
We consider the Taylor expansion of $g_\bullet$ at $t=0$:
\begin{equation}\nonumber
g(\xi)_t=g(\xi)_0+\dot{g}(\xi)_0\cdot t+\frac{1}{2}\ddot{g}(\xi)_0\cdot t^2+R(\xi)_t.
\end{equation}
Here $R(\xi)_\bullet$ is a smooth map $R(\xi)_\bullet\colon V\to T^\ast\partial M\otimes T^\ast\partial M$ made out of symmetric $(0,2)$-tensors on $\partial M$ that depends continuously on $\xi$.
It holds
\begin{equation}\nonumber
R(\xi)_0=\dot{R}(\xi)_0=\ddot{R}(\xi)_0=0.
\end{equation}
Put
\begin{equation}\nonumber
F(\xi,s):=g(\xi)-s\bigl(\bigl(\tfrac{1}{2}\ddot{g}(\xi)_0+\Lambda\cdot g(\xi)_0\bigr)\cdot t^2+R(\xi)_t\bigr)
\end{equation}
for $\xi\in K$ and $s\in[0,1]$.
This defines a continuous family $F\colon K\to C^{\infty}(V;T^\ast V\otimes T^\ast V)$ of symmetric $(0,2)$-tensor fields.
It can also be seen as a family of $(0,2)$-tensor fields on $U$.

Shrinking $U$ if necessary, we can assume that $F(\xi,s)$ is a Riemannian metric on $U$ for each $\xi\in K$ and $s\in[0,1]$.
For $\xi\in K$ and $s\in[0,1]$, we consider the difference curvature tensor
\begin{equation}\nonumber
T:=R_{F(\xi,s)}-R_{g(\xi)}
\end{equation}
on $V$.
Let $\alpha\in\mathbb{N},p\in U_\alpha$ and $X,Y,Z,W\in T_p(\partial M)$ with unit $g(\xi)_0(p)$-length.
At the point $(0,p)\in V$, it holds
\begin{align}\nonumber
R_{F(\xi,s)}(X,Y,Z,W)&=R_{F(\xi,s)_0}(X,Y,Y,X)-\frac{1}{2}(\mathrm{II}_{F(\xi,s)}\KN\mathrm{II}_{F(\xi,s)})(X,Y,Z,W)\\\nonumber
&=R_{g(\xi)_0}(X,Y,Y,X)-\frac{1}{2}(\mathrm{II}_{g(\xi)}\KN\mathrm{II}_{g(\xi)})(X,Y,Z,W)=R_{g(\xi)}(X,Y,Z,W)
\end{align}
or
\begin{equation}\nonumber
T(X,Y,Z,W)=0.
\end{equation}
Similarly, we observe
\begin{align}\nonumber
R_{F(\xi,s)}(X,Y,Z,\nu)&=(\nabla^{F(\xi,s)_0}_X\mathrm{II}_{F(\xi,s)})(Y,Z)-(\nabla^{F(\xi,s)_0}_Y\mathrm{II}_{F(\xi,s)})(X,Z)\\\nonumber
&=(\nabla^{g(\xi)_0}_X\mathrm{II}_{g(\xi)})(Y,Z)-(\nabla^{g(\xi)_0}_Y\mathrm{II}_{g(\xi)})(X,Z)=R_{g(\xi)}(X,Y,Z,\nu),
\end{align}
respectively
\begin{equation}\nonumber
T(X,Y,Z,\nu)=0.
\end{equation}
Finally,
\begin{align}\nonumber
R_{F(\xi,s)}(X,\nu,\nu,Y)&=-\tfrac{1}{2}\ddot{F}(\xi,s)_0(X,Y)+\mathrm{II}_{F(\xi,s)}(W_{F(\xi,s)}(X),Y)\\\nonumber
&=-\tfrac{1}{2}\ddot{g}(\xi)_0(X,Y)+\mathrm{II}_{g(\xi)}(W_{g(\xi)}(X),Y)+s\bigl(\tfrac{1}{2}\ddot{g}(\xi)_0(X,Y)+\Lambda\cdot g(\xi)_0(X,Y)\bigr)\\\nonumber
&=R_{g(\xi)}(X,\nu,\nu,Y)+s\bigl(\tfrac{1}{2}\ddot{g}(\xi)_0(X,Y)+\Lambda\cdot g(\xi)_0(X,Y)\bigr),\end{align}
meaning that
\begin{equation}\nonumber
T(X,\nu,\nu,Y)=s\bigl(\tfrac{1}{2}\ddot{g}(\xi)_0(X,Y)+\Lambda\cdot g(\xi)_0(X,Y)\bigr).
\end{equation}
By Perelman flexibility,
\begin{equation}\nonumber
\mathfrak{b}\KN\mathfrak{b}^\perp+\kappa_\alpha\Lambda(p)^{-1}\cdot\mathfrak{b}\KN\mathfrak{b}+B^\perp_{\tau\sqrt{\kappa_\alpha\Lambda(p)^{-1}}}\subset\mathsf{C}(0).
\end{equation}
Thus,
\begin{equation}\nonumber
s\Lambda(p)\cdot\mathfrak{b}\KN\mathfrak{b}^\perp+s\kappa_\alpha\cdot\mathfrak{b}\KN\mathfrak{b}+B^\perp_{s\tau\sqrt{\kappa_\alpha\Lambda(p)}}\subset\overline{\mathsf{C}(0)}
\end{equation}
for all $s\in[0,1]$.
Let $\iota_{F(\xi,s)_0(p)}\colon\mathbb{R}^{n-1}\to(T_p(\partial M),F(\xi,s)_0(p)=g(\xi)_0(p))$ be a linear isometry.
It holds
\begin{align}\nonumber
\Vert(\id_t\oplus\iota_{F(\xi,s)_0(p)})^\ast T-s\Lambda(p)\cdot\mathfrak{b}\KN\mathfrak{b}^\perp\Vert^2=\sum\limits_{i,l=2}^n&|s\cdot\tfrac{1}{2}\ddot{g}(\xi)_0(p)(\iota_{F(\xi,s)_0(p)}(e_i),\iota_{F(\xi,s)_0(p)}(e_l))|^2\\\nonumber
&<s^2\tau^2\cdot\kappa_\alpha\Lambda_0(p).
\end{align}
Therefore, $(\id_t\oplus\iota_{F(\xi,s)_0(p)})^\ast T+s\kappa_\alpha\cdot\mathfrak{b}\KN\mathfrak{b}\in\overline{\mathsf{C}(0)}$.
We conclude that
\begin{align}\nonumber
(\id_t\oplus\iota_{F(\xi,s)_0(p)})^\ast R_{F(\xi,s)}&=(\id_t\oplus\iota_{F(\xi,s)_0(p)})^\ast R_{g(\xi)}-s\kappa_\alpha\cdot\mathfrak{b}\KN\mathfrak{b}+\bigl((\id_t\oplus\iota_{F(\xi,s)_0(p)})^\ast T+s\kappa_\alpha\cdot\mathfrak{b}\KN\mathfrak{b}\bigr)\\\nonumber
&\phantom{=}\in\mathsf{C}(\sigma(p))+\overline{\mathsf{C}(0)}\subset\mathsf{C}(\sigma(p)).
\end{align}
Thus, $F(\xi,s)$ satisfies $\mathsf{C}(\sigma)$ along $\partial M$.
Corollary~\ref{uniformcor1} allows to assume that $F(\xi,s)\in\mathscr{R}_{\mathsf{C}(\sigma)}(U)$ for all $\xi\in K$ and $s\in[0,1]$.

Along the boundary, we find
\begin{enumerate}
\item[$\triangleright$]{$F(\xi,s)|_{\partial M}=g(\xi)|_{\partial M}$;}
\item[$\triangleright$]{$\dot{F}(\xi,s)_0=\dot{g}(\xi)_0$;}
\item[$\triangleright$]{$\ddot{F}(\xi,s)_0=(1-s)\ddot{g}(\xi)_0-2s\Lambda g(\xi)_0$;}
\item[$\triangleright$]{$F(\xi,s)_0^{(\ell)}=(1-s)g(\xi)_0^{(\ell)}$ for all $\ell\geq 3$.}
\end{enumerate}
By Corollary~\ref{Popen}, the condition of satisfying $\mathsf{C}(\sigma)$ defines a second order open partial differential relation on the total space of pointwise metrics over $M$.
The metrics in $g$ solve it globally, the metrics in $F$ solve it locally over $U$.

By the family version of the local flexibility lemma \cite[Addendum~3.4]{BH2022} of B\"ar-Hanke, we obtain an open neighbourhood $\partial M\subset U_0\subset U$ and a family
\begin{equation}\nonumber
f\colon K\times[0,1]\to C^{\infty}(M;T^\ast M\otimes T^\ast M)
\end{equation}
such that the following holds for all $\xi\in K$ and $s\in[0,1]$:
\begin{enumerate}
\item[$\triangleright$]{$f(\xi,s)\in\mathscr{R}_{\mathsf{C}(\sigma)}(M)$;}
\item[$\triangleright$]{$f(\xi,0)=g(\xi)$;}
\item[$\triangleright$]{$f(\xi,s)|_{U_0}=F(\xi,s)|_{U_0}$;}
\item[$\triangleright$]{$f(\xi,s)|_{M\setminus U}=F(\xi,s)|_{M\setminus U}$.}
\end{enumerate}
In particular, $f(\xi,s)|_{M\setminus\mathscr{U}}=g(\xi)_{M\setminus\mathscr{U}}$.
We conclude properties~$(a)-(g)$ for this $f$.
\end{proof}

\begin{remark}\label{samegeodesics}
The identification $g=\mathrm{d}t^2+g_t$ refers to a geodesic collar neighbourhood $\varrho_g\colon V_\eta\to U^g$ of $g$.
The deformations $F(\xi,s)$ are defined in terms of the collars $\varrho_{g(\xi)}$.
As soon as $F(\xi,s)$ is a Riemannian metric on $U=U^{g(\xi)}$ with its own geodesic collars, one may ask for the meaning of $F(\xi,s)_t$.
In fact, there is no subtlety because the normal geodesics of $F(\xi,s)$ and $g(\xi)$ coincide on $U$.
\end{remark}

\begin{remark}
Strictly speaking, Bär-Hanke's local flexibility lemma applies to smooth manifolds without boundary.
To obtain an admissible setting, we attach a small cylinder to $\dM$ and extend the metrics $g(\xi)$ to the new manifold -- called $M^+$ -- in continuous dependence on $\xi$.

Specifically, there is a continuous map $g^+\colon K\to\mathscr{R}(M^+)$ such that the canonical inclusion $(M,g(\xi))\to(M^+,g^+(\xi))$ is an isometric embedding for every $\xi\in K$.
We can guarantee continuity for $g^+$ by means of the Seeley extension theorem \cite{Seeley}.
See \cite[Prop.~2.8]{Frerichs2022} for details.
\end{remark}

The second deformation step is for adjusting the $1$-jet along the boundary, again respecting pointwise curvature constraints.
Several lemmas are required.

Thanks to Corollary~\ref{uniformgcn}, we can assume that all metrics involved in the next lemma have the same normal exponential map $\varrho$.
We will write $\sigma(t,p)$ for $(\sigma\circ\varrho)(t,p)$.
\begin{lemma}\label{uniform epsilon 2}
Let $K$ be a compact Hausdorff space.
Let $g_0\colon K\to C^{\infty}(\partial M;T^\ast\partial M\otimes T^\ast\partial M)$ be a continuous family of Riemannian metrics on $\partial M$ and let $h\colon K\to C^{\infty}(\partial M;T^\ast\partial M\otimes T^\ast\partial M)$ be a continuous familiy of symmetric $(0,2)$-tensor fields.
Let $U\subset\partial M$ be an open subset and let $A\subset U$ be compact.

Then there exists a smooth function $\Lambda_0=\Lambda_0(g_0,h)\in C^\infty(U;\mathbb{R}),\Lambda_0>0$, a constant $0<\rho\leq 1$, a real number $\delta_0>0$ and $\varepsilon>0$ such that
\begin{enumerate}
\item[$\triangleright$]{for every smooth function $\Lambda\in C^\infty(U;\mathbb{R})$ with $\Lambda\geq\Lambda_0$ and}
\item[$\triangleright$]{for every continuous family $g\colon K\to\mathscr{R}_{\mathsf{C}(\sigma)}(M)$ with
\begin{equation}\nonumber
g(\xi)_t=(1-\Lambda t^2)\cdot g_0(\xi)-2th(\xi)\quad\text{on $[0,\sqrt{\delta}]\times U$}
\end{equation}
for some $0<\delta<\min\Bigl\{\delta_0,\Vert\Lambda\Vert_{C^2(A)}^{-4}\Bigr\}$ and all $\xi\in K$,}
\end{enumerate}
it holds
\begin{equation}\nonumber
(\id_t\oplus j)^\ast R_{g(\xi)}(t,p)+B^\perp_{\rho\sqrt{\Lambda(p)}}\subset\mathsf{C}(\sigma(t,p)+\varepsilon)
\end{equation}
for all $\xi\in K,(t,p)\in[0,\sqrt{\delta}]\times A$ and all $\Lambda(p)^{-1}$-nearly $g(\xi)_t(p)$-isometric isomorphisms $j\colon\mathbb{R}^{n-1}\to T_p(\partial M)$.
\end{lemma}
There are two key observations in the lemma: Firstly, we see that $\varepsilon$ is independent of $\Lambda$, provided that $\Lambda$ is large enough.
Secondly, as $\Lambda$ increases, the pulled back curvature tensors are surrounded by arbitrarily large perpendicular discs which still fit into $\mathsf{C}(\sigma+\varepsilon)$.

In the lemma, we make use of the $C^2$-norm of functions and tensor fields on $A$, which is supposed to be the maximum of $C^2$-norms induced by the metrics $g_0(\xi),\xi\in K$.

\begin{proof}
We can assume that there exists a family $\varphi\colon K\to\mathscr{R}_{\mathsf{C}(\sigma)}(M)$ such that
\begin{equation}\nonumber
\varphi(\xi)_t=(1-\Lambda_\varphi t^2)\cdot g_0(\xi)-2th(\xi)\quad\text{on $[0,\sqrt{\delta_\varphi}]\times U$}
\end{equation}
for some positive smooth function $\Lambda_\varphi\colon U\to\mathbb{R}$ and some $\delta_\varphi>0$.
Otherwise, the statement is trivial.

By Corollary~\ref{uniformcor2}, applied for the manifold $[0,\sqrt{\delta_\varphi}]\times U$ and the compact subset $[0,\sqrt{\delta_\varphi}]\times A$, there exist $\varepsilon_\varphi>0$ and $a_\varphi>0$ such that
\begin{equation}\nonumber
(\id_t\oplus j)^\ast R_{\varphi(\xi)}(t,p)\in\mathsf{C}(\sigma(t,p)+\varepsilon_\varphi)
\end{equation}
for all $\xi\in K,(t,p)\in[0,\sqrt{\delta_\varphi}]\times A$ and all $a_\varphi$-nearly $\varphi(\xi)_t(p)$-isometric isomorphisms $j\colon\mathbb{R}^{n-1}\to T_p(\partial M)$.
Even more, we choose $0<\kappa_\varphi\leq 1$ so small that $(\id_t\oplus j)^\ast R_{\varphi(\xi)}-\kappa_\varphi\cdot\mathfrak{b}\KN\mathfrak{b}\in\mathsf{C}(\sigma+\varepsilon_\varphi)$ for all such data.

Let $\Lambda\colon U\to\mathbb{R}$ be a smooth function with $\Lambda>\tfrac{3}{2}+\Lambda_\varphi$ and let $g\colon K\to\mathscr{R}_{\mathsf{C}(\sigma)}(M)$ be a continuous family with
\begin{equation}\nonumber
g(\xi)_t=(1-\Lambda t^2)\cdot g_0(\xi)-2th(\xi)\quad\text{on $[0,\sqrt{\delta}]\times U$}
\end{equation}
for some $0<\delta<\min\{\delta_\varphi,\Vert\Lambda_\varphi\Vert_{C^2(A)}^{-4},\Vert\Lambda\Vert_{C^2(A)}^{-4}\}$.
The second fundamental form of $g(\xi)$ at distance $t$ from the boundary is given by
\begin{equation}\label{sff}
\mathrm{II}_t=\Lambda t\cdot g_0(\xi)+h(\xi).
\end{equation}

We will analyse the difference between $R_{g(\xi)}$ and $R_{\varphi(\xi)}$.
For our purpose, it is convenient not just to separate the different contributions as in equations~\eqref{0nu}--\eqref{2nu}, but to repartition the difference curvature tensor $R_{g(\xi)}-R_{\varphi(\xi)}$ in a slightly different way.

To this end, we define three curvature tensors $S^\varepsilon, S^{\lesssim 1}$ and $S^\Lambda$ on $[0,\sqrt{\delta}]\times A$ as follows:

\begin{alignat}{3}\nonumber
&\mathrm{(I)}&&\text{$S^\varepsilon$ is defined via}&&\\\nonumber
&&&&&S^\varepsilon(X,Y,Z,W)=R_{g(\xi)}(X,Y,Z,W)-R_{\varphi(\xi)}(X,Y,Z,W),\\\nonumber
&&&&&S^\varepsilon(X,Y,Z,\nu)=R_{g(\xi)}(X,Y,Z,\nu)-R_{\varphi(\xi)}(X,Y,Z,\nu),\\\nonumber
&&&&&S^\varepsilon(X,\nu,\nu,Y)=0.\\\nonumber
&\mathrm{(II)}&&\text{$S^{\lesssim 1}$ is defined via}&&\\\nonumber
&&&&&S^{\lesssim 1}(X,Y,Z,W)=0,\\\nonumber
&&&&&S^{\lesssim 1}(X,Y,Z,\nu)=0,\\\nonumber
&&&&&S^{\lesssim 1}(X,\nu,\nu,Y)=\mathrm{II}_t^{g(\xi)}(W_t^{g(\xi)}(X),Y)-\mathrm{II}_t^{\varphi(\xi)}(W_t^{\varphi(\xi)}(X),Y).\\\nonumber
&\mathrm{(III)}\;\;&&\text{$S^{\Lambda}$ is defined via}&&\\\nonumber
&&&&&S^{\Lambda}(X,Y,Z,W)=0,\\\nonumber
&&&&&S^{\Lambda}(X,Y,Z,\nu)=0,\\\nonumber
&&&&&S^{\Lambda}(X,\nu,\nu,Y)=-\frac{1}{2}(\ddot{g}(\xi)_t(X,Y)-\ddot{\varphi}(\xi)_t(X,Y)).
\end{alignat}
Here $X,Y,Z,W$ are tangent vectors in $T(\partial M)$.
To discuss $S^\varepsilon$, we first compute
\begin{align}\nonumber
R_{g(\xi)}(X,Y,Z,W)=R_{g(\xi)_t}(X,Y,&\,Z,W)-\frac{1}{2}\Lambda^2t^2\cdot(g_0(\xi)\KN g_0(\xi))(X,Y,Z,W)\\\nonumber
&-\Lambda t\cdot(g_0(\xi)\KN h(\xi))(X,Y,Z,W)-\frac{1}{2}(h(\xi)\KN h(\xi))(X,Y,Z,W),
\end{align}
using~\eqref{0nu} and~\eqref{sff}.
A similar formula holds true for $R_{\varphi(\xi)}$.
This yields
\begin{align}\nonumber
|S^\varepsilon(X,Y,Z,W)|&\leq|(R_{g(\xi)_t}-R_{\varphi(\xi)_t})(X,Y,Z,W)|+\frac{1}{2}(\Lambda^2-\Lambda_\varphi^2)t^2\cdot|(g_0(\xi)\KN g_0(\xi))(X,Y,Z,W)|\\\nonumber
&\phantom{\leq}+(\Lambda-\Lambda_\varphi)t\cdot|(g_0(\xi)\KN h(\xi))(X,Y,Z,W)|\\\nonumber
&\leq|(R_{g_1(\xi)_t}-R_{\varphi_1(\xi)_t})(X,Y,Z,W)|+\frac{1}{2}\Lambda^{-2}\cdot|(g_0(\xi)\KN g_0(\xi))(X,Y,Z,W)|\\\label{Stangential}
&\phantom{\leq}+\Lambda^{-1}\cdot|(g_0(\xi)\KN h(\xi))(X,Y,Z,W)|.
\end{align}
Next we know from equation~\eqref{1nu} that
\begin{equation}\nonumber
R_{g(\xi)}(X,Y,Z,\nu)=\bigl(\nabla^{g(\xi)_t}_X\mathrm{II}_t\bigr)(Y,Z)-\bigl(\nabla_Y^{g(\xi)_t}\mathrm{II}_t\bigr)(X,Z).
\end{equation}
The second fundamental form can be written as
\begin{align}\nonumber
\mathrm{II}_t=\Lambda t\cdot g_0(\xi)+h(\xi)&=\Lambda t\cdot(1-\Lambda t^2)^{-1}\cdot(g(\xi)_t+2t\cdot h(\xi))+h(\xi)\\\label{sffforSnu}
&=\frac{\Lambda t}{1-\Lambda t^2}\cdot g(\xi)_t+\frac{1+\Lambda t^2}{1-\Lambda t^2}\cdot h(\xi).
\end{align}
Since $g(\xi)_t$ is parallel for $\nabla^{g(\xi)_t}$, this yields
\begin{align}\nonumber
(\nabla_X^{g(\xi)_t}\mathrm{II}_t)(Y,Z)&=\frac{\mathrm{d}_X\Lambda\cdot t}{(1-\Lambda t^2)^2}\cdot g(\xi)_t(Y,Z)+\frac{2\,\mathrm{d}_X\Lambda\cdot t^2}{(1-\Lambda t^2)^2}\cdot h(\xi)(Y,Z)+\frac{1+\Lambda t^2}{1-\Lambda t^2}\cdot\bigl(\nabla_X^{g(\xi)_t}h(\xi)\bigr)(Y,Z)\\\nonumber
&=\frac{\mathrm{d}_X\Lambda\cdot t}{1-\Lambda t^2}\cdot g_0(\xi)(Y,Z)+\frac{1+\Lambda t^2}{1-\Lambda t^2}\cdot\bigl(\nabla_X^{g(\xi)_t}h(\xi)\bigr)(Y,Z)
\end{align}
or
\begin{align}\nonumber
R_{g(\xi)}(X,Y,Z,\nu)&=\frac{1}{1-\Lambda t^2}\cdot\bigl(\mathrm{d}_X\Lambda\cdot t\cdot g_0(\xi)(Y,Z)-\mathrm{d}_Y\Lambda\cdot t\cdot g_0(\xi)(X,Z)\bigr)\\\nonumber
&\phantom{=}+\frac{1+\Lambda t^2}{1-\Lambda t^2}\cdot\bigl(\bigl(\nabla^{g(\xi)_t}_Xh(\xi)\bigr)(Y,Z)-\bigl(\nabla^{g(\xi)_t}_Yh(\xi)\bigr)(X,Z)\bigr).
\end{align}
A similar formula holds true for $R_{\varphi(\xi)}$, so that
\begin{align}\nonumber
S^\varepsilon(X,Y,Z,\nu)&=\frac{1}{1-\Lambda t^2}\cdot\bigl(\mathrm{d}_X\Lambda\cdot t\cdot g_0(\xi)(Y,Z)-\mathrm{d}_Y\Lambda\cdot t\cdot g_0(\xi)(X,Z)\bigr)\\\nonumber
&\phantom{=}-\frac{1}{1-\Lambda_\varphi t^2}\cdot\bigl(\mathrm{d}_X\Lambda_\varphi\cdot t\cdot g_0(\xi)(Y,Z)-\mathrm{d}_Y\Lambda_\varphi\cdot t\cdot g_0(\xi)(X,Z)\bigr)\\\nonumber
&\phantom{=}+\frac{1+\Lambda t^2}{1-\Lambda t^2}\cdot\bigl(\bigl(\nabla^{g(\xi)_t}_Xh(\xi)\bigr)(Y,Z)-\bigl(\nabla^{g(\xi)_t}_Yh(\xi)\bigr)(X,Z)\bigr)\\\nonumber
&\phantom{=}-\frac{1+\Lambda_\varphi t^2}{1-\Lambda_\varphi t^2}\cdot\bigl(\bigl(\nabla^{\varphi(\xi)_t}_Xh(\xi)\bigr)(Y,Z)-\bigl(\nabla^{\varphi(\xi)_t}_Yh(\xi)\bigr)(X,Z)\bigr).
\end{align}

According to Remark~\ref{Basicremark}, there exists $r>0$ such that $B_r(\mathscr{C}_B(\mathbb{R}^n))\subset\mathsf{C}(-\varepsilon_\varphi/2)$.
By choosing $\Lambda$ large enough, we want to ensure that $(\id_t\oplus j)^\ast S^\varepsilon\in B_r(\mathscr{C}_B(\mathbb{R}^n))$ for all $\xi\in K,(t,p)\in[0,\sqrt{\delta}]\times A$ and all $a_1$-nearly $g(\xi)_t(p)$-isometric isomorphisms $j\colon\mathbb{R}^{n-1}\to T_p(\partial M)$, where $a_1$ is to be determined.

For this purpose, we fix a number $d>1$.
By compactness, it holds
\begin{align}\nonumber
\frac{1}{2}\Lambda^{-2}\cdot|(g_0(\xi)\KN g_0(\xi))(X,Y,Z,W)|&<\frac{1}{3}r/(n-1)^2\quad\text{and}\\\nonumber
\frac{1}{2}\Lambda^{-1}\cdot|(g_0(\xi)\KN h(\xi))(X,Y,Z,W)|&<\frac{1}{3}r/(n-1)^2
\end{align}
for all $\xi\in K,(t,p)\in[0,\sqrt{\delta}]\times A$ and $X,Y,Z,W\in T_p(\partial M)$ with $g_0(\xi)(p)$-norm $\leq 2d$, if $\Lambda$ is chosen sufficiently large.
Furthermore, we can arrange that the first summand in~\eqref{Stangential} is smaller than $\tfrac{1}{3}r/(n-1)^2$ for all these data, too, since
\begin{equation}\nonumber
\Vert g(\xi)_t-\varphi(\xi)_t\Vert_{C^2(A)}=\Vert \Lambda-\Lambda_\varphi\Vert_{C^2(A)}\cdot t^2\cdot\Vert g_0(\xi)\Vert_{C^2(A)}\leq2\Vert\Lambda\Vert_{C^0(A)}^{-3}\cdot\Vert g_0(\xi)\Vert_{C^2(A)}
\end{equation}
becomes arbitrarily small for sufficiently large $\Lambda$.
Note that choosing $\Lambda$ sufficiently large should imply that $\Lambda$ can be chosen independently of $\xi,t,p$ and of specific tangential vectors in $T(\partial M)$ whose $g_0(\xi)$-norm is $\leq 2d$ for some $\xi\in K$.
This convention is used throughout the remainder of the proof.

For $S^\varepsilon(X,Y,Z,\nu)$ we observe
\begin{align}\nonumber
|S^\varepsilon(X,Y,Z,\nu)|\lesssim\Lambda^{-1}+\left|\bigl(\nabla^{g(\xi)_t}_Xh(\xi)-\nabla^{\varphi(\xi)_t}_Xh(\xi)\bigr)(Y,Z)\right|+\left|\bigl(\nabla^{g(\xi)_t}_Yh(\xi)-\nabla^{\varphi(\xi)_t}_Yh(\xi)\bigr)(X,Z)\right|.
\end{align}
The notation $\lesssim$ means that the left hand side is bounded by the right hand side multiplied with a positive constant that only depends on $\varphi,g_0,h$ and $d$.
In particular, it is independent of $\xi,t,p$ and $\Lambda$ and of specific tangential vectors in $T(\partial M)$ whose $g_0(\xi)$-norm is $\leq 2d$ for some $\xi\in K$.

The same argument as for $S^\varepsilon(X,Y,Z,W)$ shows that $|S^\varepsilon(X,Y,Z,\nu)|<r/(n-1)^2$ for sufficiently large $\Lambda$.
Overall, for sufficiently large $\Lambda$, it holds
\begin{equation}\nonumber
|S^\varepsilon(X,Y,Z,W)|<r/(n-1)^2\quad\text{and}\quad|S^\varepsilon(X,Y,Z,\nu)|<r/(n-1)^2
\end{equation}
for all $\xi\in K,(t,p)\in[0,\sqrt{\delta}]\times A$ and $X,Y,Z,W\in T_p(\partial M)$ with $g_0(\xi)(p)$-norm $\leq 2d$.

Now we put $a_1:=d^2-1$ and let $j\colon\mathbb{R}^{n-1}\to T_p(\partial M)$ be an $a_1$-nearly $g(\xi)_t(p)$-isometric isomorphism for some $\xi\in K$ and $(t,p)\in[0,\sqrt{\delta}]\times A$.
By definition, $g(\xi)_t(p)(j(e_i),j(e_i))\leq d^2$ for all $2\leq i\leq n$.
For $X\in T_p(\partial M)$ with $g_t(\xi)(p)$-norm $\leq d$, we obtain
\begin{align}\nonumber
d^2\geq g_1(\xi)_t(X,X)&=(1-\Lambda t^2)\cdot g_0(\xi)(X,X)-2t\cdot h(\xi)(X,X)\\\label{g_1g_0}
&=(1-\Lambda t^2)\cdot g_0(\xi)(X,X)-2t\cdot h(\xi)\left(\frac{X}{|X|_{g_0(\xi)}},\frac{X}{|X|_{g_0(\xi)}}\right)\cdot g_0(\xi)(X,X)
\end{align}
or
\begin{equation}\nonumber
g_0(\xi)(X,X)\leq d^2\cdot\left(1-\Lambda t^2-2t\cdot h(\xi)\left(\frac{X}{|X|_{g_0(\xi)}},\frac{X}{|X|_{g_0(\xi)}}\right)\right)^{-1}
\end{equation}
for sufficiently large $\Lambda$.
Hence, by choosing $\Lambda$ large enough, we can assume that $g_0(\xi)(X,X)\leq 4d^2$.
In particular, $g_0(\xi)(p)(j(e_i),j(e_i))\leq 4d^2$ for all $2\leq i\leq n$.
This leads us to conclude that
\begin{align}\nonumber
\Vert(\id_t\oplus j)^\ast S^\varepsilon\Vert^2&\leq\sum\limits_{i_1,i_2,i_3,i_4=1}^n|(\id_t\oplus j)^\ast S^\varepsilon(e_{i_1},e_{i_2},e_{i_3},e_{i_4})|^2\\\nonumber
&<\sum\limits_{i_1,i_2,i_3,i_4=2}^nr^2/(n-1)^4=r^2.
\end{align}
Thus, $(\id_t\oplus j)^\ast S^\varepsilon\in\mathsf{C}(-\varepsilon_\varphi/2)$.

In the next step, we consider $S^{\lesssim 1}$:
Let $\xi\in K,(t,p)\in[0,\sqrt{\delta}]\times A$ and let $(v_2,\dots,v_n)$ be a $g(\xi)_t(p)$-orthonormal basis of $T_p(\partial M)$.
Then we have
\begin{align}\nonumber
W_t^{g(\xi)}(X)=\sum_{i=2}^{n}g(\xi)_t(W_t^{g(\xi)}(X),v_i)\cdot v_i&=\sum_{i=2}^n\mathrm{II}_t^{g(\xi)}(X,v_i)\cdot v_i\\\label{Weingartennorm1}
&=\sum_{i=2}^n(\Lambda t\cdot g_0(\xi)(X,v_i)+h(\xi)(X,v_i))\cdot v_i.
\end{align}
This gives
\begin{align}\nonumber
|\mathrm{II}_t^{g(\xi)}(W_t^{g(\xi)}(X),Y)|^{1/2}&=|g(\xi)_t(W_t^{g(\xi)}(X),W_t^{g(\xi)}(Y))|^{1/2}\leq|W_t^{g(\xi)}(X)|_{g(\xi)_t}\cdot|W_t^{g(\xi)}(Y)|_{g(\xi)_t}\\\nonumber
&\leq\sum_{i,l=2}^n|\Lambda t\cdot g_0(\xi)(X,v_i)+h(\xi)(X,v_i)|\cdot|\Lambda t\cdot g_0(\xi)(Y,v_l)+h(\xi)(Y,v_l)|\\\label{Weingartennorm2}
&\lesssim\sum_{i,l=2}^n(\Lambda^{-1}+1)\cdot(\Lambda^{-1}+1)\lesssim 1
\end{align}
for $X,Y\in T_p(\partial M)$ with $g_0(\xi)(p)$-norm $\leq 2d$.
A similar estimate is valid for $\varphi$.
Therefore, it holds
\begin{equation}\nonumber
\Vert(\id_t\oplus j)^\ast S^{\lesssim 1}\Vert\lesssim 1
\end{equation}
for all $a_1$-nearly $g(\xi)_t(p)$-isometric isomorphisms $j\colon\mathbb{R}^{n-1}\to T_p(\partial M)$.

Lastly, we note that the second derivative in the definition of $S^\Lambda$ is given by
\begin{equation}\nonumber
-\frac{1}{2}(\ddot{g}(\xi)_t(X,Y)-\ddot{\varphi}(\xi)_t(X,Y))=(\Lambda-\Lambda_\varphi)\cdot g_0(\xi)(X,Y).
\end{equation}
This is enough to draw the conclusion of the lemma:
Let $p\in A$.
By Perelman flexibility,
\begin{equation}\label{PFE}
(\Lambda-\Lambda_\varphi)(p)\cdot\mathfrak{b}\KN\mathfrak{b}^\perp+\kappa_\varphi\cdot\mathfrak{b}\KN\mathfrak{b}+B^\perp_{\tau\sqrt{\kappa_\varphi(\Lambda-\Lambda_\varphi)(p)}}\subset\mathsf{C}(0).
\end{equation}
We choose $\Lambda$ large enough and $0<a_2\leq a_1$ small enough so that each $a_2\sqrt{\Lambda(p)^{-1}}$-nearly $g(\xi)_t(p)$-isometric isomorphism $j\colon\mathbb{R}^{n-1}\to T_p(\partial M)$ is $\tfrac{1}{10}\tau\sqrt{\kappa_\varphi\Lambda(p)^{-1}}\cdot(n-1)^{-1}$-nearly $g_0(\xi)(p)$-isometric.
This is possible since
\begin{align}\nonumber
|g_0(\xi)(j(e_i),j(e_l))-\delta_{il}|&\leq|(1-\Lambda t^2)\cdot g_0(\xi)(j(e_i),j(e_l))-\delta_{il}|+\Lambda t^2\cdot|g_0(\xi)(j(e_i),j(e_l))|\\\nonumber
&\leq|g(\xi)_t(j(e_i),j(e_l))-\delta_{il}|+2t\cdot|h(\xi)(j(e_i),j(e_l))|+\Lambda t^2\cdot|g_0(\xi)(j(e_i),j(e_l))|\\\nonumber
&\lesssim a_2\sqrt{\Lambda^{-1}}+\Lambda^{-2}+\Lambda^{-3}\lesssim(a_2+\Lambda^{-3/2})\cdot\sqrt{\Lambda^{-1}}.
\end{align}
For such $j$, we obtain
\begin{align}\nonumber
|(\id_t\oplus j)^\ast S^\Lambda(e_i,e_1,e_1,e_l)-\delta_{il}\cdot(\Lambda-\Lambda_\varphi)(p)|&=(\Lambda-\Lambda_\varphi)(p)\cdot|g_0(\xi)(j(e_i),j(e_l))-\delta_{il}|\\\nonumber
&<\tfrac{1}{10}\tau\sqrt{\kappa_\varphi\Lambda(p)}\cdot(n-1)^{-1}
\end{align}
for all $2\leq i,l\leq n$.
Thus,
\begin{align}\nonumber
\Vert(\id_t\oplus j)^\ast S^\Lambda-(\Lambda-\Lambda_\varphi)(p)\cdot\mathfrak{b}\KN\mathfrak{b}^\perp\Vert^2&\leq\sum_{i,l=2}^{n}|(\id_t\oplus j)^\ast S^\Lambda(e_{i},e_1,e_1,e_{l})-\delta_{il}\cdot(\Lambda-\Lambda_\varphi)(p)|^2\\\nonumber
&<\sum_{i,l=2}^{n}\left(\tfrac{1}{10}\tau\sqrt{\kappa_\varphi\Lambda(p)}\right)^2\cdot(n-1)^{-2}=\left(\tfrac{1}{10}\tau\sqrt{\kappa_\varphi\Lambda(p)}\right)^2.
\end{align}
When choosing $\Lambda$ large enough, it holds $\tfrac{1}{10}\sqrt{\Lambda}\leq\tfrac{2}{10}\sqrt{\Lambda-\Lambda_\varphi}$.
This means
\begin{equation}\nonumber
(\id_t\oplus j)^\ast S^\Lambda+\kappa_\varphi\cdot\mathfrak{b}\KN\mathfrak{b}+B^\perp_{\tfrac{8}{10}\tau\sqrt{\kappa_\varphi(\Lambda-\Lambda_\varphi)(p)}}\subset\mathsf{C}(0).
\end{equation}
We also choose $\Lambda$ large enough so that $\Vert(\id_t\oplus j)^\ast S^{\lesssim 1}\Vert\leq\frac{1}{10}\tau\sqrt{\kappa_\varphi(\Lambda-\Lambda_\varphi)}$ and $\tfrac{1}{2}\sqrt{\Lambda}\leq\frac{7}{10}\sqrt{\Lambda-\Lambda_\varphi}$.
Then
\begin{equation}\nonumber
(\id_t\oplus j)^\ast (S^{\lesssim 1}+S^{\Lambda})+\kappa_\varphi\cdot\mathfrak{b}\KN\mathfrak{b}+B^\perp_{\tfrac{\tau}{2}\sqrt{\kappa_\varphi\Lambda(p)}}\subset\mathsf{C}(0).
\end{equation}
Finally, there exists $0<a_3\leq a_2$ such that for sufficiently large $\Lambda$, each $a_3$-nearly $g(\xi)_t(p)$-isometric isomorphism $j\colon\mathbb{R}^{n-1}\to T_p(\partial M)$ is $a_\varphi$-nearly $\varphi(\xi)_t(p)$-isometric.
This follows from the formula
\begin{equation}\nonumber
\varphi(\xi)_t=g(\xi)_t+(\Lambda-\Lambda_\varphi)t^2\cdot g_0(\xi).
\end{equation}
Given a $j$ which is $a_3\sqrt{\Lambda(p)^{-1}}$-nearly $g(\xi)_t(p)$-isometric, the overall conclusion is
\begin{align}\nonumber
(\id_t\oplus j)^\ast R_{g(\xi)}+B^\perp_{\tfrac{\tau}{2}\sqrt{\kappa_\varphi\Lambda(p)}}&=\bigl((\id_t\oplus j)^\ast R_{\varphi(\xi)}-\kappa_\varphi\cdot\mathfrak{b}\KN\mathfrak{b}\bigr)+(\id_t\oplus j)^\ast S^\varepsilon\\\nonumber
&+(\id_t\oplus j)^\ast(S^{\lesssim 1}+S^\Lambda)+\kappa_\varphi\cdot\mathfrak{b}\KN\mathfrak{b}+B^\perp_{\tfrac{\tau}{2}\sqrt{\kappa_\varphi\Lambda(p)}}\\\nonumber
&\phantom{=}\subset\mathsf{C}(\sigma+\varepsilon_\varphi)+\mathsf{C}(-\varepsilon_\varphi/2)+\mathsf{C}(0)\subset\mathsf{C}(\sigma+\varepsilon_\varphi/2).
\end{align}
We choose $\Lambda$ large enough so that $\sqrt{\Lambda^{-1}}\leq a_3$.

The assertion holds with $\rho:=\tfrac{\tau}{2}\sqrt{\kappa_\varphi},\delta_0:=\min\{\delta_\varphi,\Vert\Lambda_\varphi\Vert_{C^2(A)}^{-4}\}$ and $\varepsilon:=\varepsilon_\varphi/2$.
\end{proof}

\begin{lemma}\label{bumpfunction}
There exists a constant $c_0>0$ such that for each $0<\delta\leq\frac{1}{2}$ there exists a smooth function $\chi_\delta\colon[0,\infty)\to\mathbb{R}$ with
\begin{enumerate}
\item[\myicon]{$\chi_\delta(t)=t$ for $t$ near $0$, $\chi_\delta(t)=0$ for $t\geq\sqrt{\delta}$ and $0\leq\chi_\delta(t)\leq\frac{\delta}{2}$ for all $t$,}
\item[\myicon]{$0\leq\dot{\chi}_\delta(t)\leq 1$ for all $t\in[0,\delta]$ and $|\dot{\chi}_\delta(t)|\leq\sqrt{\delta}c_0$ for all $t\in[\delta,\sqrt{\delta}]$,}
\item[\myicon]{$-\frac{2}{\delta}\leq\ddot{\chi}_\delta(t)\leq 0$ for all $t\in[0,\delta]$ and $|\ddot{\chi}_\delta(t)|\leq c_0$ for all $t\in[\delta,\sqrt{\delta}]$.}
\end{enumerate}
\end{lemma}
\begin{proof}
See \cite[Lem.~3.5]{BH2023}.
\end{proof}
Given $0<\delta\leq\tfrac{1}{2}$ and a function $\chi_\delta$ as in Lemma~\ref{bumpfunction}, it holds $\chi_\delta(t)\leq t$ for all $t\in[0,\infty)$, see \cite[Rmk.~3.13]{Frerichs2025}.

We fix three covers $\left(U_\alpha^1\right)_{\alpha\in\mathbb{N}}\subset\left(U_\alpha^2\right)_{\alpha\in\mathbb{N}}\subset\left(U_\alpha^3\right)_{\alpha\in\mathbb{N}}$ of $\partial M$ as in Lemma~\ref{cover2} together with a partition of unity $\psi=\left(\psi_\alpha\right)_{\alpha\in\mathbb{N}}$ subordinate to $\left(U_\alpha^1\right)_{\alpha\in\mathbb{N}}$.
\begin{lemma}\label{closednbh}
Let $g$ be a Riemannian metric on $M$ and let $\varrho_{g}\colon V_{\eta}\to U^{g}_\eta$ be a geodesic collar neighbourhood.
For each $\alpha\in\mathbb{N}$, let $0<w_\alpha<\min\{\alpha^{-1},\inf_{p\in U_\alpha^3}\eta(p)\}$ and $\blacksquare_{w_\alpha}^{1}:=\varrho_g([0,w_\alpha]\times\overline{U_\alpha^1})$.
Then $\cup_{\alpha\in\mathbb{N}}\,\blacksquare_{w_\alpha}^{1}\subset M$ is a closed neighbourhoood of $\partial M$.
\end{lemma}
\begin{proof}
See~\cite[Lem.~3.15]{Frerichs2025}.
\end{proof}

For the remainder of this section, we fix two smooth manifolds $M_1,M_2$ of dimension $n\geq 2$ with non-empty boundary $\partial M_1=\partial M_2=:\partial M$.
We also fix a continuous function $\sigma\colon M_1\sqcup M_2\to\mathbb{R}$ with $\sigma|_{\partial M_1}=\sigma|_{\partial M_2}$.

\begin{proposition}\label{1jet}
Let $K$ be a compact Hausdorff space.
Let $g_0\colon K\to C^{\infty}(\partial M;T^\ast\partial M\otimes T^\ast\partial M)$ be a continuous family of Riemannian metrics on $\partial M$ and let $h,k\colon C^{\infty}(\partial M;T^\ast\partial M\otimes T^\ast\partial M)$ be continuous families of symmetric $(0,2)$-tensor fields with
\begin{align}\nonumber
&\bigl(\iota_{g_0}^\ast(h+k)\bigr)\KN\bigl(\iota_{g_0}^\ast(h+k)\bigr)\in\overline{\mathsf{C}(0)}\quad\text{and}\\\nonumber
&\bigl(\iota_{g_0}^\ast(h+k)\bigr)\KN\mathfrak{b}^\perp\in\overline{\mathsf{C}(0)},
\end{align}
where $\iota_{g_0(\xi)(p)}\colon\mathbb{R}^{n-1}\to(T_p(\partial M),g_0(\xi)(p))$ is a linear isometry for $\xi\in K$ and $p\in\partial M$.

Then there exists a smooth function $\Lambda_0=\Lambda_0(g_0,h,k)\in C^\infty(\partial M;\mathbb{R}),\Lambda_0>0$ such that
\begin{enumerate}
\item[$\triangleright$]{for every pair of continuous families
\begin{align}\nonumber
&g_1\colon K\to\mathscr{R}_{\mathsf{C}(\sigma)}(M_1),\\\nonumber
&g_2\colon K\to\mathscr{R}_{\mathsf{C}(\sigma)}(M_2)
\end{align}
of $\Lambda$-normal metrics with $\Lambda\in C^\infty(\partial M;\mathbb{R}),\Lambda\geq\Lambda_0,g_1(\xi)_0=g_2(\xi)_0=g_0(\xi)$ and $\mathrm{II}_{g_1(\xi)}=h(\xi),\mathrm{II}_{g_2(\xi)}=k(\xi)$ for all $\xi\in K$ and}
\item[$\triangleright$]{for each neighbourhood $\mathscr{U}\subset M_1\sqcup M_2$ of $\partial M_1\sqcup \partial M_2$}
\end{enumerate}
there exists a continuous map
\begin{equation}\nonumber
f\colon K\times[0,1]\to\mathscr{R}_{\mathsf{C}(\sigma)}(M_1\sqcup M_2)
\end{equation}
such that the following holds for all $\xi\in K$ and $s\in[0,1]$:
\begin{enumerate}
\item[(a)]{$f(\xi,0)=g_1(\xi)\sqcup g_2(\xi)$;}
\item[(b)]{$f(\xi,s)$ is $\Lambda$-normal;}
\item[(c)]{$f(\xi,s)_0=g_0(\xi)\sqcup g_0(\xi)$;}
\item[(d)]{$\mathrm{II}_{f(\xi,s)}=((1-s)h(\xi)-sk(\xi))\sqcup k(\xi)$;}
\item[(e)]{$f(\xi,s)=g_1(\xi)\sqcup g_2(\xi)$ on $M_1\sqcup M_2\setminus\mathscr{U}$.}
\end{enumerate}
\end{proposition}

\begin{proof}
Let $\mathscr{U}\subset M_1\sqcup M_2$ be a neighbourhood of $\partial M_1\sqcup\partial M_2$.
Let $\Lambda\colon\partial M\to\mathbb{R}$ be a positive smooth function and let $g_1\colon K\to\mathscr{R}_{\mathsf{C}(\sigma)}(M_1)$ and $g_2\colon K\to\mathscr{R}_{\mathsf{C}(\sigma)}(M_2)$ be two continuous families of $\Lambda$-normal metrics with $g_1(\xi)_0=g_2(\xi)_0=g_0(\xi)$ and $\mathrm{II}_{g_1(\xi)}=h(\xi),\mathrm{II}_{g_2(\xi)}=k(\xi)$.

Given that this proposition follows Corollary~\ref{uniformgcn} and Proposition~\ref{2jet} within the deformation process, we can assume that all metrics in $g_1$ have the same normal exponential map and that they satisfy the condition of $\Lambda$-normality on a common neighbourhood $U_0\subset M_1$ of $\partial M_1$.
The corresponding statement applies to $(M_2,g_2)$.

We choose a continuous positive function $\eta\colon\partial M\to\mathbb{R}$ such that $\varrho_1\colon V_\eta\to U^{g_1}$ is a common geodesic collar neighbourhood for $g_1$ with $U^{g_1}\subset U_0\cap\mathscr{U}$.
For each $\alpha\in\mathbb{N}$, let $0<\eta_\alpha<\inf_{p\in U_\alpha^3}\eta(p)$ be some fixed distance from the boundary.
Without loss of generality, we also obtain a geodesic collar neighbourhood $\varrho_2\colon V_\eta\to U^{g_2}$ for $g_2$ such that $U^{g_2}\subset\mathscr{U}$ and that all metrics in $g_2$ satisfy the condition of $\Lambda$-normality on $V_\eta$.

For every $\alpha\in\mathbb{N}$, we apply Lemma~\ref{uniform epsilon 2} for both manifolds $M_1$ and $M_2$, for $U=U_\alpha^3$ and $A=\overline{U_\alpha^2}$.
The lemma provides a smooth function $\Lambda_{0,\alpha}\in C^{\infty}(U_\alpha^3;\mathbb{R}),\Lambda_{0,\alpha}>0$, a constant $0<\rho_\alpha\leq 1$, a real number $0<\delta_{0,\alpha}<\eta_\alpha^2$ and $\varepsilon_\alpha>0$ with its distinguished properties (working for both $M_1$ and $M_2$).
This means
\begin{align}\nonumber
&(\id_t\oplus j)^\ast R_{g_1(\xi)}(t,p)+B^\perp_{\rho_\alpha\sqrt{\Lambda(p)}}\subset\mathsf{C}\bigl((\sigma\circ\varrho_1)(t,p)+\varepsilon_\alpha\bigr),\\\label{Lambdadisc}
&(\id_t\oplus j)^\ast R_{g_2(\xi)}(t,p)+B^\perp_{\rho_\alpha\sqrt{\Lambda(p)}}\subset\mathsf{C}\bigl((\sigma\circ\varrho_2)(t,p)+\varepsilon_\alpha\bigr)
\end{align}
for all $\alpha\in\mathbb{N}$ and $\xi\in K$, as well as for
\begin{enumerate}
\item[$\triangleright$]{all $(t,p)\in[0,\sqrt{\delta}]\times\overline{U_\alpha^2}$, where $0<\delta<\min\{\delta_{0,\alpha},\Vert\Lambda\Vert_{C^2(U_\alpha^3)}^{-4}\}$, and}
\item[$\triangleright$]{all linear isomorphisms $j\colon\mathbb{R}^{n-1}\to T_p(\partial M)$ which are both $\Lambda(p)^{-1}$-nearly $g_1(\xi)_t(p)$-isometric and $\Lambda(p)^{-1}$-nearly $g_2(\xi)_t(p)$-isometric,}
\end{enumerate}
provided that $\Lambda|_{U_\alpha^3}\geq\Lambda_{0,\alpha}$.

In the proof, we write $\sigma(t,p):=(\sigma\circ\varrho_1)(t,p)$ for $(t,p)\in V_\eta$.
Since $\sigma\circ\varrho_1=\sigma\circ\varrho_2$ on $\partial M$,
one can assume that $(\id_t\oplus j)^\ast R_{g_2(\xi)}+B^\perp_{\rho_\alpha\sqrt{\Lambda}}\subset\mathsf{C}(\sigma(t,p)+\varepsilon_\alpha)$ for all data as above, choosing $\delta_{0,\alpha}$ and $\varepsilon_\alpha$ eventually smaller.
Note that the potential decrease in $\varepsilon_\alpha$ does not depend on $\delta_{0,\alpha}$, but only on $\sigma$.

The deformation will be performed over each piece $U_\alpha^2\subset\partial M$, gluing everything together with the partition of unity $\psi$.
To this end, let $\delta=\left(\delta_\alpha\right)_{\alpha\in\mathbb{N}}$ be a family of real numbers satisfying
\begin{equation}\nonumber
0<\delta_\alpha<\min\left\{\frac{1}{2},\alpha^{-2},\min_{\beta\in I_\alpha}\delta_{0,\beta},\min_{\beta\in I_\alpha}\bigl(\Lambda_\beta^{-4}\bigr)\right\}\quad\text{with}\quad \Lambda_\alpha:=\Vert\Lambda\Vert_{C^2(U_\alpha^3)}
\end{equation}
for all $\alpha\in\mathbb{N}$.
For each $\alpha\in\mathbb{N}$, let $\chi_{\delta_\alpha}$ be a function as in Lemma~\ref{bumpfunction}.
We define
\begin{align}\nonumber
f^{\delta_\alpha}\colon&K\times[0,1]\to C^{\infty}(\square^2_{\eta_\alpha};T^\ast M_1\otimes T^\ast M_1)\\\nonumber
&f^{\delta_\alpha}(\xi,s)=\mathrm{d}t^2+(1-\Lambda t^2)\cdot g_0(\xi)-2t\cdot h(\xi)+2s\chi_{\delta_\alpha}(t)\cdot(h(\xi)+k(\xi))
\end{align}
where $\square_{\eta_\alpha}^2:=\varrho_1([0,\eta_\alpha)\times U_\alpha^2)$.
Given $(t,p)\in[0,\eta_\alpha)\times U_\alpha^2$, it holds
\begin{align}\nonumber
\Vert f^{\delta_\alpha}(\xi,s)(t,p)-g_1(\xi)(t,p)\Vert_{g_1(\xi)(t,p)}&=2s\chi_{\delta_\alpha}(t)\cdot\Vert h(\xi)(p)+k(\xi)(p)\Vert_{g_1(\xi)_t(p)}\\\nonumber
&\leq\delta_\alpha\cdot\sup_{\substack{\xi\in K,t\in[0,\eta_\alpha]\\p\in U_\alpha^2}}\Vert h(\xi)(p)+k(\xi)(p)\Vert_{g_1(\xi)_t(p)}.
\end{align}
Since positive definiteness is an open condition, $f^{\delta_\alpha}$ becomes a family of Riemannian metrics on $\square_{\eta_{\alpha}}^2$ for sufficiently small $\delta_\alpha$.
Then put
\begin{align}\nonumber
f^\delta\colon&K\times[0,1]\to C^{\infty}(M;T^\ast M\otimes T^\ast M),\quad f^\delta(\xi,s)=g_2(\xi)\;\;\text{on $M_2$,}\\\nonumber
&f^\delta(\xi,s)=\left\{\begin{array}{l}\mathrm{d}t^2+(1-\Lambda t^2)\cdot g_0(\xi)-2t\cdot h(\xi)+\sum\limits_{\alpha\in\mathbb{N}}\psi_\alpha\cdot 2s\chi_{\delta_\alpha}(t)\cdot(h(\xi)+k(\xi))\\
\phantom{g_1(\xi)}\hspace{0.6cm}\text{on $\bigcup\limits_{\alpha\in\mathbb{N}}\square_{\eta_\alpha}^2$,}\\
         g_1(\xi)\hspace{0.6cm}\text{on $M_1-\bigcup\limits_{i\in\mathbb{N}}\square_{\eta_\alpha}^2$}\end{array}\right.
\end{align}
for $\delta=\left(\delta_\alpha\right)_{\alpha\in\mathbb{N}}$ as small as above.
This expression is positive definite at each point $(t,p)\in\bigcup_{\alpha\in\mathbb{N}}\square_{\eta_\alpha}^2$, as can be seen from the formula
\begin{equation}\label{fullformula}
f^\delta(\xi,s)=\sum\limits_{\beta\in\mathbb{N}}\psi_\beta\cdot(\mathrm{d}t^2+(1-\Lambda t^2)\cdot g_0(\xi)-2t\cdot h(\xi)+2s\chi_{\delta_\beta}(t)\cdot(h(\xi)+k(\xi))).
\end{equation}
If $p\in U_\beta^2$ for some $\beta\in\mathbb{N}$ and $t<\eta_\beta$, then the respective summand is positive definite by the above.
If $t\geq\eta_\beta$, then $\chi_{\delta_\beta}(t)=0$ and the other three summands are equal to $g_1(\xi)(t,p)$.
There are no more summands to take into account because $\psi_\beta(p)=0$ for all $\beta\in\mathbb{N}$ with $p\notin U_\beta^2$.

We also verify the smoothness of $f^\delta(\xi,s)$ and the continuity of $f^\delta$ with respect to $(\xi,s)$.
This is obvious once we know
\begin{equation}\nonumber
f^\delta(\xi,s)=g_1(\xi)\quad\text{on $M_1-\bigcup\limits_{\alpha\in\mathbb{N}}\blacksquare_{\sqrt{\delta_\alpha}}^1$,}
\end{equation}
the latter being open in $M_1$ by Lemma~\ref{closednbh}.
To justify, let $(t,p)\in\bigcup_{\alpha\in\mathbb{N}}\square_{\eta_\alpha}^2-\bigcup_{\alpha\in\mathbb{N}}\blacksquare_{\sqrt{\delta_\alpha}}^1$.
That is $(t,p)\in\square_{\eta_\alpha}^2$ for some $\alpha\in\mathbb{N}$.
Every summand in~\eqref{fullformula} is either $\psi_\beta(p)\cdot g_1(\xi)(t,p)$ (if $p\in U_\beta^1$) or zero (if $p\notin U_\beta^1$).
Therefore, $f^\delta(\xi,s)(t,p)=\sum_{\beta\in\mathbb{N}}\psi_\beta(p)\cdot g_1(\xi)(t,p)=g_1(\xi)(t,p)$.

In the following, we need to ensure that $f^\delta(\xi,s)$ satisfies $\mathsf{C}(\sigma)$ for all $\xi\in K$ and $s\in[0,1]$.
Let $\alpha\in\mathbb{N}$ and $d>1$.
We will use the notation $\lesssim$ to mean that the left hand side is bounded by the right hand side multiplied with a positive constant that only depends on $g_0,h$ and $k$ on $U_\alpha^3$ and on $d$.
In particular, it is independent of $\delta,\xi,s,t,p$ and $\Lambda$ and of specific tangential vectors in $T(\partial M)$ whose $g_0(\xi)$-norm is $\leq 2d$ for some $\xi\in K$.

Furthermore, a statement is said to be satisfied for sufficiently small $\delta_{I_\alpha}=\left(\delta_\beta\right)_{\beta\in I_\alpha}$ if it is satisfied for all tuples $\left(\delta_\beta\right)_{\beta\in I_\alpha}$ whose maximum is smaller than a certain constant.
This constant depends only on $g_0,h$ and $k$ on $U_\alpha^3$, on $\Lambda$ on $\bigcup_{\beta\in I_\alpha}U_\beta^3$ and on $d$.

On $U_\alpha^3$ we use the abbreviation
\begin{align}\nonumber
&\gamma:=f^\delta(\xi,s)|_{M_1}\quad\text{and}\\\nonumber
&\gamma_t:=f^\delta(\xi,s)_t|_{\partial M_1}=(1-\Lambda t^2)\cdot g_0(\xi)-2t\cdot h(\xi)+\sum\limits_{\beta\in I_\alpha}\psi_\beta\cdot 2s\chi_{\delta_\beta}(t)\cdot(h(\xi)+k(\xi))
\end{align}
for $t\in[0,\eta_\alpha)$.
It holds
\begin{align}\label{sff2}
&\mathrm{II}_t^\gamma=-\frac{1}{2}\dot{\gamma}_t=h(\xi)+\Lambda t\cdot g_0(\xi)-\sum_{\beta\in I_\alpha}\psi_\beta\cdot s\dot{\chi}_\delta(t)\cdot(h(\xi)+k(\xi)),\\\nonumber
&\ddot{\gamma}_t=-2\Lambda\cdot g_0(\xi)+\sum_{\beta\in I_\alpha}\psi_\beta\cdot 2s\ddot{\chi}_\delta(t)\cdot(h(\xi)+k(\xi)).
\end{align}
We investigate $\gamma$ on the set $[0,\eta_\alpha)\times U_\alpha^2$.
One can even restrict to $[0,\max_{\beta\in I_\alpha}\sqrt{\delta_\beta})\times U_\alpha^2$, since $\gamma$ is equal to $g_1(\xi)$ on $[\max_{\beta\in I_\alpha}\sqrt{\delta_\beta},\eta_\alpha)\times U_\alpha$.

As a preparation step, we compare the $\gamma_0$- and the $\gamma_t$-norm for tangential vectors.
Let $\xi\in K,s\in[0,1]$, let $(t,p)\in[0,\max_{\beta\in I_\alpha}\sqrt{\delta_\beta})\times U_\alpha^2$ and let $X\in T_p(\partial M)$ with $|X|_{\gamma_t}\leq d$.
Then, it holds
\begin{align}\nonumber
d^2\geq\gamma_t(X,X)&=(1-\Lambda t^2)\cdot\gamma_0(X,X)-2t\cdot h(\xi)(X,X)+\sum\limits_{\beta\in I_\alpha}\psi_\beta\cdot 2s\chi_\delta(t)\cdot(h(\xi)+k(\xi))(X,X)\\\nonumber
&=(1-\Lambda t^2)\cdot\gamma_0(X,X)-2t\cdot h(\xi)\left(\frac{X}{|X|_{\gamma_0}},\frac{X}{|X|_{\gamma_0}}\right)\cdot\gamma_0(X,X)\\\nonumber
&\phantom{=(}+\sum\limits_{\beta\in I_\alpha}\psi_\beta\cdot 2s\chi_\delta(t)\cdot(h(\xi)+k(\xi))\left(\frac{X}{|X|_{\gamma_0}},\frac{X}{|X|_{\gamma_0}}\right)\cdot\gamma_0(X,X),
\end{align}
respectively
\begin{align}\nonumber
\gamma_0(X,X)\leq d^2\cdot&\left(1-\Lambda t^2-2t\cdot h(\xi)\left(\frac{X}{|X|_{\gamma_0}},\frac{X}{|X|_{\gamma_0}}\right)\right.\\\nonumber
&\phantom{1-}+\sum\limits_{\beta\in I_\alpha}\left.\psi_\beta\cdot 2s\chi_\delta(t)\cdot(h(\xi)+k(\xi))\left(\frac{X}{|X|_{\gamma_0}},\frac{X}{|X|_{\gamma_0}}\right)\right)^{-1}
\end{align}
for sufficiently small $\delta_{I_\alpha}$.
Since the last three summands in the parentheses are $\gtrsim-\max_{\beta\in I_\alpha}\sqrt{\delta_\beta}$, we obtain $|X|_{\gamma_0}\leq 2d$ for sufficiently small $\delta_{I_\alpha}$.

Now set $N_\alpha:=|I_\alpha|$.
For the curvature computations, we relabel and sort the numbers $\delta_\beta,\beta\in I_\alpha$ by size, that is
\begin{equation}\nonumber
\delta^{(1)}\leq\delta^{(2)}\leq\dots\leq\delta^{(N_\alpha)}.
\end{equation}
Furthermore, $\delta^{(0)}:=0$ and $\delta^{(N_\alpha+1)}:=\max_{\beta\in I_\alpha}\sqrt{\delta_\beta}$.
Let $1\leq m\leq N_\alpha+1$.
If $\delta^{(m)}=\delta_\beta$ for some $\beta\in I_\alpha$, it is convenient to write $\psi^{(m)}=\psi_\beta$.

On the subset $[\delta^{(m-1)},\delta^{(m)}]\times U_\alpha^2$, we analyse the difference curvature tensor
\begin{align}\nonumber
T:=R_\gamma-\left(\sum\limits_{\beta=1}^{m-1}\psi^{(\beta)}\cdot R_{g_1(\xi)}+\sum\limits_{\beta=m}^{N_\alpha}\psi^{(\beta)}\cdot\bigl((1-s\dot{\chi}_{\delta^{(\beta)}})\cdot R_{g_1(\xi)}+s\dot{\chi}_{\delta^{(\beta)}}\cdot R_{g_2(\xi)}\bigr)\right).
\end{align}
Note that $0\leq\dot{\chi}_{\delta^{(\beta)}}\leq 1$ on $[0,\delta^{(\beta)}]$.
The discussion is similar to that in Lemma~\ref{uniform epsilon 2}, but the entry $R_\gamma(X,\nu,\nu,Y)$ for $(t,p)\in[\delta^{(m-1)},\delta^{(m)}]\times U_\alpha^2$ and $X,Y\in T_p(\partial M)$ is now given as
\begin{align}\nonumber
R_\gamma(X,\nu,\nu,Y)&=-\frac{1}{2}\ddot{\gamma}_t(X,Y)+\mathrm{II}_t^\gamma(W_t^\gamma(X),Y)\\\nonumber
&=\Lambda\cdot g_0(\xi)(X,Y)-\sum\limits_{\beta\in I_\alpha}\psi_\beta\cdot s\ddot{\chi}_{\delta_\beta}(t)\cdot(h(\xi)+k(\xi))(X,Y)+\mathrm{II}_t^\gamma(W_t^\gamma(X),Y)
\end{align}
with an additional term involving $\ddot{\chi}_\delta$.
We partition $T$ into three curvature tensors $T^{\intercal},T^{\lesssim 1}$ and $T^{\ddot{\chi}_\delta}$ on $[\delta^{(m-1)},\delta^{(m)}]\times U_\alpha^2$:
\begin{alignat}{2}\nonumber
&\mathrm{(I)}&&\text{$T^\intercal$ is defined via}\\\nonumber
&&&\hspace{1.5cm}T^\intercal(X,Y,Z,W)=T(X,Y,Z,W)\\\nonumber
&&&\hspace{1.5cm}T^\intercal(X,Y,Z,\nu)=0,\\\nonumber
&&&\hspace{1.5cm}T^\intercal(X,\nu,\nu,Y)=0.\\\nonumber
&\mathrm{(II)}&&\text{$T^{\lesssim 1}$ is defined via}\\\nonumber
&&&\hspace{1.5cm}T^{\lesssim 1}(X,Y,Z,W)=0,\\\nonumber
&&&\hspace{1.5cm}T^{\lesssim 1}(X,Y,Z,\nu)=T(X,Y,Z,\nu)\\\nonumber
&&&\hspace{1.5cm}T^{\lesssim 1}(X,\nu,\nu,Y)=T(X,\nu,\nu,Y)+\sum\limits_{\beta=m}^{N_\alpha}\psi^{(\beta)}\cdot s\ddot{\chi}_{\delta^{(\beta)}}(t)\cdot(h(\xi)+k(\xi))(X,Y).\\\nonumber
&\mathrm{(III)}\;\;&&\text{$T^{\ddot{\chi}_\delta}$ is defined via}\\\nonumber
&&&\hspace{1.5cm}T^{\ddot{\chi}_\delta}(X,Y,Z,W)=0,\\\nonumber
&&&\hspace{1.5cm}T^{\ddot{\chi}_\delta}(X,Y,Z,\nu)=0,\\\nonumber
&&&\hspace{1.5cm}T^{\ddot{\chi}_\delta}(X,\nu,\nu,Y)=-\sum\limits_{\beta=m}^{N_\alpha}\psi^{(\beta)}\cdot s\ddot{\chi}_{\delta^{(\beta)}}(t)\cdot(h(\xi)+k(\xi))(X,Y).
\end{alignat}
First, consider $T^\intercal$:
Using equations~\eqref{0nu},~\eqref{sff} and~\eqref{sff2}, a straightforward computation shows that
\begin{align}\nonumber
T^\intercal&=\sum_{\beta=1}^{m-1}\psi^{(\beta)}\cdot(R_{\gamma_t}-R_{g_1(\xi)_t})+\sum_{\beta=m}^{N_\alpha}\psi^{(\beta)}\cdot(1-s\dot{\chi}_{\delta^{(\beta)}}(t))\cdot(R_{\gamma_t}-R_{g_1(\xi)_t})\\\nonumber
&\phantom{=}+\sum_{\beta=m}^{N_\alpha}\psi^{(\beta)}\cdot s\dot{\chi}_{\delta^{(\beta)}}(t)\cdot(R_{\gamma_t}-R_{g_2(\xi)_t})\\\nonumber
&\phantom{=}+\Lambda t\cdot\sum_{\beta=1}^{m-1}\psi^{(\beta)}\cdot s\dot{\chi}_{\delta^{(\beta)}}(t)\cdot g_0(\xi)\KN(h(\xi)+k(\xi))+2\Lambda t\cdot\sum_{\beta=m}^{N_\alpha}\psi^{(\beta)}\cdot s\dot{\chi}_{\delta^{(\beta)}}(t)\cdot g_0(\xi)\KN k(\xi)\\\nonumber
&\phantom{=}+\frac{1}{2}\sum_{\beta=1}^{m-1}\psi^{(\beta)}\cdot s\dot{\chi}_{\delta^{(\beta)}}(t)\cdot\bigl(\mathrm{II}^{g_1(\xi)}_t-\mathrm{II}^{g_2(\xi)}_t\bigr)\KN(h(\xi)+k(\xi))\\\label{Ttangential}
&\phantom{=}+\frac{1}{2}\sum_{\beta,\mu=1}^{N_\alpha}\psi^{(\beta)}\psi^{(\mu)}\cdot s\dot{\chi}_{\delta^{(\beta)}}(t)\cdot(1-s\dot{\chi}_{\delta^{(\mu)}}(t))\cdot(h(\xi)+k(\xi))\KN(h(\xi)+k(\xi))
\end{align}
where all curvature tensors are evaluated on vectors $X,Y,Z,W\in T(\partial M)$.
We want to apply Corollary~\ref{uniforma} to the compact space
\begin{align}\nonumber
\mathscr{K}:=\{(\id_t\oplus\iota_{g_0(\xi)})^\ast((h(\xi)&+k(\xi))\KN(h(\xi)+k(\xi))):\xi\in K,(t,p)\in[\delta^{(m-1)},\delta^{(m)}]\times\overline{U_\alpha^2},\\\nonumber
&\iota_{g_0(\xi)(p)}\colon\mathbb{R}^{n-1}\to(T_p(\partial M),g_0(\xi)(p))\;\text{a linear isometry}\}\subset\overline{\mathsf{C}(0)}\subset\mathsf{C}(-\varepsilon_\alpha/2)
\end{align}
and a small number $r>0$ with $B_r(R)\subset\mathsf{C}(-\varepsilon_\alpha/2)$ for all $R\in\mathscr{K}$.
By the lemma and by the $\On(n)$-invariance of $\mathsf{C}(-\varepsilon_\alpha/2)$, there exists $a_\alpha>0$ such that
\begin{equation}\nonumber
(\id_t\oplus j)^\ast((h(\xi)+k(\xi))\KN(h(\xi)+k(\xi)))\in\mathsf{C}(-\varepsilon_\alpha/2)
\end{equation}
for all $\xi\in K,(t,p)\in[\delta^{(m-1)},\delta^{(m)}]\times U_\alpha^2$ and all $a_\alpha$-nearly $\gamma_0(p)$-isometric isomorphisms $j\colon\mathbb{R}^{n-1}\to T_p(\partial M)$.

We choose $\delta_{I_\alpha}$ so small that each linear isometry $\iota_{\gamma_t(p)}\colon\mathbb{R}^{n-1}\to(T_p(\partial M),\gamma_t(p))$ is $a_\alpha$-nearly $\gamma_0(p)$-isometric.
This is possible since
\begin{equation}\nonumber
\gamma_0(X,Y)=(1-\Lambda t^2)^{-1}\cdot\bigl(\gamma_t(X,Y)+2t\cdot h(\xi)(X,Y)-\sum_{\beta\in I_\alpha}\psi_\beta\cdot 2s\chi_{\delta_\beta}(t)\cdot(h(\xi)+k(\xi))(X,Y)\bigr).
\end{equation}
Next, we show that for sufficiently small $\delta_{I_\alpha}$, it holds
\begin{equation}\nonumber
s\dot{\chi}_{\delta^{(\beta)}}(t)\cdot(1-s\dot{\chi}_{\delta^{(\mu)}}(t))\cdot(\id_t\oplus\iota_{\gamma_t(p)})^\ast\bigl((h(\xi)+k(\xi))\KN(h(\xi)+k(\xi))\bigr)\in\mathsf{C}(-\varepsilon_\alpha)
\end{equation}
for all $1\leq\beta,\mu\leq N_{\alpha}$.
Let $r>0$ such that $B_r(\mathscr{C}_B(\mathbb{R}^n))\subset\mathsf{C}(-\varepsilon_\alpha/2)$.
There are several cases to consider:
\begin{itemize}
\item[$\triangleright$]{$1\leq\beta\leq m-1$:
It holds $|s\dot{\chi}_{\delta^{(\beta)}}(t)\cdot(1-s\dot{\chi}_{\delta^{(\mu)}}(t))|\lesssim\max_{\beta\in I_\alpha}\sqrt{\delta_\beta}$, so the norm of the above curvature tensors is $<r$ for sufficiently small $\left(\delta_\beta\right)_{\beta\in I_\alpha}$.}
\item[$\triangleright$]{$m\leq\beta\leq N_{\alpha},1\leq\mu\leq m-1$:
It holds $0\leq s\dot{\chi}_{\delta^{(\beta)}}(t)\cdot(1-s\dot{\chi}_{\delta^{(\mu)}}(t))\leq 2$ for sufficiently small $\delta_{I_\alpha}$, which yields the claim.}
\item[$\triangleright$]{$m\leq\beta,\mu\leq N_{\alpha}$: It holds $0\leq s\dot{\chi}_{\delta^{(\beta)}}(t)\cdot(1-s\dot{\chi}_{\delta^{(\mu)}}(t))\leq 1$, so the claim is obvious.}
\end{itemize}
We deduce that
\begin{equation}\nonumber
\frac{1}{2}\sum_{\beta,\mu=1}^{N_\alpha}\psi^{(\beta)}\psi^{(\mu)}\cdot s\dot{\chi}_{\delta^{(\beta)}}(t)\cdot(1-s\dot{\chi}_{\delta^{(\mu)}}(t))\cdot(\id_t\oplus\iota_{\gamma_t(p)})^\ast\bigl((h(\xi)+k(\xi))\KN(h(\xi)+k(\xi))\bigr)\in\mathsf{C}(-\varepsilon_\alpha/2)
\end{equation}
for all $\xi\in K,(t,p)\in[\delta^{(m-1)},\delta^{(m)}]\times U_\alpha^2$ and all linear isometries $\iota_{\gamma_t(p)}\colon\mathbb{R}^{n-1}\to(T_p(\partial M),\gamma_t(p))$.

It remains to investigate the other summands in $T^\intercal$.
Since
\begin{equation}\nonumber
\Vert\gamma_t-g_1(\xi)_t\Vert_{C^2(U_\alpha^3)}\leq\sum_{\beta\in I_\alpha}2s\chi_{\delta_\beta}(t)\cdot\Vert\psi_\beta\Vert_{C^2(U_\alpha^3)}\cdot\Vert h(\xi)+k(\xi)\Vert_{C^2(U_\alpha^3)}\lesssim\max\sqrt{\delta_{I_\alpha}},
\end{equation}
one can choose $\delta_{I_\alpha}$ so small that
\begin{equation}\nonumber
\Vert\iota_{\gamma_t(p)}^\ast(R_{\gamma_t}-R_{g_1(\xi)_t})\Vert<\frac{r}{4}.
\end{equation}
The same works for $g_2(\xi)_t$.
Hence,
\begin{align}\nonumber
\Bigl\Vert\sum_{\beta=1}^{m-1}&\psi^{(\beta)}\cdot\iota_{\gamma_t(p)}^\ast(R_{\gamma_t}-R_{g_1(\xi)_t})\Bigr\Vert+\Bigl\Vert\sum_{\beta=m}^{N_\alpha}\psi^{(\beta)}\cdot(1-s\dot{\chi}_{\delta^{(\beta)}}(t))\cdot\iota_{\gamma_t(p)}^\ast(R_{\gamma_t}-R_{g_1(\xi)_t})\Bigr\Vert\\\nonumber
&+\Bigl\Vert\sum_{\beta=m}^{N_\alpha}\psi^{(\beta)}\cdot s\dot{\chi}_{\delta^{(\beta)}}(t)\cdot\iota_{\gamma_t(p)}^\ast(R_{\gamma_t}-R_{g_2(\xi)_t})\Bigr\Vert<\frac{r}{4}.
\end{align}
It also holds
\begin{equation}\nonumber
\Bigl\Vert\Lambda t\cdot\sum_{\beta=1}^{m-1}\psi^{(\beta)}\cdot s\dot{\chi}_{\delta^{(\beta)}}(t)\cdot \iota_{\gamma_t(p)}^\ast(g_0(\xi)\KN(h(\xi)+k(\xi)))\Bigr\Vert<\frac{r}{4}
\end{equation}
for sufficiently small $\delta_{I_\alpha}$ because $\Lambda t\leq 1$ and $\dot{\chi}_{\delta^{(\beta)}}\lesssim\max\sqrt{\delta_{I_\alpha}}$.
The same argument yields
\begin{equation}\nonumber
\Bigl\Vert\frac{1}{2}\sum_{\beta=1}^{m-1}\psi^{(\beta)}\cdot s\dot{\chi}_{\delta^{(\beta)}}(t)\cdot\iota_{\gamma_t(p)}^\ast\bigl(\bigl(\mathrm{II}_t^{g_1(\xi)}-\mathrm{II}_t^{g_2(\xi)}\bigr)\KN(h(\xi)+k(\xi))\bigr)\Bigr\Vert<\frac{r}{4}
\end{equation}
for sufficiently small $\delta_{I_\alpha}$.
Finally, we observe that
\begin{equation}\nonumber
\Bigl\Vert 2\Lambda t\cdot\sum_{\beta=m}^{N_\alpha}\psi^{(\beta)}\cdot s\dot{\chi}_{\delta^{(\beta)}}(t)\cdot\iota_{\gamma_t(p)}^\ast(g_0(\xi)\KN k(\xi))\Bigr\Vert<\frac{r}{4}
\end{equation}
for sufficiently small $\delta_{I_\alpha}$.
This is clear if $m=N_\alpha+1$.
Otherwise it holds $\Lambda t\leq\max\sqrt{\delta_{I_\alpha}}$.

The curvature tensor defined by the first four rows in~\eqref{Ttangential} is hence pulled back to $B_r(\mathscr{C}_B(\mathbb{R}^n))$.
Therefore, $(\id_t\oplus\iota_{\gamma_t(p)})^\ast T^\intercal\in\mathsf{C}(-\varepsilon_\alpha)$.

In the next stage, we will consider $T^{\lesssim 1}$.
It holds
\begin{equation}\nonumber
R_\gamma(X,Y,Z,\nu)=(\nabla^{\gamma_t}_X\mathrm{II}_t^\gamma)(Y,Z)-(\nabla^{\gamma_t}_Y\mathrm{II}_t^\gamma)(X,Z).
\end{equation}

From
\begin{align}\nonumber
&\Vert\mathrm{II}_t^\gamma\Vert_{C^1(U^3_\alpha)}\leq\Vert h(\xi)\Vert_{C^1(U_\alpha^3)}+\Vert g_0(\xi)\Vert_{C^1(U_\alpha^3)}+\sum_{\beta\in I_\alpha}\Vert\psi_\beta\Vert_{C^1(U_\alpha^3)}\cdot s\dot{\chi}_{\delta_\beta}(t)\cdot\Vert h(\xi)+k(\xi)\Vert_{C^1(U_\alpha^3)}\lesssim 1,\\\nonumber
&\Vert\gamma_t\Vert_{C^1(U_\alpha^3)}\leq\Vert g_0(\xi)\Vert_{C^1(U_\alpha^3)}+2\Vert h(\xi)\Vert_{C^1(U_\alpha^3)}+\sum_{\beta\in I_\alpha}\Vert\psi_\beta\Vert_{C^1(U_\alpha^3)}\cdot\Vert h(\xi)+k(\xi)\Vert_{C^1(U_\alpha^3)}\lesssim 1
\end{align}
follows $|R_\gamma(X,Y,Z,\nu)|\lesssim 1$.
A similar estimate is valid for $R_{g_1(\xi)_t}(X,Y,Z,\nu)$ and $R_{g_2(\xi)_t}(X,Y,Z,\nu)$.
This yields
\begin{equation}\nonumber
|T^{\lesssim 1}(X,Y,Z,\nu)|\lesssim 1.
\end{equation}
The entry $T^{\lesssim 1}(X,\nu,\nu,Y)$ is given as
\begin{align}\nonumber
T^{\lesssim 1}(&X,\nu,\nu,Y)=\mathrm{II}_t^\gamma(W^\gamma_t(X),Y)-\sum_{\beta=1}^{m-1}\psi^{(\beta)}\cdot s\ddot{\chi}_{\delta^{(\beta)}}(t)\cdot (h(\xi)+k(\xi))(X,Y)\\\nonumber
&-\left(\sum_{\beta=1}^{m-1}\psi^{(\beta)}\cdot\mathrm{II}_t^{g_1(\xi)}(W_t^{g_1(\xi)}(X),Y)\right.\\\nonumber
&\phantom{-\Bigl(\beta}\left.+\sum_{\beta=m}^{N_\alpha}\psi^{(\beta)}\cdot\bigl((1-s\dot{\chi}_{\delta^{(\beta)}}(t))\cdot\mathrm{II}_t^{g_1(\xi)}(W_t^{g_1(\xi)}(X),Y)+s\dot{\chi}_{\delta^{(\beta)}}(t)\cdot\mathrm{II}_t^{g_2(\xi)}(W_t^{g_2(\xi)}(X),Y)\bigr)\right).
\end{align}
As $|\ddot{\chi}_{\delta^{(\beta)}}|\lesssim 1$ for $1\leq\beta\leq m$, it holds
\begin{equation}\nonumber
\Bigl\Vert\sum_{\beta=1}^{m-1}\psi^{(\beta)}\cdot s\ddot{\chi}_{\delta^{(\beta)}}(t)\cdot (h(\xi)+k(\xi))(X,Y)\Bigr\Vert\lesssim 1.
\end{equation}
In the samer manner as in~\eqref{Weingartennorm1} and~\eqref{Weingartennorm2}, one shows $|\mathrm{II}_t^\gamma(W^\gamma_t(X),Y)|\lesssim 1$ (similarly for $g_1(\xi),g_2(\xi)$).
Therefore,
\begin{equation}\nonumber
|T^{\lesssim 1}(X,\nu,\nu,Y)|\lesssim 1.
\end{equation}
This means $\Vert(\id_t\oplus\iota_{\gamma_t(p)})^\ast T^{\lesssim 1}\Vert\lesssim 1$.

Finally, we analyse $T^{\ddot{\chi}_\delta}$.
For all linear isometries $\iota_{\gamma_t(p)},\iota_{\gamma_0(p)}$ as above and for all $2\leq i,l\leq n$, it holds
\begin{align}\nonumber
(\id_t\oplus\iota_{\gamma_t})^\ast T^{\ddot{\chi}_\delta}(e_i,e_1,e_1,e_l)&=T^{\ddot{\chi}_\delta}(\iota_{\gamma_t}(e_i),\nu,\nu,\iota_{\gamma_t}(e_l))\\\nonumber
&=T^{\ddot{\chi}_\delta}((\iota_{\gamma_t}-\iota_{\gamma_0})(e_i)+\iota_{\gamma_0}(e_i),\nu,\nu,(\iota_{\gamma_t}-\iota_{\gamma_0})(e_l)+\iota_{\gamma_0}(e_l))\\\nonumber
&=T^{\ddot{\chi}_\delta}((\iota_{\gamma_t}-\iota_{\gamma_0})(e_i),\nu,\nu,(\iota_{\gamma_t}-\iota_{\gamma_0})(e_l))\\\nonumber
&\phantom{=}+T^{\ddot{\chi}_\delta}((\iota_{\gamma_t}-\iota_{\gamma_0})(e_i),\nu,\nu,\iota_{\gamma_0}(e_l))+T^{\ddot{\chi}_\delta}(\iota_{\gamma_0}(e_i),\nu,\nu,(\iota_{\gamma_t}-\iota_{\gamma_0})(e_l))\\\nonumber
&\phantom{=}+(\id_t\oplus\iota_{\gamma_0})^\ast T^{\ddot{\chi}_\delta}(e_i,e_1,e_1,e_l).
\end{align}
The last summand is
\begin{equation}\nonumber
(\id_t\oplus\iota_{\gamma_0(p)})^\ast T^{\ddot{\chi}_\delta}=-\sum_{\beta=m}^{N_\alpha}\psi^{(\beta)}\cdot s\ddot{\chi}_{\delta^{(\beta)}}(t)\cdot\iota_{g_0(\xi)(p)}^\ast((h(\xi)+k(\xi))\KN\mathfrak{b}^\perp\in\overline{\mathsf{C}(0)}
\end{equation}
since $-s\ddot{\chi}_{\delta^{(\beta)}}\geq 0$ for $m\leq\beta\leq N_\alpha$.

It remains to investigate $(\id_t\oplus\iota_{\gamma_t(p)})^\ast T^{\ddot{\chi}_\delta}-(\id_t\oplus\iota_{\gamma_0(p)})^\ast T^{\ddot{\chi}_\delta}$.
For an arbitrary linear isometry\break$\iota_{\gamma_t(p)}\colon\mathbb{R}^{n-1}\to(T_p(\partial M),\gamma_t(p))$ and $2\leq i,l\leq n$, it holds
\begin{align}\nonumber
|\gamma_0(p)(\iota_{\gamma_t(p)}(e_i),\iota_{\gamma_t(p)}(e_l))-\delta_{il}|&\lesssim|(1-\Lambda t^2)^{-1}\cdot\gamma_t(p)(\iota_{\gamma_t(p)}(e_i),\iota_{\gamma_t(p)}(e_l))-\delta_{il}|\\\nonumber
&\phantom{\lesssim}+2t\cdot(1-\Lambda t^2)^{-1}+\sum_{\beta\in I_\alpha}\psi_\beta\cdot 2s\chi_{\delta_\beta}(t)\cdot(1-\Lambda t^2)^{-1}\\\nonumber
&\lesssim\frac{\Lambda t^2}{1-\Lambda t^2}+t+t\lesssim t.
\end{align}
The essential step is to apply Lemma~\ref{QGS}:
For sufficiently small $\delta_{I_\alpha}$ there exists a linear isometry $\iota_{\gamma_0(p)}\colon\mathbb{R}^{n-1}\to(T_p(\partial M),\gamma_0(p))$ such that
\begin{equation}\label{specialisometry}
|\gamma_0(p)(\iota_{\gamma_t(p)}(e_i),\iota_{\gamma_0(p)}(e_l))-\delta_{il}|\lesssim t.
\end{equation}
For this $\iota_{\gamma_0(p)}$, we obtain
\begin{equation}\nonumber
|(\iota_{\gamma_t(p)}-\iota_{\gamma_0(p)})(e_i)|_{\gamma_0(p)}\leq\sum\limits_{l=2}^n|\gamma_0(p)(\iota_{\gamma_t(p)}(e_i),\iota_{\gamma_0(p)}(e_l))-\delta_{il}|\lesssim t,
\end{equation}
so
\begin{equation}\nonumber
-s\ddot{\chi}_{\delta^{(\beta)}}(t)\cdot|(\iota_{\gamma_t(p)}-\iota_{\gamma_0(p)})(e_i)|_{\gamma_0(p)}\lesssim 1\quad\text{and}\quad\sqrt{-s\ddot{\chi}_{\delta^{(\beta)}}(t)}\cdot|(\iota_{\gamma_t(p)}-\iota_{\gamma_0(p)})(e_i)|_{\gamma_0(p)}\lesssim 1
\end{equation}
for all $m\leq\beta\leq N_{\alpha}$.
Hence,
\begin{align}\nonumber
|\bigl((\id_t\oplus\iota_{\gamma_t(p)})^\ast&T^{\ddot{\chi}_\delta}-(\id_t\oplus\iota_{\gamma_0(p)})^\ast T^{\ddot{\chi}_\delta}\bigr)(e_i,e_1,e_1,e_l)|\\\nonumber
&\leq\sum_{\beta=m}^{N_\alpha}\psi^{(\beta)}\cdot\Bigl(\bigl|(h(\xi)+k(\xi))\bigl(\sqrt{-s\ddot{\chi}_\delta(t)}\cdot(\iota_{\gamma_t}-\iota_{\gamma_0})(e_i),\sqrt{-s\ddot{\chi}_\delta(t)}\cdot(\iota_{\gamma_t}-\iota_{\gamma_0})(e_l)\bigr)\bigr|\\\nonumber
&\hphantom{\sum_{\beta=m}^{N_\alpha}\psi^{(\beta)}\cdot\Bigl(h(\xi)}+\left|(h(\xi)+k(\xi))\bigl(-s\ddot{\chi}_\delta(t)\cdot(\iota_{\gamma_t}-\iota_{\gamma_0})(e_i),\iota_{\gamma_0}(e_l)\bigr)\right|\\\nonumber
&\hphantom{\sum_{\beta=m}^{N_\alpha}\psi^{(\beta)}\cdot\Bigl(h(\xi)}+\left|(h(\xi)+k(\xi))\bigl(\iota_{\gamma_0}(e_i),-s\ddot{\chi}_\delta(t)\cdot(\iota_{\gamma_t}-\iota_{\gamma_0})(e_l)\bigr)\right|\Bigr)\\\nonumber
&\lesssim 1,
\end{align}
which yields $\Vert(\id_t\oplus\iota_{\gamma_t(p)})^\ast T^{\ddot{\chi}_\delta}-(\id_t\oplus\iota_{\gamma_0(p)})^\ast T^{\ddot{\chi}_\delta}\Vert\lesssim 1$.

Now, we put all the parts together:
For sufficiently small $\delta_{I_\alpha}$ each linear isometry $\iota_{\gamma_t(p)}\colon\mathbb{R}^{n-1}\to(T_p(\partial M),\gamma_t(p))$ is $\Lambda(p)^{-1}$-nearly $g_1(\xi)_t(p)$-isometric and $g_2(\xi)_t(p)$-isometric. This is because
\begin{align}\nonumber
|g_2(\xi)_t(\iota_{\gamma_t}(e_i),\iota_{\gamma_t}(e_l))-\delta_{il}|&\leq 2t\cdot|h(\xi)(\iota_{\gamma_t}(e_i),\iota_{\gamma_t}(e_l))-k(\xi)(\iota_{\gamma_t}(e_i),\iota_{\gamma_t}(e_l))|\\\nonumber
&\phantom{\leq}+\sum_{\beta\in I_\alpha}\psi_\beta\cdot2s\chi_{\delta_\beta}(t)\cdot|h(\xi)(\iota_{\gamma_t}(e_i),\iota_{\gamma_t}(e_l))+k(\xi)(\iota_{\gamma_t}(e_i),\iota_{\gamma_t}(e_l))|\lesssim\Lambda^{-1}\sqrt{t}
\end{align}
and similarly for $g_1(\xi)_t$.
By~\eqref{Lambdadisc} we have
\begin{equation}\nonumber
(1-s\dot{\chi}_{\delta^{(\beta)}}(t))\cdot(\id_t\oplus\iota_{\gamma_t(p)})^\ast R_{g_1(\xi)}(t,p)+s\dot{\chi}_{\delta^{(\beta)}}(t)\cdot(\id_t\oplus\iota_{\gamma_t(p)})^\ast R_{g_2(\xi)}(t,p)+B^\perp_{\rho_\alpha\sqrt{\Lambda(p)}}\subset\mathsf{C}(\sigma(t,p)+\varepsilon_\alpha)
\end{equation}
for all $m\leq\beta\leq N_\alpha$, all $\xi\in K,s\in[0,1]$ and $(t,p)\in[\delta^{(m-1)},\delta^{(m)}]\times U_\alpha^2$.
Thus,
\begin{equation}\nonumber
(\id_t\oplus\iota_{\gamma_t(p)})^\ast(R_\gamma(t,p)-T(t,p))+B^\perp_{\rho_\alpha\sqrt{\Lambda(p)}}\subset\mathsf{C}(\sigma(t,p)+\varepsilon_\alpha).
\end{equation}
When choosing $\Lambda_{0,\alpha}$ bigger if necessary, it holds
\begin{equation}\nonumber
(\id_t\oplus\iota_{\gamma_t(p)})^\ast T^{\lesssim 1}(t,p)+(\id_t\oplus\iota_{\gamma_t(p)})^\ast T^{\ddot{\chi}_\delta}(t,p)-(\id_t\oplus\iota_{\gamma_0(p)})^\ast T^{\ddot{\chi}_\delta}(t,p)\in B^\perp_{\rho_\alpha\sqrt{\Lambda(p)}}
\end{equation}
for $\iota_{\gamma_{0}(p)}$ as in~\eqref{specialisometry}.
This leads us to conclude that
\begin{align}\nonumber
&(\id_t\oplus\iota_{\gamma_t(p)})^\ast R_\gamma(t,p)\\\nonumber
=\;&(\id_t\oplus\iota_{\gamma_t(p)})^\ast(R_\gamma(t,p)-T(t,p))+(\id_t\oplus\iota_{\gamma_t(p)})^\ast T^{\lesssim 1}(t,p)\\\nonumber
&\quad+(\id_t\oplus\iota_{\gamma_t(p)})^\ast T^{\ddot{\chi}_\delta}(t,p)-(\id_t\oplus\iota_{\gamma_0(p)})^\ast T^{\ddot{\chi}_\delta}(t,p)\\\nonumber
&\quad+(\id_t\oplus\iota_{\gamma_t(p)})^\ast T^\intercal(t,p)+(\id_t\oplus\iota_{\gamma_0(p)})^\ast T^{\ddot{\chi}_\delta}(t,p)\\\nonumber
&\quad\phantom{+(\id_t}\in\mathsf{C}(\sigma(t,p)+\varepsilon_\alpha)+\mathsf{C}(-\varepsilon_\alpha)+\overline{\mathsf{C}(0)}\subset\mathsf{C}(\sigma(t,p)).
\end{align}
Since $1\leq m\leq N_\alpha+1$ and $(t,p)\in[\delta^{(m-1)},\delta^{(m)}]\times U_\alpha^2$ have been chosen arbitrarily, $R_\gamma$ satisfies $\mathsf{C}(\sigma)$ on $\square_{\eta_\alpha}^2$ for sufficiently small $\delta_{I_\alpha}$.
Set
\begin{equation}\nonumber
\Lambda_0\colon\partial M\to\mathbb{R}\,,\,\Lambda_0(p)=\sum_{\alpha\in\mathbb{N}}\psi_\alpha\cdot\max_{\beta\in I_\alpha}\Lambda_{0,\beta}(p).
\end{equation}
Observe that $\Lambda_0(p)\geq\Lambda_{0,\alpha}(p)$ when $p\in U_\alpha^3$.
Let $\Lambda\geq\Lambda_0$.

For the calculations on $\square_{\eta_\alpha}^2$ with fixed $\alpha\in\mathbb{N}$, we impose finitely many conditions on $\delta_{I_\alpha}$, and in particular on $\delta_\alpha$, in order to arrange the curvature bound.
Repeating the procedure for all boxes $\square_{\eta_\beta}^2,\beta\in\mathbb{N}$ will add only finitely many conditions on $\delta_\alpha$ because there are only finitely many $\beta\in\mathbb{N}$ with $\alpha\in I_\beta$.

This is the reason why one can choose all numbers $\delta_\alpha,\alpha\in\mathbb{N}$ so small that $R_\gamma$ satisfies $\mathsf{C}(\sigma)$ on $\bigcup_{\alpha\in\mathbb{N}}\square_{\eta_\alpha}^2$.
\end{proof}
The results from Proposition~\ref{2jet} and~\ref{1jet} are combined in the following theorem, where we use two auxiliary functions $S_1,S_2\colon[0,1]\to[0,1]$ defined by
\begin{equation}\nonumber
S_1(t)=\left\{\begin{array}{ll} 1, & 0\leq t\leq\tfrac{1}{2}\\
         2(1-t), & \tfrac{1}{2}\leq t\leq 1\end{array}\right.\quad\text{and}\quad
S_2(t)=\left\{\begin{array}{ll} 1-2t, & 0\leq t\leq\tfrac{1}{2}\\
         0, & \tfrac{1}{2}\leq t\leq 1\end{array}\right..
\end{equation}
We also use the notation
\begin{equation}\nonumber
\mathscr{R}_{\mathsf{C}(\sigma)}^{0}(M_1\sqcup M_2):=\{g_1\sqcup g_2\in\mathscr{R}_{\mathsf{C}(\sigma)}(M_1\sqcup M_2):(g_1)_0=(g_2)_0\}.
\end{equation}
\begin{theorem}\label{theorem1}
Let $K$ be a compact Hausdorff space and let
\begin{equation}\nonumber
g_1\sqcup g_2\colon K\to\mathscr{R}_{\mathsf{C}(\sigma)}^{0}(M_1\sqcup M_2)
\end{equation}
be continuous with
\begin{align}\nonumber
&\iota_{g(\xi)_0}^\ast(\mathrm{II}_{g_1(\xi)}+\mathrm{II}_{g_2(\xi)})\KN\iota_{g(\xi)_0}^\ast(\mathrm{II}_{g_1(\xi)}+\mathrm{II}_{g_2(\xi)})\in\overline{\mathsf{C}(0)}\quad\text{and}\\\nonumber
&\iota_{g(\xi)_0}^\ast(\mathrm{II}_{g_1(\xi)}+\mathrm{II}_{g_2(\xi)})\KN\mathfrak{b}^\perp\in\overline{\mathsf{C}(0)}
\end{align}
for all $\xi\in K$ and $p\in\partial M$, where $\iota_{g(\xi)_0(p)}\colon\mathbb{R}^{n-1}\to (T_p(\partial M),g_1(\xi)_0(p)=g_2(\xi)(p))$ is a linear isometry for $\xi\in K$ and $p\in\partial M$.

Then there exists a smooth positive function $\Lambda_0\in C^\infty(\partial M)$ such that for each $\Lambda\in C^\infty(\partial M)$ with $\Lambda\geq\Lambda_0$ and for each neighbourhood $\mathscr{U}\subset M_1\sqcup M_2$ of $\partial M_1\sqcup\partial M_2$, there is a continuous map
\begin{equation}\nonumber
f\colon K\times[0,1]\to\mathscr{R}_{\mathsf{C}(\sigma)}^{0}(M_1\sqcup M_2)
\end{equation}
so that the following holds for all $\xi\in K$ and $s\in[0,1]$:
\begin{enumerate}
\item[(a)]{$f(\xi,0)=g_1(\xi)\sqcup g_2(\xi)$;}
\item[(b)]{$f(\xi,1)$ is $\Lambda$-normal;}
\item[(c)]{$f(\xi,s)_0=g_1(\xi)_0\sqcup g_2(\xi)_0$;}
\item[(d)]{$\mathrm{II}_{f(\xi,s)}=(S_1(s)\mathrm{II}_{g_1(\xi)}-(1-S_1(s))\mathrm{II}_{g_2(\xi)})\sqcup\mathrm{II}_{g_2(\xi)}$, in particular\\
$\mathrm{II}_{f(\xi,1)}=(-\mathrm{II}_{g_2(\xi)})\sqcup\mathrm{II}_{g_2(\xi)}$;}
\item[(e)]{if $g_1(\xi)\sqcup g_2(\xi)$ is $\tilde{\Lambda}$-normal, then $f(\xi,s)$ is $\Lambda_s$-normal for $\Lambda_s=S_2(s)\tilde{\Lambda}+(1-S_2(s))\Lambda$;}
\item[(f)]{$\ddot{f}(\xi,s)_0=S_2(s)\cdot\bigl(\ddot{g}_1(\xi)_0\sqcup\ddot{g}_2(\xi)_0\bigr)-2(1-S_2(s))\cdot\Lambda\bigl(g_1(\xi)_0\sqcup g_2(\xi)_0\bigr)$;}
\item[(g)]{for $\ell\geq 3$ we have $f(\xi,s)_0^{(\ell)}=S_2(s)\cdot\bigl(g_1(\xi)_0^{(\ell)}\sqcup g_2(\xi)_0^{(\ell)}\bigr)$;}
\item[(h)]{$f(\xi,s)=g_1(\xi)\sqcup g_2(\xi)$ on $M_1\sqcup M_2\setminus\mathscr{U}$.}
\end{enumerate}
\end{theorem}
\begin{proof}
The proof is the same as in~\cite[Thm.~3.7]{BH2023}.
We apply Proposition~\ref{2jet} and Proposition~\ref{1jet} for sufficiently large $\Lambda$, one after the other, and concatenate the arising homotopies.
\end{proof}

\subsection{Spaces of metrics with singularities}
We set
\begin{align}\nonumber
\mathscr{R}_{\mathsf{C}(\sigma)}^{\mathrm{P}}(M):=\{g\in\mathscr{R}_{\mathsf{C}(\sigma)}(M): &\,\iota_{g_0}^\ast\mathrm{II}_g\KN\iota_{g_0}^\ast\mathrm{II}_g\in\overline{\mathsf{C}(0)},\\\nonumber
&\,\iota_{g_0}^\ast\mathrm{II}_g\KN\mathfrak{b}^\perp\in\overline{\mathsf{C}(0)}\}
\end{align}
for a single manifold $M$ and
\begin{align}\nonumber
\mathscr{R}_{\mathsf{C}(\sigma)}^{\mathrm{P}}(M_1\sqcup M_2):=\{g_1\sqcup g_2&\in\mathscr{R}_{\mathsf{C}(\sigma)}^0(M_1\sqcup M_2):\\\nonumber
&\iota_{g_0}^\ast(\mathrm{II}_{g_1}+\mathrm{II}_{g_2})\KN\iota_{g_0}^\ast(\mathrm{II}_{g_1}+\mathrm{II}_{g_2})\in\overline{\mathsf{C}(0)},\\\nonumber
&\iota_{g_0}^\ast(\mathrm{II}_{g_1}+\mathrm{II}_{g_2})\KN\mathfrak{b}^\perp\in\overline{\mathsf{C}(0)}\}.
\end{align}
As usual, it is understood that the boundary conditions are satisfied for all $p\in\partial M$ and all linear isometries $\iota_{g_0}\colon\mathbb{R}^{n-1}\to(T_p(\partial M),g_0(p))$ or $\iota_{g_0}\colon\mathbb{R}^{n-1}\to(T_p(\partial M),(g_1)_0(p)=(g_2)_0(p))$, respectively.

A Riemannian metric $g$ on $M$ is called a \textit{doubling metric} if $g\cup g$ is a smooth Riemannian metric on $M\cup_{g\sqcup g}M$.
This is equivalent to the condition that $g_0^{(\ell)}\equiv 0$ for $\ell$ odd.
We denote the space of doubling metrics by $\mathscr{R}_{\mathsf{C}(\sigma)}^{\mathsf{D}}(M)$, and the spaces of metrics with totally geodesic or convex boundary by $\mathscr{R}_{\mathsf{C}(\sigma)}^{\mathrm{II}=0}(M)$ and $\mathscr{R}_{\mathsf{C}(\sigma)}^{\mathrm{II}\geq 0}(M)$, respectively.

\begin{corollary}\label{bc}
Each of the inclusions
\begin{equation}\nonumber
\mathscr{R}_{\mathsf{C}(\sigma)}^{\mathsf{D}}(M)\hookrightarrow\mathscr{R}_{\mathsf{C}(\sigma)}^{\mathrm{II}=0}(M)\hookrightarrow\mathscr{R}_{\mathsf{C}(\sigma)}^{\mathrm{II}\geq 0}(M)\hookrightarrow\mathscr{R}_{\mathsf{C}(\sigma)}^{\mathrm{P}}(M)
\end{equation}
is a weak homotopy equivalence.
\end{corollary}
\begin{proof}
Let $\ast$ be any of the conditions in $\{\mathrm{II}=0,\mathrm{II}\geq 0,\mathrm{P}\}$.
We show that the inclusion $\mathscr{R}_{\mathsf{C}(\sigma)}^{\mathsf{D}}(M)\hookrightarrow\mathscr{R}_{\mathsf{C}(\sigma)}^{\ast}(M)$ is a weak homotopy equivalence.

Let $m\geq 0$ and $g\colon D^m\to\mathscr{R}_{\mathsf{C}(\sigma)}^\ast(M)$ be continuous with $g(\partial D^m)\subset\mathscr{R}_{\mathsf{C}(\sigma)}^{\mathsf{D}}(M)$.
Here $D^m\subset\mathbb{R}^m$ is the standard closed $m$-ball.
We apply Theorem~\ref{theorem1} with $K=D^m$ and the manifold $M\sqcup M$.
Restricting the resulting homotopy of metrics to one copy of $M$ yields a continuous map $f\colon D^m\times[0,1]\to\mathscr{R}(M)$ so that the following holds for all $\xi\in K$ and $s\in[0,\tfrac{3}{4}]$:
\begin{enumerate}
\item[$\triangleright$]{$f(\xi,s)\in\mathscr{R}_{\mathsf{C}(\sigma)}^\ast(M)$;}
\item[$\triangleright$]{$f(\xi,0)=g(\xi);$}
\item[$\triangleright$]{$\mathrm{II}_{f(\xi,s)}=(2S_1(s)-1)\mathrm{II}_{g(\xi)}$, in particular $\mathrm{II}_{f(\xi,\nicefrac{3}{4})}=0$;}
\item[$\triangleright$]{for $\ell\geq 3$ we have $f(\xi,s)_0^{(\ell)}=S_2(s)\cdot g(\xi)_0^{(\ell)}$, in particular $f(\xi,\tfrac{3}{4})_0^{(\ell)}=0$.}
\end{enumerate}
Thus, we obtain a homotopy $f\colon D^m\times[0,\tfrac{3}{4}]\to\mathscr{R}_{\mathsf{C}(\sigma)}^\ast(M)$ with $f(\xi,s)\in\mathscr{R}_{\mathsf{C}(\sigma)}^{\mathsf{D}}(M)$ for all $(\xi,s)\in(\partial D^m\times[0,\tfrac{3}{4}])\cup(D^m\times\{\tfrac{3}{4}\})$.
This means that the pair $(\mathscr{R}_{\mathsf{C}(\sigma)}^\ast(M),\mathscr{R}_{\mathsf{C}(\sigma)}^{\mathsf{D}}(M))$ is $m$-connected for all $m\geq 0$.
By the long exact sequence for homotopy groups, the inclusion $\mathscr{R}_{\mathsf{C}(\sigma)}^{\mathsf{D}}(M)\hookrightarrow\mathscr{R}_{\mathsf{C}(\sigma)}^\ast(M)$ is a weak homotopy equivalence.
\end{proof}

In the next theorem, we consider the glued manifold $M_1\cup_{\partial M}M_2$ which is equipped with an arbitrary, but fixed smooth structure arising from collar neighbourhoods.
\begin{theorem}\label{whe}
Each of the inclusions
\begin{equation}\nonumber
\mathscr{R}_{\mathsf{C}(\sigma)}(M_1\cup_{\partial M}M_2)\hookrightarrow\mathscr{R}_{\mathsf{C}(\sigma)}^{\mathrm{II}_1+\mathrm{II}_2\geq 0}(M_1\sqcup M_2)\hookrightarrow\mathscr{R}_{\mathsf{C}(\sigma)}^{\mathrm{P}}(M_1\sqcup M_2)
\end{equation}
is a weak homotopy equivalence.
\end{theorem}
\begin{proof}
We first prove the theorem for a particular smooth structure on $M_1\cup_{\partial M}M_2$.
In order to find this smooth structure, we apply Theorem~\ref{theorem1} for an arbitrary (single) metric in $\mathscr{R}^0(M_1\sqcup M_2)$ and the trivial family of curvature conditions where every member is $\mathscr{C}_B(\mathbb{R}^n)$.
This yields a Riemannian metric $\zeta\in\mathscr{R}^0(M_1\sqcup M_2)$ which is smooth on $M_1\cup_{\zeta}M_2$.
Let us fix the smooth manifold $M_1\cup_\zeta M_2$.

Now let $m\geq 0$ and $g=g_1\sqcup g_2\colon D^m\to\mathscr{R}_{\mathsf{C}(\sigma)}^{\ast}(M_1\sqcup M_2)$ be continuous with $\ast$ in $\{\mathrm{II}_1+\mathrm{II}_2\geq 0, \mathrm{P}\}$ and $g(\partial D^m)\subset\mathscr{R}_{\mathsf{C}(\sigma)}(M_1\cup_{\zeta}M_2)$.
We apply Proposition~\ref{diff} and Corollary~\ref{uniformgcn} for $g$ and the distinguished metric $\zeta$.
This yields a family of diffeotopies $\Omega\colon D^m\times[0,1]\to\Diff(M_1\sqcup M_2)$ such that the homotopy of metrics
\begin{equation}\nonumber
\tilde{g}\colon D^m\times[0,1]\to\mathscr{R}(M_1\sqcup M_2)\,,\,\tilde{g}(\xi,s)=\Omega(\xi,s)^\ast g(\xi)
\end{equation}
meets the following conditions for all $\xi\in K$ and $s\in[0,1]$:
\begin{enumerate}
\item[$\triangleright$]{$\tilde{g}(\xi,s)\in\mathscr{R}^\ast_{\mathsf{C}(\sigma)}(M_1\sqcup M_2)$;}
\item[$\triangleright$]{$\tilde{g}(\xi,0)=g(\xi)$;}
\item[$\triangleright$]{$\varrho_{\tilde{g}(\xi,1)}=\varrho_{\zeta}$ on some uniform neighbourhood of $\partial M_1\sqcup\partial M_2$;}
\item[$\triangleright$]{$\tilde{g}(\xi,s)_t=g(\xi)_t$ near $\partial M_1\sqcup\partial M_2$;}
\item[$\triangleright$]{$\tilde{g}(\xi,s)$ is smooth on $M_1\cup_{\zeta}M_2$ if $\xi\in\partial D^m$.}
\end{enumerate}
The last aspect follows from the fact that $\Omega(\xi,s)$ is smooth on $M_1\cup_\zeta M_2$ if $\xi\in\partial D^m$, because, for every $\xi\in\partial D^m$ and $s\in[0,1]$, the metric $(1-s)\cdot\zeta+s\cdot g(\xi)$ is smooth on $M_1\cup_{\zeta}M_2$.
This feature is the reason why we consider a specific smooth structure.

In summary, the preparatory deformation results in a continuous map $\tilde{g}\colon D^m\times[0,1]\to\mathscr{R}_{\mathsf{C}(\sigma)}^{\ast}(M_1\sqcup M_2)$ with $\tilde{g}(\xi,s)\in\mathscr{R}_{\mathsf{C}(\sigma)}(M_1\cup_\zeta M_2)$ for all $(\xi,s)\in\partial D^m\times[0,1]$ and $\varrho_{\tilde{g}(\xi,1)}=\varrho_\zeta$ for all $\xi\in D^m$.

Next we apply Theorem~\ref{theorem1} to the family $\tilde{g}(\xi,1),\xi\in D^m$ and obtain a continuous map $f\colon D^m\times[0,1]\to\mathscr{R}_{\mathsf{C}(\sigma)}^{\ast}(M_1\sqcup M_2)$ such that $f(\xi,s)\in\mathscr{R}_{\mathsf{C}(\sigma)}(M_1\cup_\zeta M_2)$ for all $(\partial D^m\times[0,1])\cup(D^m\times\{1\})$, using parts $\textit{(d)},\textit{(f)}$ and $\textit{(g)}$ of Theorem~\ref{theorem1}.
This yields the assertion for $M_1\cup_\zeta M_2$.

For a general smooth structure on $M_1\cup_{\partial M}M_2$ arising from arbitrary collar neighbourhoods, we use the existence of a diffeomorphism $M_1\cup_{\partial M}M_2\to M_1\cup_{\zeta}M_2$ which maps $M_1$ onto $M_1$ and $M_2$ onto $M_2$, see e.g.~\cite[Chap.~8, Thm.~2.1]{Hirsch}.
\end{proof}
\subsection{Examples}
We will show that Theorem~\ref{whe} implies all gluing results in Table~\ref{overview} where the boundary condition is 'convex'.
Furthermore, it will imply the result for $\PIC$.
\begin{proposition}
Let $h$ be a symmetric bilinear form on $\mathbb{R}^{n-1}$.
\begin{enumerate}
\item[(a)]{If $h$ is positive semi-definite, then $h\KN h$ and $h\KN\mathfrak{b}^\perp$ have $\mathcal{R}\geq 0$.}
\item[(b)]{If $h\KN\mathfrak{b}^\perp$ has $\ric\geq 0$, then $h$ is positive semidefinite.}
\end{enumerate}
\end{proposition}
\begin{proof}
We choose an orthonormal basis $(u_2,\dots,u_n)$ of $\mathbb{R}^{n-1}$ which is diagonalising for $h$, i.e. there exist $\lambda_2,\dots,\lambda_n\in\mathbb{R}$ such that
\begin{equation}\nonumber
h(u_i,u_l)=\lambda_i\delta_{il}
\end{equation}
for all $2\leq i,l\leq n$.
Setting $u_1:=e_1$, the family $(u_1,u_2,\dots,u_n)$ is an orthonormal basis of $\mathbb{R}^n$.

\textit{(a)} It holds
\begin{align}\nonumber
\langle\mathcal{R}(h\KN h)(u_{i_1}\wedge u_{i_2}),u_{i_3}\wedge u_{i_4}\rangle&=-(h\KN h)(u_{i_1},u_{i_2},u_{i_3},u_{i_4})\\\nonumber
&=-2h(u_{i_1},u_{i_4})h(u_{i_2},u_{i_3})+2h(u_{i_1},u_{i_3})h(u_{i_2},u_{i_4})\\\nonumber
&=-2\lambda_{i_1}\lambda_{i_2}\cdot(\delta_{i_1i_4}\delta_{i_2i_3}-\delta_{i_1i_3}\delta_{i_2i_4})
\end{align}
for all $2\leq i_1<i_2<n$ and $2\leq i_3<i_4\leq n$.
From this we conclude that
\begin{equation}\nonumber
\mathcal{R}(h\KN h)(u_i\wedge u_l)=2\lambda_i\lambda_l\cdot u_i\wedge u_l
\end{equation}
for all $2\leq i<l\leq n$.
Furthermore, $\mathcal{R}(h\KN h)(u_1\wedge u_i)=0$ for all $2\leq i\leq n$.

Next it holds
\begin{align}\nonumber
\langle\mathcal{R}(h\KN\mathfrak{b}^\perp)(u_{i_1}\wedge u_{i_2}),u_{i_3}\wedge u_{i_4}\rangle&=-(h\KN\mathfrak{b}^\perp)(u_{i_1},u_{i_2},u_{i_3},u_{i_4})\\\nonumber
&=-h(u_{i_1},u_{i_4})\cdot\delta_{i_2i_3}\delta_{1i_2}-h(u_{i_2},u_{i_3})\cdot\delta_{i_1i_4}\delta_{1i_1}\\\nonumber
&\phantom{=-}+h(u_{i_1},u_{i_3})\cdot\delta_{i_2i_4}\delta_{1i_2}+h(u_{i_2},u_{i_4})\cdot\delta_{i_1i_3}\delta_{1i_1}\\\nonumber
&=\lambda_{i_2}\cdot\delta_{i_2i_4}\delta_{i_1i_3}\delta_{1i_1}
\end{align}
for all $1\leq i_1<i_2\leq n$ and $1\leq i_3<i_4\leq n$.
This yields
\begin{equation}\nonumber
\mathcal{R}(h\KN\mathfrak{b}^\perp)(u_1\wedge u_i)=\lambda_i\cdot u_1\wedge u_i
\end{equation}
for all $2<i\leq n$, and $\mathcal{R}(h\KN\mathfrak{b}^\perp)(u_i\wedge u_j)=0$ for all other $u_i\wedge u_j$.
The assertion follows.

\textit{(b)} Suppose that $h\KN\mathfrak{b}^\perp$ has $\ric\geq 0$.
For all $2\leq i\leq n$ holds
\begin{equation}\nonumber
\lambda_i=h(u_i,u_i)=\sum_{l=1}^n(h\KN\mathfrak{b}^\perp)(u_i,u_l,u_l,u_i)=\ric(h\KN\mathfrak{b}^\perp)(u_i,u_i)\geq 0.\tag*{\qedhere}
\end{equation}
\end{proof}
\begin{corollary}
Let $\mathsf{C}\subset\mathscr{C}_B(\mathbb{R}^n)$ be an open, convex and $\On(n)$-invariant cone with
\begin{equation}\nonumber
\mathsf{C}_{\mathcal{R}>0}\subset\mathsf{C}\subset\mathsf{C}_{\ric>0}.
\end{equation}
Let $h$ be a symmetric bilinear form on $\mathbb{R}^{n-1}$.
Then $h\KN h\in\overline{\mathsf{C}}$ and $h\KN\mathfrak{b}^\perp\in\overline{\mathsf{C}}$ if and only if $h$ is positive semidefinite.
\end{corollary}
\begin{example}
Let $h$ be a symmetric bilinear form on $\mathbb{R}^{n-1}$.
Then both $h\KN h$ and $h\KN\mathfrak{b}^\perp$ have non-negative isotropic curvature if and only if $h$ is $2$-non-negative.
\end{example}
\begin{proof}
It was shown by Chow~\cite[Lem.~3]{Chow} that $h\KN h$ and $h\KN \mathfrak{b}^{\perp}$ have non-negative isotropic curvature if $h$ is $2$-non-negative.

Now suppose that $h\KN \mathfrak{b}^{\perp}$ has non-negative isotropic curvature.
Let $v_1,v_2\in\mathbb{R}^{n-1}$ be orthonormal.
Choose another unit vector $v_3\in\mathbb{R}^{n-1}$ which is orthogonal to $v_1$ and $v_2$.
Then $(e_1,v_1,v_2,v_3)$ is an orthonormal $4$-frame in $\mathbb{R}^n$, and
\begin{align}\nonumber
h(v_1,v_1)+h(v_2,v_2)&=(h\KN\mathfrak{b}^\perp)(e_1,v_1,v_1,e_1)+(h\KN\mathfrak{b}^\perp)(e_1,v_2,v_2,e_1)\\\nonumber
&\phantom{=}+(h\KN \mathfrak{b}^\perp)(v_1,v_3,v_3,v_1)+(h\KN\mathfrak{b}^\perp)(v_2,v_3,v_3,v_2)\\\nonumber
&\phantom{=}+2(h\KN\mathfrak{b}^\perp)(e_1,v_3,v_1,v_2)\geq 0.\qedhere
\end{align}
\end{proof}
\begin{example}
Let $h$ be a symmetric bilinear form on $\mathbb{R}^{n-1}$.
Consider an $\mathrm{Ad}_{\mathrm{SO}(n,\mathbb{C})}$-invariant subset $S\subset\mathfrak{so}(n,\mathbb{C})$.
It was shown by Gururaja-Maity-Seshadri~\cite[Thm.~1.4]{GMS} that $\mathsf{C}(S)\subset\mathsf{C}_{\PIC}$.
Therefore, we can draw the following conclusion:
\begin{itemize}
\item[$\triangleright$]{If $h$ is positive semidefinite, then $h\KN h\in\overline{\mathsf{C}(S)}$ and $h\KN\mathfrak{b}^\perp\in\overline{\mathsf{C}(S)}$.}
\item[$\triangleright$]{If $h\KN\mathfrak{b}^\perp\in\overline{\mathsf{C}(S)}$, then $h$ is $2$-non-negative.}
\end{itemize}
\end{example}


\section{Deformations in Gromov-Lawson flexible families}
In this section, we fix a Gromov-Lawson flexible family $\mathsf{C}$ of curvature conditions and a real number $0<\tau\leq 1$ such that $\mathfrak{b}\KN\mathfrak{b}^\perp+B_\tau\subset\mathsf{C}(0)$.
Furthermore, let $M$ be a smooth manifold of dimension $n\geq 2$ with non-empty boundary and let $\sigma\colon M\to\mathbb{R}$ be a continuous function.
Both of them remain unchanged throughout this section.

We will reproduce the proof of Theorem~\ref{theorem1}, beginning with Lemma~\ref{uniform epsilon 2}.
\begin{lemma}\label{bflemma}
Let $K$ be a compact Hausdorff space.
Let $g_0\colon K\to C^{\infty}(\partial M;T^\ast\partial M\otimes T^\ast\partial M)$ be a continuous family of Riemannian metrics on $\partial M$ and let $h\colon K\to C^{\infty}(\partial M;T^\ast\partial M\otimes T^\ast\partial M)$ be a continuous familiy of symmetric $(0,2)$-tensor fields.
Let $U\subset\partial M$ be an open subset and let $A\subset U$ be compact.

Then there exists a smooth function $\Lambda_0=\Lambda_0(g_0,h)\in C^\infty(U;\mathbb{R}),\Lambda_0>0$, a constant $0<\rho\leq 1$ and a real number $\delta_0>0$ such that
\begin{enumerate}
\item[$\triangleright$]{for every smooth function $\Lambda\in C^\infty(U;\mathbb{R})$ with $\Lambda\geq\Lambda_0$ and}
\item[$\triangleright$]{for every continuous family $g\colon K\to\mathscr{R}_{\mathsf{C}(\sigma)}(M)$ with
\begin{equation}\nonumber
g(\xi)_t=(1-\Lambda t^2)\cdot g_0(\xi)-2th(\xi)\quad\text{on $[0,\sqrt{\delta}]\times U$}
\end{equation}
for some $0<\delta<\min\Bigl\{\delta_0,\Vert\Lambda\Vert_{C^2(A)}^{-4}\Bigr\}$ and all $\xi\in K$,}
\end{enumerate}
it holds
\begin{equation}\nonumber
(\id_t\oplus j)^\ast R_{g(\xi)}(t,p)+B_{\rho\cdot\Lambda(p)}\subset\mathsf{C}(\sigma(t,p))
\end{equation}
for all $\xi\in K,(t,p)\in[0,\sqrt{\delta}]\times A$ and all $\Lambda(p)^{-1}$-nearly $g(\xi)_t(p)$-isometric isomorphisms $j\colon\mathbb{R}^{n-1}\to T_p(\partial M)$.
\end{lemma}

\begin{proof}
The proof is similar to that of Lemma~\ref{uniform epsilon 2}.
Instead of equation~\eqref{PFE}, we use the fact that
\begin{equation}\nonumber
(\Lambda-\Lambda_\varphi)(p)\cdot\mathfrak{b}\KN\mathfrak{b}^\perp+B_{\tau\cdot(\Lambda-\Lambda_\varphi)(p)}\subset\mathsf{C}(0).
\end{equation}
There is no need to consider $\kappa_\varphi$ in GL-flexible families.
The assertion holds with $\rho:=\tfrac{\tau}{2}$.
\end{proof}

We fix three covers $\left(U_\alpha^1\right)_{\alpha\in\mathbb{N}}\subset\left(U_\alpha^2\right)_{\alpha\in\mathbb{N}}\subset\left(U_\alpha^3\right)_{\alpha\in\mathbb{N}}$ of $\partial M$ as in Lemma~\ref{cover2} together with a partition of unity $\psi=\left(\psi_\alpha\right)_{\alpha\in\mathbb{N}}$ subordinate to $\left(U_\alpha^1\right)_{\alpha\in\mathbb{N}}$.

\begin{proposition}\label{1jet2}
Let $K$ be a compact Hausdorff space.
Let $g_0\colon K\to C^{\infty}(\partial M;T^\ast\partial M\otimes T^\ast\partial M)$ be a continuous family of Riemannian metrics on $\partial M$ and let $h,k\colon C^{\infty}(\partial M;T^\ast\partial M\otimes T^\ast\partial M)$ be continuous families of symmetric $(0,2)$-tensor fields with
\begin{align}\nonumber
\bigl(\iota_{g_0}^\ast(h-k)\bigr)\KN\mathfrak{b}^\perp\in\overline{\mathsf{C}(0)},
\end{align}
where $\iota_{g_0(\xi)(p)}\colon\mathbb{R}^{n-1}\to(T_p(\partial M),g_0(\xi)(p))$ is a linear isometry for $\xi\in K$ and $p\in\partial M$.

Then there exists a smooth function $\Lambda_0=\Lambda_0(g_0,h,k)\in C^\infty(\partial M;\mathbb{R}),\Lambda_0>0$ such that
\begin{enumerate}
\item[$\triangleright$]{for every continuous family
\begin{align}\nonumber
g\colon K\to\mathscr{R}_{\mathsf{C}(\sigma)}(M)
\end{align}
of $\Lambda$-normal metrics with $\Lambda\in C^\infty(\partial M;\mathbb{R}),\Lambda\geq\Lambda_0,g(\xi)_0=g_0(\xi)$ and $\mathrm{II}_{g(\xi)}=h(\xi)$ for all $\xi\in K$ and}
\item[$\triangleright$]{for each neighbourhood $\mathscr{U}\subset M$ of $\partial M$}
\end{enumerate}
there exists a continuous map
\begin{equation}\nonumber
f\colon K\times[0,1]\to\mathscr{R}_{\mathsf{C}(\sigma)}(M)
\end{equation}
such that the following holds for all $\xi\in K$ and $s\in[0,1]$:
\begin{enumerate}
\item[(a)]{$f(\xi,0)=g(\xi)$;}
\item[(b)]{$f(\xi,s)$ is $\Lambda$-normal;}
\item[(c)]{$f(\xi,s)_0=g_0(\xi)$;}
\item[(d)]{$\mathrm{II}_{f(\xi,s)}=(1-s)\mathrm{II}_{g(\xi)}+sk(\xi)$;}
\item[(e)]{$f(\xi,s)=g(\xi)$ on $M\setminus\mathscr{U}$.}
\end{enumerate}
\end{proposition}

\begin{proof}
Let $\mathscr{U}\subset M$ be a neighbourhood of $\partial M$.
Let $\Lambda\colon\partial M\to\mathbb{R}$ be a positive smooth function and let $g\colon K\to\mathscr{R}_{\mathsf{C}(\sigma)}(M)$ be a continuous family of $\Lambda$-normal metrics with $g(\xi)_0=g_0(\xi)$ and $\mathrm{II}_{g(\xi)}=h(\xi)$.

Given that this proposition follows Corollary~\ref{uniformgcn} and Proposition~\ref{2jet} within the deformation process, we can assume that all metrics in $g$ have the same normal exponential map and that they satisfy the condition of $\Lambda$-normality on a common neighbourhood $U_0\subset M$ of $\partial M$.

We choose a continuous positive function $\eta\colon\partial M\to\mathbb{R}$ such that $\varrho\colon V_\eta\to U^{g}$ is a common geodesic collar neighbourhood for $g$ with $U^{g}\subset U_0\cap\mathscr{U}$.
For each $\alpha\in\mathbb{N}$, let $0<\eta_\alpha<\inf_{p\in U_\alpha^3}\eta(p)$ be some fixed distance from the boundary.

For every $\alpha\in\mathbb{N}$, we apply Lemma~\ref{bflemma} to $U=U_\alpha^3$ and $A=\overline{U_\alpha^2}$.
The lemma provides a smooth function $\Lambda_{0,\alpha}\in C^{\infty}(U_\alpha^3;\mathbb{R}),\Lambda_{0,\alpha}>0$, a constant $0<\rho_\alpha\leq 1$ and a real number $0<\delta_{0,\alpha}<\eta_\alpha^2$ such that
\begin{align}\nonumber
(\id_t\oplus j)^\ast R_{g(\xi)}(t,p)+B_{\rho_\alpha\cdot\Lambda(p)}\subset\mathsf{C}\bigl(\sigma(t,p)\bigr)
\end{align}
for all $\alpha\in\mathbb{N}$ and $\xi\in K$, as well as for
\begin{enumerate}
\item[$\triangleright$]{all $(t,p)\in[0,\sqrt{\delta}]\times\overline{U_\alpha^2}$, where $0<\delta<\min\{\delta_{0,\alpha},\Vert\Lambda\Vert_{C^2(U_\alpha^3)}^{-4}\}$, and}
\item[$\triangleright$]{all $\Lambda(p)^{-1}$-nearly $g(\xi)_t(p)$-isometric isomorphisms $j\colon\mathbb{R}^{n-1}\to T_p(\partial M)$,}
\end{enumerate}
provided that $\Lambda|_{U_\alpha^3}\geq\Lambda_{0,\alpha}$.

The deformation will be performed over each piece $U_\alpha^2\subset\partial M$, gluing everything together with the partition of unity $\psi$.
To this end, let $\delta=\left(\delta_\alpha\right)_{\alpha\in\mathbb{N}}$ be a family of real numbers satisfying
\begin{equation}\nonumber
0<\delta_\alpha<\min\left\{\frac{1}{2},\alpha^{-2},\min_{\beta\in I_\alpha}\delta_{0,\beta},\min_{\beta\in I_\alpha}\bigl(\Lambda_\beta^{-4}\bigr)\right\}\quad\text{with}\quad \Lambda_\alpha:=\Vert\Lambda\Vert_{C^2(U_\alpha^3)}
\end{equation}
for all $\alpha\in\mathbb{N}$.
For each $\alpha\in\mathbb{N}$, let $\chi_{\delta_\alpha}$ be a function as in Lemma~\ref{bumpfunction}.
Put $\square_{\eta_\alpha}^2:=\varrho([0,\eta_\alpha)\times U_\alpha^2)$.
We define
\begin{align}\nonumber
f^\delta\colon&K\times[0,1]\to C^{\infty}(M;T^\ast M\otimes T^\ast M),\\\nonumber
&f^\delta(\xi,s)=\left\{\begin{array}{l}\mathrm{d}t^2+(1-\Lambda t^2)\cdot g_0(\xi)-2t\cdot h(\xi)+\sum\limits_{\alpha\in\mathbb{N}}\psi_\alpha\cdot 2s\chi_{\delta_\alpha}(t)\cdot(h(\xi)-k(\xi))\\
\phantom{g_1(\xi)}\hspace{0.6cm}\text{on $\bigcup\limits_{\alpha\in\mathbb{N}}\square_{\eta_\alpha}^2$,}\\
         g_1(\xi)\hspace{0.6cm}\text{on $M_1-\bigcup\limits_{i\in\mathbb{N}}\square_{\eta_\alpha}^2$}\end{array}\right.
\end{align}
for $\delta=\left(\delta_\alpha\right)_{\alpha\in\mathbb{N}}$ as small as above.
As in Proposition~\ref{1jet}, one can verify that $f^\delta$ is a continuous family of smooth Riemannian metrics if $\delta=(\delta_\alpha)_{\alpha\in\mathbb{N}}$ is sufficiently small.

Let $\alpha\in\mathbb{N}$ and $d>1$.
On $U_\alpha^3$ we use the abbreviation
\begin{align}\nonumber
&\gamma:=f^\delta(\xi,s)\quad\text{and}\\\nonumber
&\gamma_t:=f^\delta(\xi,s)_t=(1-\Lambda t^2)\cdot g_0(\xi)-2t\cdot h(\xi)+\sum\limits_{\beta\in I_\alpha}\psi_\beta\cdot 2s\chi_{\delta_\beta}(t)\cdot(h(\xi)-k(\xi))
\end{align}
for $t\in[0,\eta_\alpha)$.
It holds
\begin{align}\nonumber
&\mathrm{II}_t^\gamma=-\frac{1}{2}\dot{\gamma}_t=h(\xi)+\Lambda t\cdot g_0(\xi)-\sum_{\beta\in I_\alpha}\psi_\beta\cdot s\dot{\chi}_\delta(t)\cdot(h(\xi)-k(\xi)),\\\nonumber
&\ddot{\gamma}_t=-2\Lambda\cdot g_0(\xi)+\sum_{\beta\in I_\alpha}\psi_\beta\cdot 2s\ddot{\chi}_\delta(t)\cdot(h(\xi)-k(\xi)).
\end{align}
We investigate $\gamma$ on the set $[0,\eta_\alpha)\times U_\alpha^2$.
One can even restrict to $[0,\max_{\beta\in I_\alpha}\sqrt{\delta_\beta})\times U_\alpha^2$, since $\gamma$ is equal to $g(\xi)$ on $[\max_{\beta\in I_\alpha}\sqrt{\delta_\beta},\eta_\alpha)\times U_\alpha$.

For sufficiently small $\delta_{I_\alpha}$, it holds $|X|_{\gamma_0}\leq 2d$ for all $\xi\in K,s\in[0,1],(t,p)\in[0,\max_{\beta\in I_\alpha}\sqrt{\delta_\beta})\times U_\alpha^2$ and all $X\in T_p(\partial M)$ with $|X|_{\gamma_t}\leq d$.
This is similar to Proposition~\ref{1jet}.

In the following, we analyse the difference curvature tensor
\begin{equation}\nonumber
T:=R_\gamma-R_{g(\xi)}.
\end{equation}
Set $N_\alpha:=|I_\alpha|$.
Then relabel and sort the numbers $\delta_\beta,\beta\in I_\alpha$ by size, that is
\begin{equation}\nonumber
\delta^{(1)}\leq\delta^{(2)}\leq\dots\leq\delta^{(N_\alpha)}.
\end{equation}
Furthermore, $\delta^{(0)}:=0$ and $\delta^{(N_\alpha+1)}:=\max_{\beta\in I_\alpha}\sqrt{\delta_\beta}$.
Let $1\leq m\leq N_\alpha+1$.
If $\delta^{(m)}=\delta_\beta$ for some $\beta\in I_\alpha$, it is convenient to write $\psi^{(m)}=\psi_\beta$.

We partition $T$ into two curvature tensors $T^{\lesssim 1}$ and $T^{\ddot{\chi}_\delta}$ on $[\delta^{(m-1)},\delta^{(m)}]\times U_\alpha^2$:
\begin{alignat}{2}\nonumber
&\mathrm{(I)}&&\text{$T^{\lesssim 1}$ is defined via}\\\nonumber
&&&\hspace{1.5cm}T^{\lesssim 1}(X,Y,Z,W)=T(X,Y,Z,W),\\\nonumber
&&&\hspace{1.5cm}T^{\lesssim 1}(X,Y,Z,\nu)=T(X,Y,Z,\nu),\\\nonumber
&&&\hspace{1.5cm}T^{\lesssim 1}(X,\nu,\nu,Y)=T(X,\nu,\nu,Y)+\sum\limits_{\beta=m}^{N_\alpha}\psi^{(\beta)}\cdot s\ddot{\chi}_{\delta^{(\beta)}}(t)\cdot(h(\xi)-k(\xi))(X,Y).\\\nonumber
&\mathrm{(II)}\;\;&&\text{$T^{\ddot{\chi}_\delta}$ is defined via}\\\nonumber
&&&\hspace{1.5cm}T^{\ddot{\chi}_\delta}(X,Y,Z,W)=0,\\\nonumber
&&&\hspace{1.5cm}T^{\ddot{\chi}_\delta}(X,Y,Z,\nu)=0,\\\nonumber
&&&\hspace{1.5cm}T^{\ddot{\chi}_\delta}(X,\nu,\nu,Y)=-\sum\limits_{\beta=m}^{N_\alpha}\psi^{(\beta)}\cdot s\ddot{\chi}_{\delta^{(\beta)}}(t)\cdot(h(\xi)-k(\xi))(X,Y).
\end{alignat}
Let $(t,p)\in[\delta^{(m-1)},\delta^{(m)}]\times U_\alpha^2$.
Analogously to Proposition~\ref{1jet}, it holds $\Vert\gamma_t\Vert_{C^2(U_\alpha^3)}\lesssim 1$ and $\Vert\mathrm{II}_t^\gamma\Vert_{C^0(U_\alpha^3)}\lesssim 1$, as well as $\Vert g(\xi)_t\Vert_{C^2(U_\alpha^3)}\lesssim 1$ and $\Vert\mathrm{II}_t^{g(\xi)}\Vert_{C^0(U_\alpha^3)}\lesssim 1$.
This yields $|T^{\lesssim 1}(X,Y,Z,W)|\lesssim 1$ for all $X,Y,Z,W\in T(\partial M)$ with $\gamma_0$-norm $\leq 2d$.

The computations for $T^{\lesssim 1}(X,Y,Z,\nu)$ and $T^{\lesssim 1}(X,\nu,\nu,Y)$ are similar to those in Proposition~\ref{1jet}.
We conclude that $\Vert(\id_t\oplus\iota_{\gamma_t(p)})^\ast T^{\lesssim 1}\Vert\lesssim 1$ for all linear isometries $\iota_{\gamma_t(p)}\colon\mathbb{R}^{n-1}\to(T_p(\partial M),\gamma_t(p))$.

One can show that for each linear isometry $\iota_{\gamma_t(p)}\colon\mathbb{R}^{n-1}\to(T_p(\partial M),\gamma_t(p))$, there exists a linear isometry $\iota_{\gamma_0(p)}\colon\mathbb{R}^{n-1}\to(T_p(\partial M),\gamma_0(p))$ such that $\Vert(\id_t\oplus\iota_{\gamma_t(p)})^\ast T^{\ddot{\chi}_\delta}-(\id_t\oplus\iota_{\gamma_0(p)})^\ast T^{\ddot{\chi}_\delta}\Vert\lesssim 1$, provided that $\delta_{I_\alpha}$ is sufficiently small.
It also holds $(\id_t\oplus\iota_{\gamma_0(p)})^\ast T^{\ddot{\chi}_\delta}\in\overline{\mathsf{C}(0)}$ due to the boundary condition $\iota_{\gamma_0}^\ast(h-k)\KN\mathfrak{b}^\perp\in\overline{\mathsf{C}(0)}$.
This is again similar to Proposition~\ref{1jet}.

For sufficiently small $\delta_{I_\alpha}$ each linear isometry $\iota_{\gamma_t(p)}\colon\mathbb{R}^{n-1}\to(T_p(\partial M),\gamma_t(p))$ is $\Lambda(p)^{-1}$-nearly $g(\xi)_t(p)$-isometric.
Hence, when choosing $\Lambda_{0,\alpha}$ bigger if necessary, it holds
\begin{equation}\nonumber
(\id_t\oplus\iota_{\gamma_t(p)})^\ast T^{\lesssim 1}(t,p)+(\id_t\oplus\iota_{\gamma_t(p)})^\ast T^{\ddot{\chi}_\delta}(t,p)-(\id_t\oplus\iota_{\gamma_0(p)})^\ast T^{\ddot{\chi}_\delta}(t,p)\in B_{\rho_\alpha\cdot\Lambda(p)}
\end{equation}
for $\iota_{\gamma_{0}(p)}$ as above.
This leads us to conclude that
\begin{align}\nonumber
&(\id_t\oplus\iota_{\gamma_t(p)})^\ast R_\gamma(t,p)\\\nonumber
=\;&(\id_t\oplus\iota_{\gamma_t(p)})^\ast R_{g(\xi)}(t,p)+(\id_t\oplus\iota_{\gamma_t(p)})^\ast T^{\lesssim 1}(t,p)\\\nonumber
&\quad+(\id_t\oplus\iota_{\gamma_t(p)})^\ast T^{\ddot{\chi}_\delta}(t,p)-(\id_t\oplus\iota_{\gamma_0(p)})^\ast T^{\ddot{\chi}_\delta}(t,p)+(\id_t\oplus\iota_{\gamma_0(p)})^\ast T^{\ddot{\chi}_\delta}(t,p)\\\nonumber
&\quad\phantom{+(\id_t}\in\mathsf{C}(\sigma(t,p))+\overline{\mathsf{C}(0)}\subset\mathsf{C}(\sigma(t,p)).
\end{align}
Since $1\leq m\leq N_\alpha+1$ and $(t,p)\in[\delta^{(m-1)},\delta^{(m)}]\times U_\alpha^2$ were chosen arbitrarily, $R_\gamma$ satisfies $\mathsf{C}(\sigma)$ on $\square_{\eta_\alpha}^2$ for sufficiently small $\delta_{I_\alpha}$.
Set
\begin{equation}\nonumber
\Lambda_0\colon\partial M\to\mathbb{R}\,,\,\Lambda_0(p)=\sum_{\alpha\in\mathbb{N}}\psi_\alpha\cdot\max_{\beta\in I_\alpha}\Lambda_{0,\beta}(p).
\end{equation}
This function does the job, the reasoning is as in the proof of Proposition~\ref{1jet}.
\end{proof}

\begin{theorem}\label{theorem2}
Let $K$ be a compact Hausdorff space and let
\begin{equation}\nonumber
g\colon K\to\mathscr{R}_{\mathsf{C}(\sigma)}(M)
\end{equation}
be continuous.
Let $k\colon K\to C^{\infty}(\partial M;T^\ast\partial M\otimes\partial M)$ be a continuous family of symmetric $(0,2)$-tensor fields satisfying
\begin{align}\nonumber
\iota_{g(\xi)_0}^\ast(\mathrm{II}_{g(\xi)}-k)\KN\mathfrak{b}^\perp\in\overline{\mathsf{C}(0)}
\end{align}
for all $\xi\in K$ and $p\in\partial M$, where $\iota_{g(\xi)_0(p)}\colon\mathbb{R}^{n-1}\to (T_p(\partial M),g(\xi)_0(p))$ is a linear isometry for $\xi\in K$ and $p\in\partial M$.

Then there exists a smooth positive function $\Lambda_0\in C^\infty(\partial M)$ such that for each $\Lambda\in C^\infty(\partial M)$ with $\Lambda\geq\Lambda_0$ and for each neighbourhood $\mathscr{U}\subset M$ of $\partial M$, there is a continuous map
\begin{equation}\nonumber
f\colon K\times[0,1]\to\mathscr{R}_{\mathsf{C}(\sigma)}(M)
\end{equation}
so that the following holds for all $\xi\in K$ and $s\in[0,1]$:
\begin{enumerate}
\item[(a)]{$f(\xi,0)=g(\xi)$;}
\item[(b)]{$f(\xi,1)$ is $\Lambda$-normal;}
\item[(c)]{$f(\xi,s)_0=g(\xi)_0$;}
\item[(d)]{$\mathrm{II}_{f(\xi,s)}=S_1(s)\mathrm{II}_{g_1(\xi)}+(1-S_1(s))k(\xi)$, in particular $\mathrm{II}_{f(\xi,1)}=k(\xi)$;}
\item[(e)]{if $g(\xi)$ is $\tilde{\Lambda}$-normal, then $f(\xi,s)$ is $\Lambda_s$-normal for $\Lambda_s=S_2(s)\tilde{\Lambda}+(1-S_2(s))\Lambda$;}
\item[(f)]{$\ddot{f}(\xi,s)_0=S_2(s)\ddot{g}(\xi)_0-2(1-S_2(s))\Lambda g(\xi)_0$;}
\item[(g)]{for $\ell\geq 3$ we have $f(\xi,s)_0^{(\ell)}=S_2(s)\cdot g(\xi)_0^{(\ell)}$;}
\item[(h)]{$f(\xi,s)=g(\xi)$ on $M\setminus\mathscr{U}$.}
\end{enumerate}
\end{theorem}

As with the Perelman type, we can derive comparison results for spaces of metrics with boundary conditions.
Set
\begin{align}\nonumber
\mathscr{R}_{\mathsf{C}(\sigma)}^{\mathrm{GL}}(M):=\{g\in\mathscr{R}_{\mathsf{C}(\sigma)}(M): \iota_{g_0}^\ast\mathrm{II}_g\KN\mathfrak{b}^\perp\in\overline{\mathsf{C}(0)}\}
\end{align}
and recall that
\begin{align}\nonumber
\mathscr{R}_{\mathsf{C}(\sigma)}^{\mathrm{GL}}(M_1\sqcup M_2):=\{g_1\sqcup g_2\in\mathscr{R}_{\mathsf{C}(\sigma)}^0(M_1\sqcup M_2): \iota_{g_0}^\ast(\mathrm{II}_{g_1}+\mathrm{II}_{g_2})\KN\mathfrak{b}^\perp\in\overline{\mathsf{C}(0)}\}
\end{align}
for two manifolds $M_1,M_2$ with boundary $\partial M_1=\partial M_2$.

\begin{corollary}
Each of the inclusions
\begin{equation}\nonumber
\mathscr{R}_{\mathsf{C}(\sigma)}^{\mathsf{D}}(M)\hookrightarrow\mathscr{R}_{\mathsf{C}(\sigma)}^{\mathrm{II}=0}(M)\hookrightarrow\mathscr{R}_{\mathsf{C}(\sigma)}^{\mathrm{II}\geq 0}(M)\hookrightarrow\mathscr{R}_{\mathsf{C}(\sigma)}^{\mathrm{GL}}(M)
\end{equation}
is a weak homotopy equivalence.
\end{corollary}
This corollary is a direct consequence of Theorem~\ref{theorem2}.
Note that more comparison results apply.
In fact, one can reproduce~\cite[Chap.~4.1]{BH2023}.

\begin{theorem}\label{whe2}
Each of the inclusions
\begin{equation}\nonumber
\mathscr{R}_{\mathsf{C}(\sigma)}(M_1\cup_{\partial M}M_2)\hookrightarrow\mathscr{R}_{\mathsf{C}(\sigma)}^{\mathrm{II}_1+\mathrm{II}_2\geq 0}(M_1\sqcup M_2)\hookrightarrow\mathscr{R}_{\mathsf{C}(\sigma)}^{\mathrm{GL}}(M_1\sqcup M_2)
\end{equation}
is a weak homotopy equivalence.
\end{theorem}
\begin{proof}
The proof is similar to that of Theorem~\ref{whe}.
In the end of the proof, we apply Theorem~\ref{theorem2} to $M_1\sqcup M_2$ with $k(\xi)=-\mathrm{II}_{g_2(\xi)}\sqcup\mathrm{II}_{g_2(\xi)}$, where $g=g_1\sqcup g_2\colon D^m\to\mathscr{R}^{\mathrm{GL}}_{\mathsf{C}(\sigma)}(M_1\sqcup M_2)$ is a family to be deformed.
\end{proof}

Finally, we will show that Theorem~\ref{whe2} implies the gluing results for $\scal_k>0,\scal>0$ and $\mathcal{C}_m>0$.

\begin{example}
Let $h$ be a symmetric bilinear form on $\mathbb{R}^{n-1}$.
\begin{itemize}
\item[\textit{(a)}]{Let $1\leq k\leq n$.
Then $h\KN\mathfrak{b}^\perp$ has $\scal_k\geq 0$ if
\begin{equation}\nonumber
\text{$h$ is\;\;}\left\{\begin{array}{ll} \text{\phantom{(}$k$-non-negative} & \text{for $1\leq k\leq n-2$,} \\[0.2cm]
         \text{$(k-1)$-non-negative} & \text{for $k=n-1,n$.}\end{array}\right.
\end{equation}
If $h\KN\mathfrak{b}^\perp$ has $\scal_k\geq 0$ then
\begin{equation}\nonumber
\text{$h$ is\;\;}\left\{\begin{array}{ll} \text{\phantom{(}$k$-non-negative} & \text{for $1\leq k\leq n-1$,} \\[0.2cm]
         \text{$(k-1)$-non-negative} & \text{for $k=n$.}\end{array}\right.
\end{equation}
}
\item[\textit{(b)}]{Let $1\leq m\leq n-1$.
Then $h\KN\mathfrak{b}^\perp$ has $\mathcal{C}_m\geq 0$ if and only if $h$ is $m$-non-negative.}
\end{itemize}
\end{example}
\begin{proof}
\textit{(a)} The proof is similar to that in~\cite[Prop.~2.10]{RW}.
We choose an orthonormal basis $(u_2,\dots,u_n)$ of $\mathbb{R}^{n-1}$ which is diagonalising for $h$, i.e. there exist $\lambda_2,\dots,\lambda_n\in\mathbb{R}$ such that
\begin{equation}\nonumber
h(u_i,u_l)=\lambda_i\delta_{il}
\end{equation}
for all $2\leq i,l\leq n$.
Without loss of generality, $\lambda_2\leq\dots\leq\lambda_n$.
Setting $u_1:=e_1$, the family $(u_1,u_2,\dots,u_n)$ is an orthonormal basis of $\mathbb{R}^n$ which is diagonalising for $\ric(h\KN\mathfrak{b}^\perp)$.
Similar to previous calculations, one shows that
\begin{equation}\nonumber
\ric(h\KN\mathfrak{b}^\perp)=
\left(\begin{array}{cccc}
\sum_{i=2}^n\lambda_i&0&\cdots&0\\
0&\lambda_2&&0\\
\vdots&&\ddots&\\
0&0&&\lambda_n
\end{array}\right).
\end{equation}
The sum of the $k$ smallest eigenvalues of $\ric(h\KN\mathfrak{b}^\perp)$ is either
\begin{equation}\nonumber
\sum_{i=2}^{k+1}\lambda_i\quad\text{or}\quad\sum_{i=2}^n\lambda_i+\sum_{i=2}^k\lambda_i,
\end{equation}
where $\lambda_{n+1}:=\sum_{i=2}^n\lambda_i$.

We first consider $k=n$.
The sum of all eigenvalues of $\ric(h\KN\mathfrak{b}^\perp)$ is $\sum_{i=2}^n2\lambda_i$.
This is non-negative if and only if $h$ is $(n-1)$-non-negative.

Next we consider $1\leq k\leq n-1$.
When $h$ is $k$-non-negative for $1\leq k\leq n$ and $(k-1)$-non-negative for $k=n-1$, the sum of the $k$ smallest eigenvalues of $\ric(h\KN\mathfrak{b}^\perp)$ is $\sum_{i=2}^{k+1}\lambda_i$.
If this was not the case, the inequality
\begin{equation}\nonumber
\sum_{i=2}^{k+1}\lambda_i>\sum_{i=2}^n\lambda_i+\sum_{i=2}^k\lambda_i
\end{equation}
would imply $\sum_{i=2,i\neq k+1}^{n}\lambda_i<0$, so $\sum_{i=2}^{n-1}\lambda_i<0$.
This contradicts $k$-non-negativity for $1\leq k\leq n-2$ and $(k-1)$-non-negativity for $k=n-1$.
We conclude that $\scal_k\geq 0$.

For the other direction, we distinguish $1\leq k\leq n-2$ and $k=n-1$.
If $k=n-1$ and $h\KN\mathfrak{b}^\perp$ has $\scal_k\geq 0$, the sum of the $n-1$ smallest eigenvalues can be either $\sum_{i=2}^{n}\lambda_i$ or $\sum_{i=2}^n\lambda_i+\sum_{i=2}^{n-1}\lambda_i$.

In the first case, we immediately see that $h$ is $(n-1)$-non-negative.
In the second case, it holds $\sum_{i=2}^{n-1}\lambda_i\leq 0$, so $\scal_k\geq 0$ yields $\sum_{i=2}^n\lambda_i\geq 0$.
Thus, $h$ is $(n-1)$-non-negative.

Now let $1\leq k\leq n-2$.
If $\ric(h\KN\mathfrak{b}^\perp)$ has $\scal_k\geq 0$, then the sum of the $k$ smallest eigenvalues is $\sum_{i=2}^{k+1}\lambda_i$.
If not, the inequality above would again imply $\sum_{i=2,i\neq k+1}^{n}\lambda_i<0$.
The condition $\scal_k\geq 0$ yields $\lambda_i\geq 0$ for all $i\geq k+1$.
Since $k+1<n$, we would obtain $\sum_{i=2}^{k+1}\lambda_i<0$ which is a contradiction to $\scal_k\geq 0$.
Hence, as the sum of the $k$ smallest eigenvalues is $\sum_{i=2}^{k+1}\lambda_i$, we conclude that $h$ is $k$-non-negative.

\textit{(b)} See~\cite[Prop.~3.1]{CJW}.
\end{proof}



\begin{bibdiv}
\begin{biblist}

\bib{Almeida}{article}{
   author={Almeida, S.},
   title={Minimal hypersurfaces of a positive scalar curvature manifold},
   journal={Math. Z.},
   number={190},
   date={1985},
   pages={73--82},
}

\bib{BGM}{article}{
   author={B\"{a}r, C.},
   author={Gauduchon, P.},
   author={Moroianu, A.}
   title={Generalized cylinders in semi-Riemannian and Spin geometry},
   journal={Math. Z.},
   number={249},
   date={2005},
   pages={545--580},
}

\bib{BH2022}{article}{
   author={B\"{a}r, C.},
   author={Hanke, B.},
   title={Local flexibility for open partial differential relations},
   journal={Comm. Pure Appl. Math.},
   number={75},
   date={2022},
   pages={1377--1415},
}

\bib{BH2023}{article}{
   author={B\"{a}r, C.},
   author={Hanke, B.},
   title={Boundary conditions for scalar curvature},
   conference={
      title={Perspectives in Scalar Curvature. Vol. 2},
   },
   book={
      publisher={World Sci. Publ., Hackensack, NJ},
   },
   isbn={978-981-124-999-0},
   isbn={978-981-124-935-8},
   isbn={978-981-124-936-5},
   date={2023},
   pages={325--377},
   review={\MR{4577919}},
}

\bib{Besse}{book}{
   author={Besse, A. L.},
   title={Einstein manifolds},
   series={Ergebnisse der Mathematik und ihrer Grenzgebiete (3)},
   volume={10},
   publisher={Springer-Verlag, Berlin},
   date={1987},
   pages={xii+510},
   isbn={978-3-540-74120-6},
}

\bib{BWW}{article}{
   author={Botvinnik, B.},
   author={Walsh, M. G.},
   author={Wraith, D. J.},
   title={Homotopy groups of the observer moduli space of Ricci positive metrics},
   journal={Geom. Topol.},
   number={23},
   date={2019},
   pages={3003--3040},
}

\bib{Brendle}{book}{
   author={Brendle, S.},
   title={Ricci Flow and the Sphere Theorem},
   series={Graduate Studies in Mathematics},
   volume={111},
   publisher={American Mathematical Society, Providence, RI},
   date={2010},
   pages={vii+176},
   isbn={978-0-8218-4938-5},
}

\bib{BHJ}{article}{
   author={Brendle, S.},
   author={Hirsch, S.},
   author={Johne, F.},
   title={A generalization of Geroch's conjecture},
   journal={Comm. Pure Appl. Math},
   number={77},
   date={2024},
   pages={441--456},
}


\bib{BMN}{article}{
   author={Brendle, S.},
   author={Marques, F. C.},
   author={Neves, A.},
   title={Deformations of the hemisphere that increase scalar curvature},
   journal={Invent. Math.},
   number={185},
   date={2011},
   pages={175--197},
}


\bib{BrH}{book}{
   author={Bridson, M. R.},
   author={Haefliger, A.},
   title={Metric Spaces of Non-Positive Curvature},
   series={Grundlehren der mathematischen Wissenschaften},
   volume={319},
   publisher={Springer Berlin, Heidelberg},
   date={1999},
   pages={xxi+643},
   isbn={978-3-540-64324-1},
}

\bib{BBI}{book}{
   author={Burago, D.},
   author={Burago, Y.},
   author={Ivanov, S.},
   title={A course in metric geometry},
   series={Graduate Studies in Mathematics},
   volume={33},
   publisher={American Mathematical Society, Providence, RI},
   date={2001},
   pages={xiv+415},
   isbn={0-8218-2129-6},
}

\bib{Chow}{article}{
   author={Chow, T.-K. A.},
   title={Positivity of Curvature on Manifolds with Boundary},
   journal={Int. Math. Res. Notices},
   number={2022},
   date={2022},
   pages={11401--11426},
}

\bib{CJW}{article}{
   author={Chow, T.-K. A.},
   author={Johne, F.},
   author={Wan, J.}
   title={Preserving Positive Intermediate Curvature},
   journal={J. Geom. Anal.},
   number={33, 366},
   date={2023},
   pages={},
}

\bib{CW}{arxiv}{ 
   author={Chow, T.-K. A.}, 
   author={Wang, Y.}
   title={Minimal Two-Spheres and Manifolds with Positive Isotropic Curvature}, 
   note={Preprint available on \url{https://arxiv.org/abs/2609.06910}}, 
   year={2026} 
}


\bib{Frerichs2022}{thesis}{ 
   author={Frerichs, H.}, 
   title={Skalarkrümmung auf Mannigfaltigkeiten mit nicht-kompaktem Rand}, 
   note={available at \url{https://nbn-resolving.de/urn:nbn:de:bvb:384-opus4-1081965}}, 
   year={2022} 
}

\bib{Frerichs2025}{arxiv}{ 
   author={Frerichs, H.}, 
   title={Scalar curvature deformations with non-compact boundaries}, 
   note={Preprint available on \url{https://arxiv.org/abs/2403.03941}}, 
   year={2025} 
}

\bib{GL}{article}{
   author={Gromov, M.},
   author={Lawson, H. B.},
   title={Spin and scalar curvature in the presence of a fundamental group. I},
   journal={Ann. of Math. (2)},
   number={111},
   date={1980},
   pages={209--230},
}

\bib{GMS}{article}{
   author={Gururaja, H. A.},
   author={Maity, S.},
   author={Seshadri, H.},
   title={On Wilking's criterion for the Ricci flow},
   journal={Math. Z.},
   number={274},
   date={2013},
   pages={471--481},
}

\bib{Hirsch}{book}{
   author={Hirsch, M. W.},
   title={Differential topology},
   series={Graduate Texts in Mathematics},
   volume={33},
   publisher={Springer New York, NY},
   date={1976},
   pages={x+222},
   isbn={978-0-387-90148-0},
}

\bib{Hoelzel}{article}{
   author={Hoelzel, S.},
   title={Surgery stable curvature conditions},
   journal={Math. Ann.},
   number={365},
   date={2016},
   pages={13--47},
}

\bib{Kosovskii}{article}{
   author={Kosovskii, N.N.},
   title={Gluing of Riemannian manifolds of curvature at least $\kappa$},
   journal={St. Petersbg. Math. J.},
   number={14},
   date={2003},
   pages={467--478},
   note={translation from \textit{Algebra Anal.}~\textbf{14} (2002), 140--157},
}

\bib{MM}{article}{
   author={Micallef, M. J.},
   author={Moore, J. D.},
   title={Minimal two-spheres and the topology of manifolds with positive curvature on totally isotropic two-planes},
   journal={Ann. of Math. (2)},
   number={127},
   date={1988},
   pages={199--227},
}

\bib{Perelman}{article}{
   author={Perelman, G.},
   title={Construction of Manifolds of Positive Ricci Curvature with Big Volume and Large Betti Numbers},
   journal={Comparison Geometry, MSRI Publications},
   number={30},
   date={1997},
   pages={157--163},
}

\bib{RW}{arxiv}{ 
   author={Reiser, P.}, 
   author={Wraith, D. J.},
   title={A generalization of the Perelman gluing theorem and applications}, 
   note={Preprint available on \url{https://arxiv.org/abs/2308.06996}}, 
   year={2024} 
}

\bib{Seeley}{article}{
   author={Seeley, R.T.},
   title={Extension of $C^{\infty}$ functions defined in a half space},
   journal={Proc. Amer. Math. Soc.},
   number={15},
   date={1964},
   pages={625--626},
}

\bib{Sha}{article}{
   author={Sha, J.-P.},
   title={$p$-convex Riemannian manifolds},
   journal={Invent. Math.},
   number={83},
   date={1986},
   pages={437--447},
}

\bib{Schlichting2012}{arxiv}{ 
   author={Schlichting, A.}, 
   title={Gluing Riemannian manifolds with curvature operators at least $\kappa$}, 
   note={Preprint available on \url{https://arxiv.org/abs/1210.2957}}, 
   year={2012} 
}


\bib{Schlichting2014}{thesis}{ 
   author={Schlichting, A.}, 
   title={Smoothing singularities of Riemannian metrics while preserving lower curvature bounds}, 
   note={Ph.D. thesis, available at \url{http://dx.doi.org/10.25673/4040}}, 
   year={2014} 
}

\bib{Wilking}{article}{
   author={Wilking, B.},
   title={A Lie algebraic approach to Ricci flow invariant curvature conditions and Harnack inequalities},
   journal={J. Reine Angew. Math.},
   number={679},
   date={2013},
   pages={223--247},
}


\bib{Wolfson}{article}{
   author={Wolfson, J.},
   title={Manifolds with $k$-positive Ricci curvature},
   conference={
      title={Variational problems in differential geometry. London Math. Soc. Lecture Note Ser., Vol. 394},
   },
   book={
      publisher={Cambridge Univ. Press, Cambridge},
   },
   date={2011},
   pages={182--201},
}












\end{biblist}
\end{bibdiv}
\end{document}